\documentclass[graybox]{svmult}
\usepackage{type1cm}        
\usepackage{makeidx}         
\usepackage{graphicx}        
\usepackage{multicol}        
\usepackage[bottom]{footmisc}
\usepackage{newtxtext}       %
\usepackage{cite} 

\PassOptionsToPackage{hyphens}{url}\usepackage{hyperref} 

\usepackage{amsmath}
\usepackage{amsfonts}
\usepackage{amssymb}
\usepackage{amsthm}
\usepackage[utf8]{inputenc}
\usepackage{bbm}
\usepackage{graphicx}
\usepackage{tkz-berge}
\usepackage{tikz}
\usetikzlibrary{
	graphs,
	graphs.standard
}
\usetikzlibrary{fit,shapes}
\usepackage{hyperref}
\usepackage{url}
\usepackage{subcaption}
\usepackage{mathdots}
\usepackage[sc]{mathpazo}
\usepackage[T1]{fontenc} 
\usepackage{mathtools}

\renewcommand{\theenumii}{\roman{enumii}} 

\usepackage{enumitem}

\mathtoolsset{showonlyrefs=true}

\usepackage[font=small,labelfont=bf]{caption}

\numberwithin{theorem}{section}
\numberwithin{example}{section}
\numberwithin{lemma}{section}
\numberwithin{remark}{section}
\numberwithin{definition}{section}
\numberwithin{corollary}{section}
\numberwithin{proposition}{section}

\allowdisplaybreaks

\newcommand{\C}{\mathbb{C}}
\newcommand{\F}{\mathcal{F}}
\newcommand{\R}{\mathbb{R}}
\renewcommand{\E}{\vec{E}}
\newcommand{\M}{\mathrm{M}}
\newcommand{\N}{\mathbb{N}}

\renewcommand{\P}{\vec{P}}
\newcommand{\V}{\mathcal{V}}

\newcommand{\Kum}{{}_{1}\hspace{-0.7pt}F_{1}}

\newcommand{\Z}{\mathbb{Z}}

\renewcommand{\Re}{\mathrm{Re}}

\newcommand{\onorm}[1]{\| #1 \|_{\mathrm{op}}}

\renewcommand{\H}{\mathrm{H}}
\newcommand{\LL}{\boldsymbol{\lambda}}

\newcommand{\TT}{\mathrm{T}}
\newcommand{\X}{\vec{X}}
\newcommand{\tr}{\operatorname{tr}}
\newcommand{\diag}{\operatorname{diag}}

\newcommand{\bvec}[1]{\boldsymbol{#1}}
\newcommand{\norm}[1]{\| #1 \|}
\newcommand{\cnorm}[1]{|\!|\!| #1 |\!|\!|}

\renewcommand{\vec}[1]{\boldsymbol{#1}}

\newcommand{\inner}[1]{\langle #1 \rangle}

\begin{document}
\hypersetup{bookmarksdepth=-1}

\title*{Hunter's positivity theorem and random vector norms}
\titlerunning{Random-vector norms} 

\author{Ludovick Bouthat\orcidID{0000-0002-2094-3028},\\
\'Angel Ch\'avez, and\\ Stephan Ramon Garcia\orcidID{0000-0001-8971-5448}}

\institute{Ludovick Bouthat \at D\'epartement de math\'ematiques et de statistique, Universit\'e Laval, 2325 Rue de l'Universit\'e, Qu\'ebec, QC G1V 0A6 CA, \email{lubou309@ulaval.ca} \and \'Angel Ch\'avez \at Mathematics Department, Regis University, 3333 Regis Blvd., Denver, CO 80221 USA, \email{chave360@regis.edu}
\and Stephan Ramon Garcia \at Department of Mathematics and Statistics, Pomona College, 610 N. College Ave. Claremont, CA 91711 USA \email{stephan.garcia@pomona.edu}\\ \url{https://stephangarcia.sites.pomona.edu/}}

\maketitle

\abstract{A theorem of Hunter ensures that the complete homogeneous symmetric polynomials of even degree are positive definite functions. A probabilistic interpretation of Hunter's theorem suggests a broad generalization: the construction of so-called random vector norms on square complex matrices.  This paper surveys these ideas, starting from the fundamental notions and developing the theory to its present state.  We study numerous examples and present a host of open problems.
}

\hypersetup{bookmarksdepth}

\section{Introduction}
The complete homogeneous symmetric (CHS) polynomial of degree $k$ in $n$ variables is
the sum of all distinct degree-$k$ monomials in the given variables.  A remarkable theorem of D.B.~Hunter, proved in 1977, asserts that the CHS polynomials of even degree are positive definite functions; that is, they vanish only at the origin and are otherwise positive \cite{Hunter}.  Since then, this result has been rediscovered and reproved many times; see 
\cite[Thm.~1]{Aguilar},
\cite[Lem.~3.1]{Barvinok},
\cite{Baston},
\cite[Thm.~2]{BGON},
\cite[Cor.~17]{GOOY},
\cite[Thm.~2.3]{Roventa},
and \cite[Thm.~1]{Tao}.  As Tao observed, these polynomials are also Schur convex, meaning that they respect the majorization of real vectors \cite[Thm.~1.ii]{Tao}. 

When CHS polynomials of even degree are applied to the eigenvalue lists of Hermitian matrices, one obtains a norm on the real vector space of complex Hermitian matrices.  Unlike traditional norms, these norms rely upon the eigenvalues themselves, rather than their absolute values.  More importantly, these norms extend naturally to all square complex matrices.  This process involves noncommutative $*$-polynomials and reveals intriguing combinatorial and probabilistic interpretations.

We present four proofs of Hunter's theorem, including a probabilistic approach that suggests a fruitful path toward generalizations.  It turns out that the CHS polynomials arise as expectations of certain random vectors whose entries are drawn from an exponential distribution.\footnote{They can also be viewed as moments of a certain probability density, a B-spline in the terminology of Curry--Schoenberg \cite{CS,BGON} or as Peano's kernel from the theory of divided differences \cite{ReciprocalSchur, Davis, Phil}.}  
Other distributions provide different families of random vectors, all of which provide new norms on the Hermitian matrices, and by a general complexification process, to all square complex matrices.

Surprisingly, the original motivation for these results stems from numerical semigroup theory (that is, in combinatorics), in which
CHS polynomials appeared unexpectedly in the statistical study of factorization lengths \cite{GOOY, GOU, GOOW}.

Although the complete story is not yet fully written, many of the fundamental ideas relevant to operator theorists, matrix theorists, and functional analysts can now be presented in a deliberate and coherent manner.  Our aim in this survey article, which is based upon material in \cite{Aguilar, Bouthat1, BGON, Ours, CGH}, is to start with the basic notions and notations, many of which are probabilistic or combinatorial in nature, and develop the theory of random vector norms to the point where the energetic reader might consider tackling some of the open problems we pose at the conclusion.

\vspace{-8pt}
\subsection{Organization}
\vspace{-3pt}
This paper is organized as follows. Section \ref{Section:Preliminaries} establishes our notation and contains important background material. Section \ref{Section:Birkhoff} presents results about the polytope of doubly stochastic matrices and its connection with majorization. Section \ref{Section:Hunter} concerns Hunter's positivity theorem, which motivates, directly and indirectly, much of the rest of this paper. Section \ref{Section:Classification} concerns the classification of weakly unitarily invariant norms on the space of complex Hermitian matrices, along with a method to extend these norms to all square complex matrices.

Random vector norms are defined in Section \ref{Section:Main} and several of their properties are established in Theorem \ref{Theorem:Main2}, the main result of this document. The rest of the paper is dedicated to studying these norms. Section \ref{Section:Examples} contains examples of random vector norms arising from familiar distributions. Section \ref{Section:LowerandUpperBounds} establishes lower and upper bounds for these new norms, comparing them to more standard quantities. Section \ref{Section:Submultiplicativity} is dedicated to studying the conditions under which a given random vector norm is submultiplicative. Section \ref{Sec:CMS} comes full circle and investigates the norms arising from the CHS polynomials, the starting point for the study of random vector norms. 
We wrap up in Section \ref{Section:Questions} with a list of open problems that we hope will motivate readers and spur further research.

\section{Background material}\label{Section:Preliminaries}
This section contains the background material and notation needed for our explorations.  Subsection \ref{Subsection:Numbers} concerns numbers, matrices, and functions. Subsection \ref{Subsection:SchurConvexity} covers symmetric polynomials, the majorization partial order, and Schur-convexity. Subsection \ref{Subsection:Partitions} concerns partitions and the power-sum expansions for complete homogeneous symmetric polynomials. Subsection \ref{Subsection:Probability} contains a brief review of basic probability theory.  Subsection \ref{subsec:ineq} concerns several famous inequalities which are used in this document. Finally, Subsection \ref{subsec:func} describes the complete Bell polynomials and the Kummer confluent hypergeometric function.


\subsection{Basic notation}\label{Subsection:Numbers}

\noindent\textbf{Numbers.}
Let $\N$, $\Z$, $\R$, and $\C$ denote the set of natural numbers, integers, real numbers, and complex numbers,
respectively.
For emphasis, let us state that $\N$ does not include 0.  This reminder is prudent, since in the study of numerical semigroups (which originally inspired our investigations), the convention is that $\N$ contains $0$.  We denote numbers in lower-case
italic or Greek letters. In this document, the lower-case italic letter $d$ shall be used repeatedly. Depending on the context, it should be interpreted either as a real number greater than $1$, a natural number, or an even number greater than $2$. In these cases, the domain will be thoroughly noted. 

\medskip\noindent\textbf{Vectors.}
We denote the elements of $\C^n$, and hence of $\R^n$, by lower-case boldface letters (occasionally Greek). 
If matrix multiplication is involved, we regard them as column vectors, that is, as $n \times 1$ matrices.  
We endow $\C^n$ and $\R^n$ with the inner product $\inner{ \vec{x}, \vec{y} } = \sum_{i=1}^n x_i \overline{y_i}$, in which $\vec{x} = (x_1,x_2,\ldots,x_n)$
and $\vec{y} = (y_1,y_2,\ldots,y_n)$.
The inner product $\inner{\cdot, \cdot}$ is conjugate-linear in the second slot.

\medskip\noindent\textbf{Matrices.}
The superscript $^{\intercal}$ denotes the transpose and $^{*}$ denotes the conjugate transpose of a matrix.  
We identify the $1 \times 1$ matrix $\vec{y}^* \vec{x}$ with the scalar $\inner{ \vec{x} , \vec{y} }$.
A matrix is \emph{Hermitian} if it equals its conjugate transpose.
Let $\M_n$ denote the set of $n \times n$ complex matrices and $\H_n \subset \M_n$ the subset of $n\times n$ Hermitian complex matrices. 
We reserve the letter $A$ for Hermitian matrices (so $A=A^*$ implicity) and $Z$ for arbitrary square complex matrices. A matrix $U\in \M_n$ is \emph{unitary} if $UU^*=U^*U=I_n$. 
Let $\diag(x_1,x_2,\ldots,x_n) \in \M_n$ denote the $n \times n$ diagonal matrix
with diagonal entries $x_1,x_2,\ldots,x_n$, in that order.  If $\vec{x} = (x_1,x_2,\ldots,x_n)$ is understood from context,
we may write $\diag(\vec{x})$ for brevity.

\medskip\noindent\textbf{Eigenvalues and singular values.}
The eigenvalues of each $A\in \H_n$ are real.  We write them in nonincreasing order, 
repeated according to multiplicity:
\begin{equation*}
\lambda_1(A)\geq \lambda_2(A)\geq \cdots \geq \lambda_n(A).
\end{equation*}
We denote this collectively by $\LL(A) = (\lambda_1(A), \lambda_2(A), \ldots, \lambda_n(A))$, or simply 
$\LL=(\lambda_1, \lambda_2, \ldots, \lambda_n)$
if $A$ is unambiguous from context.
The \emph{singular values} of $Z \in \M_n$ are the eigenvalues
of the positive semidefinite matrix $(Z^*Z)^{1/2}$.  We denote them
$\sigma_1(Z) \geq \sigma_2(Z) \geq \cdots \geq \sigma_n(Z) \geq 0$,
repeated according to multiplicity.

\medskip\noindent\textbf{Norms.}
A \emph{norm} $\| \,\cdot\,\|$ on $\M_n$ or $\H_n$ is a positive-definite function that is absolutely homogeneous and satisfies
the triangle inequality.  In particular, a norm must satisfy
\begin{enumerate}[label=(\alph*)]
	\itemsep.5em
	\item $\norm{A} \geq 0$ with $\norm{A} = 0$ if and only if $A = 0$;
	\item $\norm{cA} = |c| \norm{A}$; and
	\item $\norm{A+B} \leq \norm{A} + \norm{B}$.
\end{enumerate}
A norm on $\M_n$ is a \emph{matrix norm} if it is also submultiplicative: $\norm{AB}\leq \norm{A} \norm{B}$.
Commonly encountered matrix norms include the \emph{operator norm}, which is defined by $\onorm{A} = \sigma_1(Z)$, and
the \emph{Frobenius norm}, which is defined by 
\begin{equation*}
\norm{A}_{\text{F}} = (\tr A^*\!A)^{1/2}= \bigg(\sum_{i,j=1}^n |a_{ij}|^2\bigg)^{\!\!1/2} \!= \bigg(\sum_{k=1}^n \sigma_k^2(Z)  \bigg)^{\!\!1/2}\!.
\end{equation*}

\subsection{Symmetry and Schur-convexity}\label{Subsection:SchurConvexity}
A multivariate function is \emph{symmetric} if it is invariant under all permutations of its arguments. 
For example, the \emph{trace} of $A \in \M_n$ is a symmetric function of its eigenvalues and also of its diagonal entries:
$\tr A = \sum_{i=1}^n a_{ii} = \sum_{i=1}^n \lambda_i(A)$.

\medskip\noindent\textbf{Power-sum symmetric polynomials.}
The \emph{power-sum symmetric polynomial} of degree $k$ in the variables $x_1,x_2,\ldots,x_n$, where $k\in \N$, is defined by
\begin{equation}\label{eq:PowerSumDefinition}
    p_k(x_1,x_2, \ldots, x_n)=x_1^k+x_2^k+\cdots +x_n^k,
\end{equation}
often written $p_k(\vec{x})$ or $p_k$ if the variables are clear from context. If $A\in \H_n$, then the eigenvalues of $A^k$ are the $k$th powers of the eigenvalues of $A$, repeated according to multiplicity. We therefore have the following important relationship for $A\in \H_n$:
\begin{equation}\label{eq:TracePower}
    p_k(\LL(A))  =  \tr (A^k) .
\end{equation}

\medskip\noindent\textbf{CHS polynomials.}
The \emph{complete homogeneous symmetric polynomial} (CHS) of degree $d$ in the variables $x_1, x_2, \ldots ,x_n$, where $d\geq 0$ is an integer, is the sum
\begin{equation}\label{eq:CHS}
h_d(x_1,x_2,\ldots,x_n) = \sum_{1 \leq i_1 \leq \cdots \leq i_{d} \leq n} x_{i_1} x_{i_2}\cdots x_{i_d},
\end{equation}
of all monomials of degree $d$ in $x_1,x_2,\ldots,x_n$ \cite[Sec.~7.5]{StanleyBook2}.
We may write $h_d(\vec{x})$ if the variables are clear from context.  
An elementary combinatorial argument of the stars-and-bars variety shows that there are $\binom{n+d-1}{d}$
summands in \eqref{eq:CHS}.  
The first several CHS polynomials in $x_1$ and $x_2$ are 
\begin{align*}
h_0(x_1,x_2)&= 1,\\
h_1(x_1,x_2)&= x_1+x_2,\\
h_2(x_1,x_2)&= x_1^2+x_1 x_2+x_2^2, \\
h_3(x_1,x_2)&= x_1^3+ x_1^2 x_2+x_1 x_2^2 +x_2^3, \\
h_4(x_1,x_2)&= x_1^4  + x_1^3 x_2 + x_1^2 x_2^2 + x_1 x_2^3 + x_2^4,\\
h_5(x_1,x_2) &= x_1^5+x_1^4 x_2 +x_1^3x_2^2 +x_1^2 x_2^3 +x_1 x_2^4 +x_2^5, \quad\text{and}\\
h_6(x_1,x_2) &= x_1^6+ x_1^5 x_2 + x_1^4 x_2^2 +x_1^3 x_2^3 +x_1^2x_2^4  +x_1 x_2^5 +x_2^6.
\end{align*}
The generating function for the CHS polynomials is 
\begin{equation}
 \sum _{{k=0}}^{\infty }h_{k}(x_1,x_2,\ldots ,x_n)t^{k}
 =\prod _{{i=1}}^{n}\sum _{{j=0}}^{\infty }(x_it)^{j}
 =\prod_{i=1}^{n} \frac{1}{1-x_i t}.\label{eq:CHSGenerating}
\end{equation}
Its radius of convergence is $\min\{ |x_1|^{-1},\ldots, |x_n|^{-1}\}$, in which we permit $+\infty$ as an argument
of the minimum whenever $x_k=0$ for some $k\in \{1,2,\ldots,n\}$.
 As is customary in combinatorics, the coefficient $c_k$ of $t^k$ in $f(t) = \sum_{r=0}^{\infty} c_r t^r$ is denoted by $[t^k]f(t)$.  Therefore, \eqref{eq:CHSGenerating} ensures that
\begin{equation}\label{eq:CHSGenerating2}
h_k(x_1,x_2,\ldots,x_n) = [t^k]\prod_{i=1}^{n} \frac{1}{1-x_i t}.
\end{equation}
For $d \geq 1$, the power-sum symmetric and CHS polynomials are related by the Newton--Gerard identities \cite[Section 10.12]{Seroul}
\begin{equation}
kh_{k}=\sum _{i=1}^{k}h_{k-i}p_{i} \quad \text{for }~ k \geq 1.\label{eq:CHSPowerSumRec}
\end{equation}
We remark that the power-sum expansion for the CHS polynomials appears below in \eqref{eq:hdgen} and gives another presentation of the CHS polynomials. Finally, we remark that Hunter's theorem (see Theorem \ref{Theorem:Hunter} below) is a recurring theme in this paper.  It says that the CHS polynomials of even degree are positive-definite functions on $\R^n$.

The CHS polynomials can be extended to any order $z\in \C$ if $\Re\,z>-1$, using the formula
\begin{equation}\label{abhz}
	h_z (x_1,x_2,\ldots,x_n)=\binom{z+n-1}{n-1} \int_{\R} x^z  F(x;x_1,x_2,\ldots,x_n)\,\mathrm{d}x,
\end{equation}
where $F(x;x_1,\ldots,x_n)$ is the Curry--Schoenberg B-spline, introduced in~\cite{CS} and defined by
\begin{equation}\label{eq:PeanoSecond}
	F(x;x_1,x_2,\ldots,x_n) = \frac{n-1}{2} \sum_{j=1}^n \frac{|x_j-x|(x_j-x)^{n-3}}{\prod_{k \neq j}(x_j-x_k)}.
\end{equation}
This generalization was established in \cite[Thm.~1]{BGON} and is used in Subsection \ref{Subsection:CHSExtension}.

\medskip\noindent\textbf{Majorization.}
Let $\widetilde{\bvec{x}}=(\widetilde{x}_1, \widetilde{x}_2, \ldots, \widetilde{x}_n)$ denote the nonincreasing rearrangement of $\bvec{x}=(x_1,x_2, \ldots, x_n) \in \R^n$. For example, $\widetilde{\vec{x}} = (4,3,2,1)$ when $\vec{x} = (1,3,2,4)$.
We say that $\bvec{y}$  \emph{majorizes} $\bvec{x}$, denoted $\bvec{x}\prec \bvec{y}$, if 
\vspace{-3pt}
\begin{equation*}
\sum_{i=1}^n \widetilde{x}_i = \sum_{i=1}^n \widetilde{y}_i
\quad \text{and} \quad
\sum_{i=1}^k \widetilde{x}_i \leq \sum_{i=1}^k \widetilde{y}_i
\quad \text{for}~ 1 \leq k \leq n-1.
\end{equation*} For example, $(3,3,3)\prec (5,3,1)$ and $ (2, 0 ,0, -2 ,0) \prec (2, 1, 0, -1, -2)$.

\medskip\noindent\textbf{Schur-convexity.}
A function $f:\R^n\to \R$ is \emph{Schur-convex} if  $f(\bvec{x})\leq f(\bvec{y})$ whenever $\bvec{x}\prec \bvec{y}$. For example, the power-sum symmetric polynomials \eqref{eq:PowerSumDefinition} and the CHS polynomials \eqref{eq:CHS} are Schur-convex. The Schur-convexity of the CHS polynomials is proved in Subsection \ref{Subsection:TaoCHS}. We remark that any convex symmetric function is Schur-convex \cite[p.~258]{Roberts}. The converse of this statement is false, however. For example, the function  $f(x,y)=(x+y)^3$ is Schur-convex but not convex.


\subsection{Partitions, polynomials, and traces}\label{Subsection:Partitions}
Threading the power-sum symmetric polynomials over partitions and relating
the result to traces of Hermitian matrices is particularly important. We therefore review the basic notions of partition combinatorics which are relevant to our work.

\medskip\noindent\textbf{Partitions.}
A \emph{partition} of $d\in \N$ is an $r$-tuple $\bvec{\pi}=(\pi_1, \pi_2, \ldots, \pi_r) \in \N^r$ such that
\begin{equation}\label{eq:PartitionDef}
\pi_1 \geq \pi_2 \geq \cdots \geq \pi_r  \quad \text{and}\quad
\pi_1+ \pi_2 + \cdots + \pi_r = d.
\end{equation}
The number of \emph{parts} $r$ depends on the partition $\bvec{\pi}$; clearly $r \leq d$.
We write $\bvec{\pi} \vdash d$ if $\bvec{\pi}$ is a partition of $d$ and 
we write $| \bvec{\pi}| = r$ for the number of parts in the partition \cite[Sec.~1.7]{StanleyBook1}.  
For example, $(20,9,6) \vdash 35$ and $|(20,9,6)| = 3$.

\medskip\noindent\textbf{Combinatorial coefficients.}
Suppose $\bvec{\pi}=(\pi_1, \pi_2, \ldots, \pi_r)\vdash d$. Define
\begin{equation}\label{eq:DefY}
\kappa_{\bvec{\pi}} = \kappa_{\pi_1} \kappa_{\pi_2} \cdots \kappa_{\pi_{r}}
\quad \text{and} \quad
y_{\bvec{\pi}}= \frac{d!}{\prod_{i\geq 1}(i!)^{m_i}m_i!},\vspace{-3pt}
\end{equation} 
in which the $\kappa_i$s are defined shortly in \eqref{eq:CumulantGen} and $m_i$ is the multiplicity of $i$ in $\bvec{\pi}$. For example, $d=13$ and
$\bvec{\pi} = (4,4,2,1,1,1)$ yields 
\begin{equation*}
\kappa_{\bvec{\pi}} = \kappa_4^2 \kappa_2 \kappa_1^3
\quad\text{and}\quad
y_{\bvec{\pi}}= \frac{13!}{(1!^3 3!) (2!^1 1!) (4!^2 2!)} = 450{,}450.
\end{equation*}
Note that in \cite{CGH,Ours}, $y_{\bvec{\pi}}$ was instead defined as $\prod_{i\geq 1}(i!)^{m_i}m_i!$. This diverging notation is explained in further detail in Remark \ref{Rem:redef}. Let us note for the moment that the coefficients $y_{\bvec{\pi}}$ are precisely the number of partitions of a set of size $d$ corresponding to the integer partition
\begin{equation*} d=\underbrace {1+\cdots +1} _{m_{1}}\,+\,\underbrace {2+\cdots +2} _{m_{2}}\,+\,\underbrace {3+\cdots +3} _{m_{3}}+\cdots \vspace{-3pt}\end{equation*}
of the integer $d$. As such, they are always integers.

\medskip\noindent\textbf{Partitions, power sums, and traces.}
Suppose $\bvec{\pi}=(\pi_1, \pi_2, \ldots, \pi_r)\vdash d$.  Let
\begin{equation*}
    p_{\bvec{\pi}} = p_{\pi_1}p_{\pi_1}\cdots p_{\pi_r},
\end{equation*}
in which the power-sum symmetric polynomials $p_{\pi_i}$ are defined in \eqref{eq:PowerSumDefinition}.
Recall that the trace of a matrix is the sum of its eigenvalues, repeated according to multiplicity.
If $A \in \H_n$ has eigenvalues $\bvec{\lambda} = (\lambda_1,\lambda_2,\ldots,\lambda_n)$, then
$p_k(\bvec{\lambda})= \tr A^k$ by \eqref{eq:TracePower} so
\begin{equation}\label{eq:pTrace}
p_{\bvec{\pi}}(\bvec{\lambda}) =p_{\pi_1}(\bvec{\lambda})p_{\pi_2}(\bvec{\lambda})\cdots p_{\pi_r}(\bvec{\lambda})
=(\tr A^{\pi_1})(\tr A^{\pi_2})\cdots (\tr A^{\pi_{r}}).
\end{equation} 
This identity connects eigenvalues, traces, and partitions to symmetric polynomials.

\medskip\noindent\textbf{Partitions and CHS polynomials.}
Another expression for the CHS polynomials is
\begin{equation}\label{eq:hdgen}
 h_{d}(x_1,x_2,\ldots, x_n)=\sum_{\bvec{\pi} \,\vdash\, d} \frac{p_{\bvec{\pi}}(x_1,x_2,\ldots, x_n)}{z_{\bvec{\pi}}} ,
\end{equation}   
in which the sum runs over all partitions $\bvec{\pi}=(\pi_1, \pi_2, \ldots,\pi_r)$ of $d$ and 
\begin{equation}\label{eq:zI}
z_{\bvec{\pi}} = \prod_{i \geq 1} i^{m_i} m_i!,
\end{equation} 
where $m_i$ is the multiplicity of $i$ in $\bvec{\pi}$ \cite[Prop.~7.7.6]{StanleyBook2}. For example, the partition $\bvec{\pi} = (4,4,2,1,1,1)$ yields $z_{\bvec{\pi}}= (1^3 3!) (2^1 1!) (4^2 2!) = 384$ \cite[(7.17)]{StanleyBook2}. The integer $z_{\bvec{\pi}}$ is  the size of the centralizer of a permutation of conjugacy class $\bvec{\pi}$.


\subsection{Probability theory}\label{Subsection:Probability}
A \emph{probability space} is a measure space $(\Omega, \F, \P)$, in which 
$\F$ is a $\sigma$-algebra of subsets of $\Omega$ and the function $\P:\F\to[0,1]$ satisfies $\P(\Omega)=1$.
In the present work, $\Omega$ is a $\sigma$-algebra of subsets of $\R$, or of some compact subset thereof.

\medskip\noindent\textbf{Random variables.}
A \emph{random variable} is a measurable function $X: \Omega\to \R$. 
We assume that $\Omega \subset \R$ and $X$ is nondegenerate, that is, nonconstant almost everywhere.
The pushforward measure $X_{*}\P$ of $X$ is the \emph{probability distribution} of $X$.  
The \emph{cumulative distribution function} (CDF) of $X$ is the function defined by 
\begin{equation*}
F_X(x)=\P(X\leq x),
\end{equation*}
which is the pushforward measure of $(-\infty, x]$. If $X_{*}\P$ is absolutely continuous with respect to Lebesgue measure $m$,  then the corresponding  Radon--Nikodym derivative
$f_X= dX_* P/dm$ is the \emph{probability density function} (PDF) of $X$ \cite[Ch.~1]{Billingsley}. We remark that the CDF of a random variable always exists, while the PDF may not. 

\medskip\noindent\textbf{Expectation.}
The \emph{expectation} of a random variable $X$, when it exists, is defined by
\begin{equation*}
\E[X]=\int_{\Omega} X \,d\P,
\end{equation*}
sometimes written as $\E X$ when a proliferation of brackets is undesirable.
For $p\geq 1$, let $L^p(\Omega, \F, \P)$ denote the (Banach) space of random variables satisfying
\begin{equation*}
\norm{X}_{L^p}=(\E |X|^p)^{1/p} < \infty.
\end{equation*}

\medskip\noindent\textbf{Moments.}
Suppose $k\in \N$. The $k$th \emph{moment} of a random variable $X$ is defined by
\begin{equation*}
\mu_k = \E[X^k]
\end{equation*}
if it exists.  The \emph{mean} of $X$ is $\mu_1$ and 
the \emph{variance} of $X$ is $\mu_2 - \mu_1^2$; we often write $\mu$ and $\sigma^2$, respectively. Jensen's inequality ensures that 
the variance is positive since $X$ is nondegenerate.
If $X$ has PDF $f_X$, then 
\begin{equation*}
\mu_k = \int_{-\infty}^{\infty} x^k f_X(x)\,dm(x).
\end{equation*}
Similarly, the $d$th \emph{standardized absolute moment} of the distribution $X$ is defined as 
\begin{equation*}
\Tilde{\mu}_d = \frac{\E\big[|X-\mu|^d \big]}{\sigma^d}.
\end{equation*}

\medskip\noindent\textbf{Independence.}
A collection $X_1, X_2, \ldots, X_n$ of random variables  is \emph{independent} if 
\begin{equation*}
\P\Big( \bigcap_{i=1}^{n}\{X_i\leq x_i \}\Big)\!=\prod_{i=1}^n F_{X_i}(x_i).
\end{equation*}
Independent random variables behave nicely with respect to expectation and this fact is important in our work. If $X_1, X_2, \ldots, X_n$ are independent random variables, then
\begin{equation*}
\E \big[X_1^{i_1} X_2^{i_2}\cdots X_n^{i_n}\big]=\prod_{k=1}^n\E  \big[X_k^{i_k}\big]
\end{equation*}
for all $i_1, i_2, \ldots, i_n \in \N$ whenever both sides exist. We remark that a collection $X_1, X_2, \ldots, X_n$ of random variables are called \emph{independent and identically distributed} (iid) if they are independent and have identical cumulative distributions.

\medskip\noindent\textbf{Moment generating function and cumulants.}
The \emph{moment generating function}  of  a random variable $X$, when it exists, is the power series
\begin{equation}\label{eq:MGF}
M(t)=\E \big[e^{tX}\big]=\sum_{k=0}^{\infty} \E \big[X^k\big] \frac{t^k}{k!} = \sum_{k=0}^{\infty} \mu_k\frac{t^k}{k!}.
\end{equation}
If $X$ admits a moment generating function $M(t)$, 
then the $r$th \emph{cumulant} $\kappa_r$ of $X$ is defined by the \emph{cumulant generating function} 
\begin{equation}\label{eq:CumulantGen}
K(t)=\log M(t)=\sum_{r=1}^{\infty} \kappa_r \frac{t^r}{r!}.
\end{equation} 
The first two cumulants of $X$ correspond to the mean and variance of $X$. In particular,  
\begin{equation*}
\kappa_1 = \mu_1 \quad \text{and}\quad  \kappa_2 = \mu_2 - \mu_1^2.
\end{equation*}
If $X$ does not admit a moment generating function but $X\in L^d(\Omega, \F, \P)$ for some  
$d\in \N$, then we can define $\kappa_1, \kappa_2, \ldots, \kappa_d$ by the following recursion \cite[Sec.~9]{Billingsley}:
\begin{equation}
\mu_r=\sum_{\ell=0}^{r-1} \binom{r-1}{\ell} \mu_{\ell}\kappa_{r-\ell} \quad \text{for $1 \leq r \leq d$.}\label{eq:CumulantMoment}
\end{equation}

\medskip\noindent\textbf{Characteristic function.}
Whenever the PDF of a random variable $X$ does not exist, it is useful to instead consider the \emph{characteristic function} of $X$, defined by
\begin{equation}
	\varphi_X(t)=\E\big[e^{itX}\big],\label{eq:characteristic}
\end{equation}
as it completely determines the behavior and properties of the probability distribution of $X$. The characteristic function of a random variable always exists. If $X$ admits a probability density function, then its characteristic function is simply the Fourier transform (with sign reversal) of this probability density function. In the same way, if the moment generating function $M(t)$ of a random variable exists, then $\varphi_X(-it)=M(t)$. Lastly, we note that Hsu \cite[Thm.~4.1]{hsu1951absolute} determined that if $\varphi$ is the characteristic function of $X$ and 
$$P_n(t) = \sum_{k=0}^n \frac{\varphi^{(k)}(0)}{k!} t^k,$$ 
then the $d$th absolute moment of the random variable $X$, for a real number $d\geq 1$, is given by the formula
\begin{equation}
	\frac{\E\big[|X|^d\big]}{\Gamma(d+1)} = \frac{2 \sin\!\big( \tfrac{d\pi}{2} \big)}{\pi} \int_0^\infty \frac{\Re\big( P_{\lfloor d \rfloor}\!(t)-\varphi(t) \big)}{t^{d+1}} \,\mathrm{d}t.\label{eq:absolutemoment}
\end{equation}

\subsection{Famous inequalities}
\label{subsec:ineq}

In this document, and particularly in Section \ref{Section:LowerandUpperBounds}, many famous inequalities are used. We state them explicitly for clarity and completeness.

\medskip\noindent\textbf{H\"older's inequality.} Let $X,Y \hspace{-.5pt}\in\hspace{-.5pt} (\Omega,\mathcal{F},\vec{P})$ be random variables and let $p,q \hspace{-.5pt}\in\hspace{-.5pt} [1,\infty]$ be such that $\frac{1}{p}+\frac{1}{q}=1$. Then
\vspace{-4pt}
\begin{equation*}
\E\hspace{-.3pt}\big[ |XY|\big] \leq \E\hspace{-.3pt}\big[ |X|^p\big]^{\smash{\frac{1}{p}}} \, \E\hspace{-.3pt}\big[ |Y|^q\big]^{\smash{\frac{1}{q}}}.
\end{equation*}

\medskip\noindent\textbf{Jensen's inequality.} Let $(\Omega,\mathcal{F},\vec{P})$ be a probability space and let $f :\Omega \to \R^d$ be an integrable function and $X$ an integrable real-valued random variable. Moreover, let $\vec{x}_{1},\ldots ,\vec{x}_{n} \in \R^d$ and $a_{1},\ldots ,a_{n} \geq 0$. If $\varphi : \mathbb {R}^d \to \mathbb {R}$ is a convex function, then:
\begin{enumerate}[label=(\alph*)]
	\itemsep.5em
	\item ${\displaystyle \varphi \left(\E [X]\right)\leq \E [\varphi (X)]}$ \hfill (Probabilistic form);
	\item ${\displaystyle \varphi \!\left(\int _{\Omega }f\,\mathrm {d} \mu \right)\leq \int _{\Omega }\varphi \circ f\,d\vec{P} }$
	\hfill (Integral form);
	\item ${\displaystyle \varphi \!\left({\frac {\sum a_{i}\vec{x}_{i}}{\sum a_{i}}}\right)\leq {\frac {\sum a_{i}\varphi (\vec{x}_{i})}{\sum a_{i}}}}$\hfill  (Finite form).
\end{enumerate}
The inequalities are reversed if $\varphi$ is concave.

\bigskip\noindent\textbf{Lyapunov's inequality.} Let $X\in L^t(\Omega,\mathcal{F},\vec{P})$ and suppose that $0<s\leq t$. Then
\vspace{-3pt} 
\begin{equation}
	\E\hspace{-.3pt}\big[ |X|^s\big]^{\frac{1}{s}} \leq \E\hspace{-.3pt}\big[ |X|^t\big]^{\frac{1}{t}}. \label{eq:Lyapunov}
\end{equation}

\medskip\noindent\textbf{Marcinkiewicz--Zygmund's inequality} \cite{Marc1937}\textbf{.} Let $1\leq d < \infty$. If $X_i\in L^d(\Omega,\mathcal{F},\vec{P})$, for $i=1,2,\dots,n$, are independent random variables such that $\E[X_i]=0$, then 
\begin{equation}
	A_d \E \!\left[ \bigg( \sum_{i=1}^n |X_i|^2 \bigg)^{\!\frac{d}{2}} \right] \!\leq \E \!\left[ \bigg| \sum_{i=1}^n X_i \bigg|^{d} \right] \leq \!B_d \E \!\left[ \bigg( \sum_{i=1}^n |X_i|^2 \bigg)^{\!\frac{d}{2}} \right], \label{Eq:M-Z}
\end{equation}
where $A_d$ and $B_d$ are positive constants which depend only on $d$ and not on $n$ nor on the underlying distribution of the random variables. 

\medskip

\begin{remark}\label{rem - constant}
	The optimal constants $A_d$ and $B_d$ in \eqref{Eq:M-Z} are not known, although it is immediate that $A_d\leq 1 \leq B_d$ (consider the case $n=1$). 
\end{remark}

\subsection{Special functions}
\label{subsec:func}

\medskip\noindent\textbf{Complete Bell polynomials.} 
The \emph{complete Bell polynomials} \cite[Sec.~II]{Bell} are  the polynomials $B_\ell$ defined by 
\begin{equation}\label{eq:ExpoBell}
\sum_{\ell=0}^{\infty} B_{\ell}(x_1, x_2, \ldots, x_{\ell}) \frac{t^{\ell}}{\ell !}=\exp\!\bigg( \sum_{j=1}^{\infty} x_j \frac{t^j}{j!}\bigg).
\end{equation} We remark that the complete Bell polynomials arise frequently in probability theory given the form of \eqref{eq:CumulantGen}. In particular, the moments and cumulants of a random variable $X$ satisfy the relation $\mu_{\ell}=B_{\ell}(\kappa_1, \kappa_2 , \ldots, \kappa_{\ell})$ for all nonnegative integers $\ell$, in which $B_0=1$. The first several complete Bell polynomials are given by
\begin{align}
B_0& = 1, \nonumber\\
B_1(x_1)&= x_1, \nonumber\\
B_2(x_1,x_2)&= x_1^2+x_2, \nonumber \\
B_3(x_1,x_2,x_3)&= x_1^3+3 x_1 x_2+x_3, \quad \text{and}\nonumber\\
B_4(x_1,x_2,x_3,x_4)&= x_1^4+6x_1^2x_2+4x_1x_3+3x_2^2+x_4. \label{eq:Bell4}
\end{align}

\medskip\noindent\textbf{Kummer's confluent hypergeometric function.} 
Given an $a\in \C$ and $n \in \N$, the \emph{rising factorial} $a^{(n)}$ is defined by 
$a^{(0)}=1$ and $a^{(n)}=a(a+1)(a+2)\cdots (a+n-1)$. The \emph{Kummer confluent hypergeometric functions} are then defined by
\begin{equation}
	\Kum(a;b;z) = \sum _{{n=0}}^{\infty }{\frac  {a^{{(n)}}z^{n}}{b^{{(n)}}n!}}. \label{eq:Kummer}
\end{equation}
Some special values of this function are
\begin{align}
	\Kum(a;b;0) &= 1, \label{Eq:Kummer0} \\
	\Kum(0;b;z) &= 1, \nonumber \\
	\Kum(a;a;z) &= e^z, \quad \text{and}\nonumber \\
	\Kum(-1;b;z) &= 1-\tfrac{z}{b}. \nonumber
\end{align}


\section{Doubly stochastic matrices and majorization}\label{Section:Birkhoff}

A square matrix with nonnegative entries is \emph{doubly stochastic} if each row and column sums to $1$. The theory surrounding these objects is rich, and permeates several areas of mathematics. We present in this section three results central to this theory which are crucial to the proof of Lewis' classification theorem in Subsection \ref{Subsection:Classify}: Birkhoff's theorem in Subsection \ref{subsec:Birkhoff}, Hardy, Littlewood, and P\'olya's theorem in Subsection \ref{subsec:HLP}, and the Schur--Ostrowski criterion in Subsection \ref{subsec:S-O}. A preliminary classical result is first needed and developed in Subsection \ref{subsec:Hall}.


\subsection{Hall's theorem}
\label{subsec:Hall}

A \emph{graph} is a pair $G = (V, E)$, in which $V$ is a set of elements called vertices, and $E$ is a set of pairs of vertices, whose elements are called edges. A graph is \emph{bipartite} if the set of vertices $V$ can be divided into two disjoint and independent sets $V_1$ and $V_2$, that is, every edge in $E$ connects a vertex in $V_1$ to one in $V_2$. In that case, one usually write $G=(V_1,V_2,E)$. Lastly, a \emph{$V_1$-perfect matching} is a set of disjoint edges, which covers every vertex in $V_1$.

\definecolor{myblue}{RGB}{80,80,160}
\definecolor{mygreen}{RGB}{80,140,80}

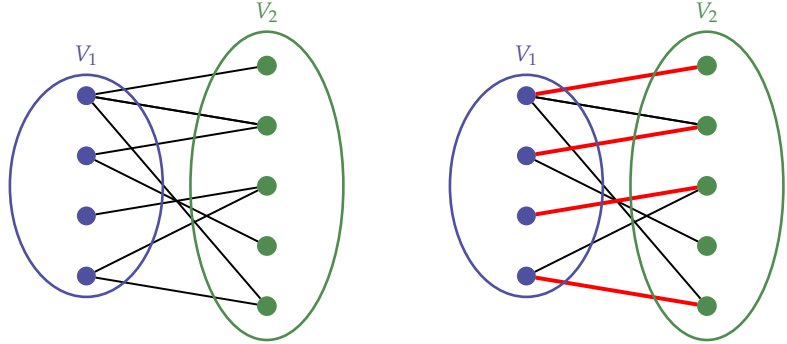
\begin{figure}[!htb]
	\vspace{-10pt}
	\centering
	\begin{minipage}{.5\textwidth}
		\centering
		\begin{tikzpicture}[every fit/.style={ellipse,draw,inner sep=-8pt,text width=2cm, line width=1pt}, scale=0.4]
			\GraphInit[vstyle=Normal]
			\SetUpVertex[Math,Lpos=-180]
			\tikzset{VertexStyle/.style = {shape=circle, fill=myblue,
					minimum size=2pt,inner sep=-2pt,text opacity=0}
			}
			\grEmptyPath[form=2,x=0,y=0,RA=2,rotation=90,prefix=V]{4}
			\SetVertexShade[BallColor=mygreen,OuterSep=0pt]
			\SetUpVertex[Lpos=0]
			\tikzset{VertexStyle/.style = {shape=circle, fill=mygreen,
					minimum size=2pt,inner sep=-2pt,text opacity=0}
			}
			\grEmptyPath[form=2,x=6,y=0,RA=2,rotation=90,prefix=U]{5}
			\SetUpEdge[lw=.75pt,color=black]
			\Edges(U4,V3,U3,V2,U1)
			\Edges(U3,V3,U0,V0,U2,V1)
			\node [mygreen,fit=(U0) (U4),label=above:\textcolor{mygreen}{$V_2$}] {};
			\node [myblue,fit=(V0) (V3),label=above:\textcolor{myblue}{$V_1$}] {};
		\end{tikzpicture}
	\end{minipage}%
	\begin{minipage}{0.5\textwidth}
		\centering
		\begin{tikzpicture}[every fit/.style={ellipse,draw,inner sep=-8pt,text width=2cm, line width=1pt}, scale=0.4]
			\GraphInit[vstyle=Normal]
			\SetUpVertex[Math,Lpos=-180]
			\tikzset{VertexStyle/.style = {shape=circle, fill=myblue,
					minimum size=2pt,inner sep=-2pt,text opacity=0}
			}
			\grEmptyPath[form=2,x=0,y=0,RA=2,rotation=90,prefix=V]{4}
			\SetVertexShade[BallColor=mygreen,OuterSep=0pt]
			\SetUpVertex[Lpos=0]
			\tikzset{VertexStyle/.style = {shape=circle, fill=mygreen,
					minimum size=2pt,inner sep=-2pt,text opacity=0}
			}
			\grEmptyPath[form=2,x=6,y=0,RA=2,rotation=90,prefix=U]{5}
			\SetUpEdge[lw=.75pt,color=black]
			\Edges(V3,U3,V2,U1)
			\Edges(U3,V3,U0,V0,U2,V1)
			\SetUpEdge[lw=1.5pt,color=red]
			\Edges(U4,V3)
			\Edges(U3,V2)
			\Edges(U2,V1)
			\Edges(U0,V0)
			\node [mygreen,fit=(U0) (U4),label=above:\textcolor{mygreen}{$V_2$}] {};
			\node [myblue,fit=(V0) (V3),label=above:\textcolor{myblue}{$V_1$}] {};
		\end{tikzpicture}
	\end{minipage}
	\caption{A bipartite graph (\textsc{left}) and a $V_1$-perfect matching (\textsc{right}).}
\end{figure}

A well-known theorem of Hall \cite{Hall1935} characterizes when a finite bipartite graph has a perfect matching.
Below $|S|$ denotes the cardinality of a finite set $S$.

\begin{theorem}[Hall]\label{Thm:Hall}\index{Hall's marriage theorem}
	If $G=(V_1,V_2,E)$ is a bipartite graph, then there is a $V_1$-perfect matching if and only if every set $W\subset V_1$ of vertices is connected
	to at least $|W|$ vertices in $V_2$.
\end{theorem}
\begin{proof}
	If a $V_1$-perfect matching exists, then $W$ must have at least $|W|$ neighbors simply because of the edges of the perfect matching. The converse is proved by induction on the size of $|V_1|$. The result is clear if $|V_1|=1$. Hence, suppose that $|V_1|>1$ and that the result holds for every order smaller than $|V_1|$.
	\smallskip
	
	Choose an arbitrary $v_1\in V_1$. Then $v_1$ must then have at least one neighbor $v_2 \in V_2$ by hypothesis. Consider the graph $G'$ induced by removing $v_1$ and $v_2$ from $G$. 
	Finding a $(V_1\!\setminus\!\{v_1\})$-perfect matching in $G'$ and adding the edge $(v_1,v_2)$ would yield a $V_1$-perfect matching in $G$. 
	If every set $W\subset V_1\!\setminus\!\{v_1\}$ of vertices is connected
	to at least $|W|$ vertices in $V_2\!\setminus\! \{v_2\}$, then a $(V_1\!\setminus\!\{v_1\})$-perfect matching in $G'$ exists by the induction hypothesis, so we are done. Otherwise, $W$ must have exactly $|W|$ neighbors in $V_2$.
	
	In that case, denote by $U\subset V_2$ the set of neighbors of $W$. Note that $|U|=|W|$. By the induction hypothesis, there is a $W$-perfect matching that matches each vertex of $W$ to a vertex of $U$. Moreover, if $W' = V_1\!\setminus\!W$, then $W'\cup W$ has at least $|W' \cup W| = |W'|+|W| = |W'|+|U|$ neighbors in $V_2$.  Note that $W'$ must then have at least $|W'|$ neighbors in $V_2\!\setminus\! U$, since all the vertices of $W$ can be matched to a vertex of $U$ in a one-to-one manner. Hence, by the induction hypothesis, there is a $W'$-perfect matching that matches all the vertices of $W'$ to $V_2\!\setminus\! U$. Considering the disjoint union of the above $W$-perfect matching and $W'$-perfect matching yield a $(W' \cup W)$-perfect matching in $G$. Since $W' \cup W=V_1$, we are done.
\end{proof}

\vspace{-12pt}
\subsection{Birkhoff's theorem}
\label{subsec:Birkhoff}
\vspace{-2pt}

The following is a celebrated theorem of Birkhoff \cite{Birkhoff}. We remark that 
$n^2-n+1$ works in place of $n^2$ \cite[Thm.~8.7.2]{HJ}.

\begin{theorem}[Birkhoff]\label{Theorem:Birkhoff}
	If $D \in \M_n$ is doubly stochastic, then there exist nonnegative constants $c_1,c_2,\ldots,c_{n^2}$ satisfying $\sum_{i = 1}^{n^2} c_i = 1$ and permutation matrices $P_1,P_2,\ldots,P_{n^2}$ such that
	\begin{equation*}
		D = \sum_{i = 1}^{\,n^2} c_i P_i.
	\end{equation*}
\end{theorem}

\begin{proof}
	We proceed by induction on the number of nonzero entries of the $n\times n$ doubly stochastic matrix $D$. Note that $D$ has at least $n$ nonzero entries, otherwise some row will sum to $0$. If $D$ has exactly $n$ nonzero entries, then every row and column contains exactly one nonzero entry, which is $1$, so $D$ is a permutation matrix.
	
	Suppose that $D$ has $m\geq n+1$ nonzero entries and that the conclusion holds for any doubly stochastic matrix with less than $m$ nonzero entries. Consider the bipartite graph $G=(V_1,V_2,E)$, in which 
	$V_1=\{1,\ldots,n\}$ is the set of rows, $V_2=\{1,\ldots,n\}$ is the set of columns, and $(i,j)\in E$ if and only if $d_{ij}>0$. If $\{i_1,i_2,\ldots,i_k\}=W\subset V_1$ has less than $|W|=k$ neighbors, say $\{j_1,j_2,\ldots,j_m\}$ where $m<k$, then
	\vspace{-1pt}
	\begin{equation*}
	k = \sum_{\iota=1}^k \sum_{j=1}^n d_{i_\iota j} = \sum_{\iota=1}^k \sum_{\ell=1}^m d_{i_\iota j_\ell} \leq \sum_{i=1}^n \sum_{\ell=1}^m d_{i j_\ell} = \sum_{\ell=1}^m \sum_{i=1}^n d_{i j_\ell} = m < k.
	\end{equation*}
	This contradiction shows that every $W\subset V_1$ contains at least $|W|$ neighbors. Hall's condition from Theorem \ref{Thm:Hall} thus ensures that a $V_1$-perfect matching exists in $G$. Since $|V_1|=|V_2|$, the matching is one-to-one, and it then corresponds to some permutation matrix $P$ satisfying $p_{ij}=1$ only if $d_{ij}>0$.
	
	If $\mu:=\min\{d_{ij}: p_{ij}=1\}$ denotes the smallest nonzero entry of $D$ among those corresponding to the position of a $1$ in $P$, then $\tfrac{1}{1-\mu}(D-\mu P)$ is a doubly stochastic matrix with one less nonnegative entry than $D$. The induction hypothesis provides nonnegative constants $a_1,a_2,\ldots,a_{n^2}$ satisfying $\sum_{i = 1}^{n^2} a_i = 1$ such that
	\vspace{-2pt}
	\begin{equation*}
		\tfrac{1}{1-\mu}(D-\mu P) = \sum_{i = 1}^{\,n^2} a_i P_i.
	\end{equation*}
	Multiplying both sides by $1-\mu$ and adding $\mu P$ yield the desired result.
\end{proof}

The set of all $n\times n$ doubly stochastic matrices is a convex polytope, and
Birkhoff's theorem ensures that it is the unique convex polytope having the $n\times n$ permutation matrices as its extreme points.

Several proofs of Birkhoff's theorem exist. Between 1953 and 1963, at least eight different proofs were given: Hoffman \& Wielandt \cite{HoffmanWielandt1953}, Dulmage \& Halperin \cite{Dulmage1955}, Hammersley \& Mauldon \cite{Hammersley1956}, Mirsky \cite{Mirsky1958}, Berge \cite{berge1958theorie, Berge1971}, Vogel \cite{Vogel1961}, R\'ev\'esz \cite{Revesz1962} and Ryser \cite{Ryser1963}. The proof presented above is from
Dulmage \& Halperin \cite{Dulmage1955}.


\subsection{Hardy, Littlewood and P\'olya's theorem}
\label{subsec:HLP}

Majorization and doubly stochastic matrices are closely related concepts. The following theorem of
Hardy, Littlewood, and P\'olya characterizes this connection \cite{Hardy}.

\begin{theorem}[Hardy, Littlewood, P\'olya]\label{Lemma:MajorStoch}
	If $\vec{y}\prec\vec{x}$, then there exists a doubly stochastic matrix $D$ such that $\vec{y} = D \vec{x}$. 
\end{theorem}

\begin{proof}
	We proceed by induction. The case $n=2$ is trivial, so assume that the implication holds for all dimensions up to $n-1$, and let $\vec{x},\vec{y}\in\R^n$. Since $\widetilde{\vec{x}}$ and $\widetilde{\vec{y}}$ are obtained from $\vec{x}$ and $\vec{y}$ by permutations, and since the product of a doubly stochastic matrix and a permutation matrix is doubly stochastic, assume without loss of generality that $\vec{x}=\widetilde{\vec{x}}$ and $\vec{y}=\widetilde{\vec{y}}$. Since $\vec{y} \prec \vec{x}$, it follows that $x_n \leq y_1 \leq x_1$. Therefore, there exists an integer $k\in\{1,2,\ldots,n\} $ such that 
	\begin{equation*}
	x_k \leq y_1 \leq x_{k-1}.
	\end{equation*}
	Moreover, since $x_k \leq y_1 \leq x_1$, it follows that $y_1 = tx_1 +(1-t)x_k$ for some $t\in [0,1]$.
	Let $T_1 = tI+(1-t)P$, where $P$ is the permutation matrix which interchange the coordinate $1$ and $k$, and note that $T_1$ is doubly stochastic. Then
	\begin{equation*}
	T_1 \vec{z} = \big( tz_1+(1-t)z_k , \, z_2 , \ldots , z_{k-1} ,\, (1-t)z_1+tz_k , \, z_{k+1} , \ldots ,\, z_n \big)^{\!\intercal}
	\end{equation*}
	for all $\vec{z}\in \R^n$. Note that the first coordinate of $T_1\vec{x}$ is $y_1$. Let
	\begin{align*}
		\vec{y'} &= (y_2,\dots,y_n)^{\intercal},\\
		\vec{x'} &= (x_2,\dots,x_{k-1},(1-t)x_1+tx_k,x_{k+1},\dots,x_n)^{\intercal},
	\end{align*}
	and observe that $T_1\vec{x} = \smash{\left[\begin{smallmatrix}
		y_1\\\vec{x'}
	\end{smallmatrix} \right]}$. Since 
	\begin{equation*}
	x_1\geq \dots \geq x_{k-1} \geq y_1 \geq y_2 \geq \dots \geq y_n,
	\end{equation*}
	it follows that
	\begin{equation*}
	\sum_{j=2}^m y_j' = \sum_{j=2}^m y_j \leq \sum_{j=2}^m x_j = \sum_{j=2}^m x_j'\, ~\quad\mathrm{for}~~m=2,3,\dots, k-1.
	\end{equation*}
	For $m\in\{k,k+1,\ldots,n\}$, 
	\vspace{-2pt}
	\begin{align*}
		\sum_{j=2}^m x_j' &= \sum_{j=2}^{k-1} x_j +[(1-t)x_1 + tx_k] + \sum_{j=k+1}^m x_j\\[-1pt]
		&= \sum_{j=1}^m x_j -tx_1 + (t-1)x_k = \sum_{j=1}^m x_j -y_1\\[-1pt]
		&\geq \sum_{j=1}^m y_j -y_1 = \sum_{j=2}^m y_j = \sum_{j=2}^m y_j'.\\[-21pt]
	\end{align*}
	Moreover, the inequality becomes an equality when $m=n$ since $\vec{y}\prec \vec{x}$. Therefore, $\vec{y'} \prec \vec{x'}$ and the induction hypothesis ensures that there exists a doubly stochastic matrix $D'$ such that $\vec{y'} = D' \vec{x'}$. Thus,
	\begin{equation*}
	\begin{bmatrix}
		1&0\\0&D'
	\end{bmatrix} \!T_1 \vec{x} = \begin{bmatrix}
	1&0\\0&D'
	\end{bmatrix}\!\begin{bmatrix}
		y_1\\\vec{x'}
	\end{bmatrix}  = \begin{bmatrix}
		y_1\\\vec{y'}
	\end{bmatrix}  = \vec{y}.
	\end{equation*}
	Since both $\left[\begin{smallmatrix}
		1&0\\0&D'
	\end{smallmatrix}\right]$ and $T_1$ are doubly stochastic and the product of doubly stochastic matrices is doubly stochastic, this establishes the desired result.
\end{proof}

\vspace{-13pt}
\subsection{The Schur--Ostrowski criterion}\label{subsec:S-O}
\vspace{-3pt}

The Schur--Ostrowski criterion is a powerful  tool for establishing the Schur-convexity of a symmetric function. The following theorem appears in \cite[p.~259]{Roberts}.

\begin{theorem}[Schur--Ostrowski]
	Suppose $f$ is a symmetric function on $\R^n$ with continuous partial derivatives. Then $f$ is Schur-convex if and only if
	\begin{equation}\label{eq:S-O}
		(x_i-x_j)\bigg( \frac{\partial}{\partial x_i}-\frac{\partial}{\partial x_j} \bigg) f(x_1, x_2, \ldots, x_n)\geq 0 
	\end{equation}
	for all $1\leq i<j\leq n$, with equality if and only if $x_i=x_j$.
\end{theorem}
\begin{proof}
	Since $f$ is a symmetric function on $\R^n$, without loss of generality suppose that $x_1 \geq x_2 \geq \cdots \geq x_n$. 
	Then the desired result is equivalent to showing that $\frac{\partial}{\partial x_k} f(\vec{x})$ is nonincreasing in $k=1,2,\ldots,n$ for all $\vec{x}$ with nonincreasing coefficients. 
	Define $\overline{x}_k = \sum_{i=1}^{\smash{k}} \widetilde{x}_k$ for $k=1,2,\ldots,n$ (and similarly for $\vec{y}$). By definition, $\vec{x} \prec \vec{y}$ if and only if 
	\begin{equation*}
		\overline{x}_n = \overline{y}_n
		\quad \text{and} \quad
		\overline{x}_k \leq \overline{y}_k
		\quad \text{for}~ 1 \leq k \leq n-1.
	\end{equation*}
	In other words, $\vec{x} \prec \vec{y}$ if and only if $\overline{\vec{x}} \leq \overline{\vec{y}}$ in the componentwise ordering (that is, $\vec{u} \leq \vec{v}$ if and only if $u_k \leq v_k$ for $k=1,2,\ldots,n$). But for the componentwise ordering, it is immediate that $\vec{u} \leq \vec{v} \implies g(\vec{u}) \leq g(\vec{v})$ if and only if the function $g:\R^n \to \R$ is increasing in each argument. Defining
	\begin{equation*}
		g(z_1,z_2,\ldots,z_n) = f(z_1,z_2-z_1,\ldots,z_n-z_{n-1}),
	\end{equation*}
	we find that $g(\overline{\vec{x}}) = f(\vec{x})$. It follows that $\vec{x} \prec \vec{y} \implies f(\vec{x}) \leq f(\vec{y})$ if and only if $\overline{\vec{x}} \leq \overline{\vec{y}} \implies g(\overline{\vec{x}}) \leq g(\overline{\vec{y}})$, which in turn is verified if and only if $g$ is increasing in each argument. 
	Interpreting $f$ as a function of $\vec{x}$, we can write $g(\vec{z})=f(x_1(\vec{z}),x_2(\vec{z}),\ldots,x_n(\vec{z}))$, where $x_1(\vec{z}) = z_1$ and $x_k(\vec{z})=z_k-{z_{k-1}}$ for $k=2,3,\ldots,n$. Then, if $k=1,2,\ldots,n-1$,
	\begin{align*}
		\frac{d g}{d z_k} (\vec{z}) &= \frac{d f}{d z_k} (z_1,z_2-z_1,\ldots,z_n-z_{n-1}) \\
		&= \frac{\partial f}{\partial x_k} \frac{\partial x_k}{z_k} + \frac{\partial f}{\partial x_{k+1}} \frac{\partial x_{k+1}}{z_k} \\
		&= \frac{\partial f}{\partial x_k} - \frac{\partial f}{\partial x_{k+1}} ,
	\end{align*}
	while $\frac{d g}{d z_n} (\vec{z}) = \frac{\partial f}{\partial x_n}(\vec{x})$. 
	Hence, $g$ is increasing in each argument if and only if $\frac{\partial f}{\partial x_n}(\vec{x})\geq 0$ and 
	$\left(\frac{\partial f}{\partial x_k} - \frac{\partial f}{\partial x_{k+1}}\right)\!(\vec{x}) \geq 0$, that is, if and only if $\frac{\partial}{\partial x_k} f(\vec{x})$ is decreasing in $k=1,2,\ldots,n$, which completes the proof.
\end{proof}


\section{Hunter's positivity theorem}\label{Section:Hunter}

Hunter proved that the even degree CHS polynomials \eqref{eq:CHS} are positive definite \cite{Hunter}. Hunter's  theorem is a recurring theme in our paper and motivates a lot of our work.

\begin{theorem}[Hunter]\label{Theorem:Hunter}
Let $d \geq 0$ be an even integer.  Then $h_d(x_1,x_2,\ldots,x_n)$ is a positive definite function on 
$\R^n$.  That is, $h_d(x_1,x_2,\ldots,x_n) > 0$ for all $\vec{x} \neq \vec{0}$.
\end{theorem}

\begin{remark}
The case $d=0$ of Hunter's theorem is trivial, and the case $d=2$ can be addressed with a simple algebraic argument. In particular,  \eqref{eq:CHSPowerSumRec} implies
\begin{equation}
h_2(x_1,x_2,\ldots,x_n)=\frac{1}{2} \sum_{i=1}^n x_i^2 + \frac{1}{2} (x_1+x_2 + \cdots+x_n)^2.\label{eq:CHS2Bound}
\end{equation}
We see that $h_2(x_1,x_2,\dots,x_n) > 0$ unless $x_1=x_2=\cdots=x_n=0$. We remark that the formula $4h_{4}=h_{3}p_{1}+h_{2}p_{2}+h_1 p_{3} + p_4$ also follows from \eqref{eq:CHSPowerSumRec} but does not appear to admit a similar sum-of-squares proof of positivity for the case $d=4$.
\end{remark}

Hunter's positivity theorem for CHS polynomials has been rediscovered and proved many times. Subsections \ref{Subsection:CHSHunter} and \ref{Subsection:TaoCHS} contain Hunter's original proof and Tao's proof, respectively. Subsections \ref{Subsection:CHSExpectation} and \ref{Subsection:CHSExtension} contain newer probabilistic proofs of positive definiteness. Additional proofs of positive definiteness can be found, for example, in \cite[Lem.~3.1]{Barvinok}, \cite{Baston}, \cite[Cor.~17]{GOOY}, and \cite[Thm.~2.3]{Roventa}.


\subsection{Hunter's proof}\label{Subsection:CHSHunter}

Hunter established a lower bound for the even degree CHS polynomials to prove positive definiteness. Theorem \ref{Theorem:HunterBound} below is \cite[Thm.~1]{Hunter} and it immediately implies Theorem \ref{Theorem:Hunter}. In particular, the CHS polynomials $h_d(\vec{x})$ are homogeneous and can always be scaled so that $\vec{x}\neq \vec{0}$ lies on the unit sphere in $\R^n$.  We note that a generalization of Theorem \ref{Theorem:HunterBound} for fractional degrees is proved in Subsection \ref{Subsection:CHSExtension}.

\begin{theorem}[Hunter]\label{Theorem:HunterBound}
Let $d \geq 2$ be an even integer. The CHS polynomials satisfy 
\begin{equation*}
h_d(x_1, x_2, \ldots, x_n)\geq \frac{1}{2^{\frac{d}{2}} (\frac{d}{2})!} 
\end{equation*} on the unit sphere sphere $x_1^2+x_2^2+\cdots +x_n^2=1$. Moreover, equality holds if and only if $d=2$ and $x_1+x_2+\cdots +x_n=0$.
\end{theorem}

\begin{proof}
 We write $d=2k$ for some $k\in \N$ and proceed by induction on $k$. If $k=1$, then the bound $h_2\geq \frac{1}{2}$ follows from relation \eqref{eq:CHS2Bound}. Moreover, equality holds if and only if $x_1+x_2+\cdots +x_n=0$. The unit sphere in $\R^n$ is compact. The  CHS polynomials must therefore attain an extremum $\vec{x}$ on the unit sphere in $\R^n$. The method of Lagrange multipliers ensures there exists $\lambda$ such that
\begin{equation}
\frac{\partial h_{2k}(\vec{x})}{\partial x_i}+2\lambda x_i=0 \label{eq:Lagrange}
\end{equation}
for each $i=1, 2,\ldots, n$. Multiply each side of \eqref{eq:Lagrange} by $x_i$ and sum over all $i$ to get
\begin{equation}
\sum_{i=1}^n x_i\frac{\partial h_{2k}(\vec{x})}{\partial x_i}+2\lambda=0.\label{eq:Derivation}
\end{equation} 
The derivation $\sum_{i=1}^n x_i\frac{\partial}{\partial x_i}$ is a linear operator on the vector space of degree-$2k$ homogeneous polynomials and has eigenvalue $2k$. Therefore,  \eqref{eq:Derivation} implies
\begin{equation}
2kh_{2k}(\vec{x})+2\lambda=0.\label{eq:Lagrange2}
\end{equation} 
The classic identity $\frac{\partial h_d(\vec{x})}{\partial x_i}=h_{d-1}(\vec{x}, x_i)$ follows immediately from differentiating both sides of the generating function \eqref{eq:CHSGenerating}. Combining this with \eqref{eq:Lagrange} and \eqref{eq:Lagrange2} yields
\begin{equation}
h_{2k-1}(\vec{x}, x_i)=2kx_ih_{2k}(\vec{x})\label{eq:Stationary}
\end{equation}
for each $i=1, 2,\ldots, n$. If all the coefficients of $\vec{x}$ are identical (that is, if $\vec{x}=\pm n^{-\frac{1}{2}}(1,1,\ldots,1)$), then a computation provides
\begin{equation}
h_{2k}(\pm n^{-\frac{1}{2}}, \pm n^{-\frac{1}{2}}, \ldots, \pm n^{-\frac{1}{2}})=\frac{(n+2k-1)!}{(2k)! (n-1)! n^k}>\frac{1}{2^k k!}
\end{equation} 
and the desired bound is satisfied. Otherwise, any stationary point $\vec{x}$ must have at least two distinct coordinates. Let $a$ and $b$ denote the values of these coordinates. The identity $h_{d}(\vec{x}, a)-h_d(\vec{x}, b)=(a-b)h_{d-1}(\vec{x}, a, b)$ holds for $a\neq b$ and follows from the generating function for $h_{d}(\vec{x}, a)-h_d(\vec{x}, b)$ \cite[Lem.~1]{Hunter}. Then \eqref{eq:Stationary} implies
\begin{equation*}
(a-b)h_{2k-2}(\vec{x}, a, b)=h_{2k-1}(\vec{x}, a)-h_{2k-1}(\vec{x}, b)=2k(a-b)h_{2k}(\vec{x}),
\end{equation*} which yields $2kh_{2k}(\vec{x})=h_{2(k-1)}(\vec{x}, a, b)$. Since $\sum_{i=1}^n x_i^2+a^2+b^2> 1$, the inductive hypothesis ensures that
\vspace{-2pt}
\begin{equation*}
h_{2k}(\vec{x})=\frac{h_{2(k-1)}(\vec{x}, a,b)}{2k}>\frac{1}{2k}\cdot \frac{1}{2^{k-1}(k-1)!}=\frac{1}{2^k k!}. \qedhere
\end{equation*} 
\end{proof}

\vspace{-13pt}
\subsection{Tao's proof}\label{Subsection:TaoCHS}
\vspace{-2pt}

Tao established the positive definiteness and Schur-convexity of the CHS polynomials in his August 6, 2017 blog entry \cite[Thm.~1]{Tao}.

\begin{theorem}[Tao]\label{Thm:Tao}
Let $n\in \N$ and $d\geq 0$ be even. The following hold for $\vec{x}\in \R^n$.

\begin{enumerate}[label=(\alph*)]
\itemsep.5em
\item (Positive definiteness) $h_d(\vec{x})\geq 0$ with equality if and only if $\vec{x}=\vec{0}$.
\item (Schur-convexity) $h_d(\vec{x})\leq h_d(\vec{y})$ whenever $\vec{x}\prec \vec{y}$. Moreoever, equality holds if and only if $\vec{x}$ is permutation of $\vec{y}$.
\item (Schur--Ostrowski criterion) The relation $(x_i-x_j)\big( \frac{\partial}{\partial x_i}-\frac{\partial}{\partial x_j}\big)h_d(\vec{x})\geq 0$ holds for all $1\leq i <j\leq n$ with strict inequality unless $x_i=x_j$.
\end{enumerate}
\end{theorem}

\begin{proof}
We proceed by induction on the even integer $d\geq 0$. The results hold trivially for $d=0, 2$, 
so assume that $d\geq 4$ and the claims hold for $d-2$. Apply the differential operator $(x_i-x_j)\big( \frac{\partial}{\partial x_i}-\frac{\partial}{\partial x_j}\big)$ to each side of \eqref{eq:CHSGenerating} and get
\vspace{-3pt}
\begin{equation}
(x_i-x_j)\Big( \frac{\partial}{\partial x_i}-\frac{\partial}{\partial x_j}\Big)h_d(\vec{x})=(x_i-x_j)^2h_{d-2}(\vec{x}, x_i, x_j).\label{eq:Two} \vspace{-1pt}
\end{equation} Part (c) then follows from (a) and the inductive hypothesis. Schur--Ostrowski's criterion establishes (b) from (c). Finally, (a) follows from the fact that $(x, x, \ldots, x)\prec \vec{x}$, in which $x$ is the arithmetic mean of $x_1, x_2, \ldots, x_n$. In particular, (b) implies 
\vspace{-1pt}
\begin{equation}
h_d(\vec{x})\geq h_d(x, x,\ldots, x)=x^d h_d(1, 1,\ldots, 1)=x^d\binom{n+d-1}{d}. \qedhere
\end{equation}
\end{proof}

\begin{remark}
Tao writes in his blog, ``The proof in Hunter of positive definiteness is arranged a little differently than the one above, but still relies ultimately on the identity (2). I wonder if there is a genuinely different way to establish positive definiteness that does not go through this identity.'' Tao is referring to identity \eqref{eq:Two}. The proofs of positive definiteness in the following subsections do not  go directly through \eqref{eq:Two}.
\end{remark}


\subsection{Probabilistic proof}\label{Subsection:CHSExpectation}
\vspace{-4pt}
The following proof of Theorem \ref{Theorem:Hunter} appears in \cite{Aguilar},
although the idea is from an anonymous comment in \cite{Tao}.
A \emph{standard exponential} random variable is a random variable $X$ distributed according to the PDF
\vspace{-2pt}
\begin{equation*}
f_X(x)=\begin{cases}
e^{-x} & \text{if $x\geq 0$},\\
0 & \text{if $x<0$}.
\end{cases}\vspace{-2pt}
\end{equation*}
The moments of $X$ satisfy $\E[X^k]=k!$ for nonnegative integral $k$. Let $X_1, X_2, \ldots, X_n$ be independent standard exponential random variables and let $\bvec{X}=(X_1, X_2, \ldots, X_n)$.  The linearity of expectation, independence, and the multinomial theorem imply 
\vspace{-1pt}
\begin{align*}
\E[ \inner{ \bvec{X}, \vec{x} }^d]
&= \E[( x_1 X_1+ x_2 X_2\cdots + x_nX_n)^d] \\
&= \E\!\left[ \sum_{k_1+k_2+\cdots+ k_n=d}  \frac{d!}{k_1!\,k_2!\,\cdots\, k_n!} x_1^{k_1}x_2^{k_2}\cdots x_n^{k_n}X_1^{k_1}X_2^{k_2}\cdots X_n^{k_n}  \right]\\
&= d! \sum_{k_1+k_2+\cdots+ k_n=d}  \frac{\E[ X_1^{k_1}] \E[X_2^{k_2}]\cdots \E[X_n^{k_n}]}{k_1!\,k_2!\,\cdots\, k_n!}  x_1^{k_1}x_2^{k_2}\cdots x_n^{k_n}  \\
&= d!\sum_{k_1+k_2+\cdots+ k_n=d} x_1^{k_1}x_2^{k_2}\cdots x_n^{k_n} \\
&= d! \, h_d(\vec{x})\\[-19pt]
\end{align*}
for any $\vec{x} = (x_1,x_2,\ldots,x_n) \in \R^n$  and nonnegative integer $d$. If $d\geq 0$ is even, then
\vspace{-2pt}
\begin{equation}\label{eq:NewHunter}
 h_d(\vec{x}) = \frac{1}{d!} \E\big[  \inner{ \bvec{X} , \vec{x} }^{d}  \big] \geq 0.
\end{equation}
If the expectation above is zero, then the nonnegativity of $ \inner{ \bvec{X} , \vec{x} } ^{d}$ ensures that  
\begin{align}
x_1X_1+x_2X_2+\cdots+x_nX_n=0 \label{eq:CHSNondegenerate} 
\end{align} 
almost surely. Suppose \eqref{eq:CHSNondegenerate} has a nontrivial solution $\vec{x}$ with nonzero entries $x_{i_1}, x_{i_2}, \ldots, x_{i_k}$. If $k=1$, then $X_{i_k}=0$ almost surely, which contradicts our random variables being exponentially distributed. If $k>1$, then \eqref{eq:CHSNondegenerate} implies
\begin{align}
X_{i_1}=a_{i_2}X_{i_2}+a_{i_3}X_{i_3}+\cdots + a_{i_k}X_{i_K}, \label{eq:CHSNondegenerate2} 
\end{align} almost surely, in which $a_{i_j}=-x_{i_j}/x_{i_1}$. 
Equation \eqref{eq:CHSNondegenerate2} contradicts the assumption of independence, so \eqref{eq:CHSNondegenerate} has no nontrivial solution. Positive definiteness of $h_d$ follows.

\vspace{-5pt}
\subsection{Spline-theoretic proof}\label{Subsection:CHSExtension}
\vspace{-4pt}
The following spline-theoretic approach to the positive definiteness of the CHS polynomials is from \cite{BGON},
although the basic ideas were already present in \cite{GOOY}.
Let $a_1 < a_2 < \cdots < a_n$ be real numbers, let 
\vspace{-2pt}
\begin{equation*}
f[a_1,a_2, \ldots,a_n]=\sum_{j=1}^n \frac{f(a_j)}{\prod_{k\neq j}(a_j-a_k)},\vspace{-2pt}
\end{equation*}
denote the $n$th divided difference of $f$, and let 
\vspace{-1pt}
\begin{equation}\label{eq:PeanoSecond}
F(x;a_1,a_2,\ldots,a_n) = \frac{n-1}{2} \sum_{j=1}^n \frac{|a_j-x|(a_j-x)^{n-3}}{\prod_{k \neq j}(a_j-a_k)}.\vspace{-1pt} 
\end{equation}
denote the corresponding Curry--Schoenberg B-spline, which is a continuous probability density with support $[a_1,a_n]$ \cite{CS}.
A result of Peano ensures that if $f: [a_1,a_n] \to \R$ is $(n-1)$-times continuously differentiable, then
\vspace{-1pt}
\begin{equation}\label{eq:PeanoFirst}
f[a_1, \ldots,a_n]=\frac{1}{(n-1)!}\int_{a_1}^{a_n} f^{(n-1)}(x) F(x;a_1, \ldots,a_n)\,dx, \vspace{-1pt}
\end{equation}
in which $f^{(n-1)}$ denotes the $(n-1)$st derivative of $f$; see \cite{CS} or \cite[Chap. III, Sec. 3.7]{Davis}. 
Consequently, if $f^{(n-1)}$ is nonnegative and not identically zero
on $\R$, then the continuity of $F$ and \eqref{eq:PeanoFirst} imply that $f[a_1,\ldots,a_n] >0$.

Let $p$ be a nonnegative integer and let $f(x)=x^{p+n-1}$ in \eqref{eq:PeanoFirst}.  Then
\vspace{-1pt}
\begin{equation}\label{eq:PeanoThird}
h_p(a_1,a_2,\ldots,a_n)=\binom{p+n-1}{n-1}\int_{\R} x^{p} F(x;a_1, \ldots,a_n)\,dx; \vspace{-1pt}
\end{equation}
see \cite[Thm.~1.2.1]{Phil}. For $p=2,4,6,\ldots$, it follows that $h_p(a_1,a_2,\ldots,a_n) >0$.

\begin{remark}
The moment formula \eqref{eq:PeanoThird} suggests a possible definition of CHS polynomials of nonintegral degree.
This matter, and related Hunter-type lower bounds, are examined in \cite{BGON}. 
\end{remark}


\vspace{-2pt}
\section{Classification and complexifications of norms}
\label{Section:Classification}
\vspace{-2pt}
A norm $\|\cdot\|$ on the space $\M_n$ of $n\times n$ complex matrices is \emph{unitarily invariant} if 
\vspace{-3pt}
\begin{equation*}
\norm{UZV}=\norm{Z} \vspace{-1pt}
\end{equation*}
for all $Z\in \M_n$ and unitary $U,V \in \M_n$. A norm on $\R^n$ which is invariant under entrywise sign changes and permutations is  a \emph{symmetric gauge function}. A theorem of von Neumann asserts that any unitarily invariant norm on $\M_n$ is a symmetric gauge function applied to the singular values \cite[Thm.~7.4.7.2]{HJ}. 
For example, the Schatten norms are unitarily invariant and defined for  $d\geq 1$ by
\vspace{-1pt}
\begin{equation*}
    \norm{A}_{S_d}=\big( |\sigma_1|^d+|\sigma_2|^d+\cdots+ |\sigma_n|^d\big)^{\smash{1/d}}, \vspace{-1pt}
\end{equation*} 
in which $\sigma_1 \geq \sigma_2 \geq  \cdots \geq \sigma_n \geq 0$ are the singular values of $A\in \M_n$. 
\smallskip

A norm $\norm{\cdot}$ on the space $\H_n$ of $n\times n$ complex Hermitian matrices is \emph{weakly unitarily invariant} if 
$\norm{U^*AU}=\norm{A}$ for all $A\in \H_n$ and unitary $U \in \M_n$.
For example, the numerical radius is weakly unitarily invariant on $\H_n$ \cite{LiUnitary}.  Lewis proved that any weakly unitarily invariant norm on $\H_n$ is a symmetric norm on $\R^n$ evaluated on the eigenvalues of a Hermitian matrix \cite[Sec.~8]{LewisGroup}. A proof of this result is given in Subsection \ref{Subsection:Classify}. In Subsection \ref{Section:Complexification}, a method is developed for extending norms on $\H_n$ to the whole space $\M_n$.


\vspace{-6pt}
\subsection{Lewis' classification theorem}\label{Subsection:Classify}
\vspace{-5pt}

In this section we give a simple proof of Lewis' classification of weakly unitarily invariant norms on $\H_n$. The proof of Theorem \ref{Theorem:Classify} below avoids Lewis' theory of group invariance in convex matrix analysis \cite{LewisGroup}, the framework that underpinned \cite{Aguilar,Ours}.  The approach we take here comes from \cite{CGH}, which uses more standard techniques, such as Birkhoff's theorem on doubly stochastic matrices \cite{Birkhoff}.

\begin{theorem}[Lewis]\label{Theorem:Classify}
    A norm $\norm{\cdot}$ on $\H_n$ is  weakly unitarily invariant if and only if
    there is a symmetric norm $f:\R^n\to\R$ such that $\norm{A}=f( \lambda_1, \lambda_2, \ldots, \lambda_n)$.
\end{theorem}

The proof of Theorem \ref{Theorem:Classify} given here first appeared in \cite{CGH}. It follows immediately from Proposition \ref{Proposition:OnlyIf} \cite[Prop.~5]{CGH} and Proposition \ref{Proposition:RntoHermitian} \cite[Prop.~9]{CGH} below.

\begin{proposition}\label{Proposition:OnlyIf}
If $\norm{\cdot}$ is a weakly unitarily invariant norm on $\H_n$, then there is a symmetric 
norm $f$  on $\R^n$ such that $\norm{A}=f(  \boldsymbol{\lambda}(A))$ for all $A\in \H_n$.
\end{proposition}
        
\begin{proof}
    Weak unitary invariance implies that $\norm{A}$ is a symmetric function in the eigenvalues of $A$. 
    In particular, $\norm{A}=f(  \boldsymbol{\lambda}(A) )$ for some symmetric function $f$.  The diagonal matrix $\diag(\vec{a})$ is Hermitian for any  $\vec{a}=( a_1, a_2,\dots, a_n)\in \R^n$ and  $\bvec{\lambda}(\diag(\vec{a})) = P\vec{a}$ for some permutation matrix $P$. Symmetry of $f$ implies
    \begin{equation*}
        f(\vec{a}) =f(P\vec{a}) = f\big(\bvec{\lambda}(\diag{\vec{a}})\big) = \norm{\diag{\vec{a}}}.
    \end{equation*} 
    Consequently, $f$ must inherit the properties of a norm on $\R^n$.
\end{proof}

To establish Proposition \ref{Proposition:RntoHermitian}, the following lemma is first needed.

\begin{lemma}\label{Lemma:EigenPerms}
If $A, B\in \H_n$, then there exist permutation matrices $P_1,P_2,\ldots,P_{n^2}$ and 
nonnegative constants $c_1,c_2,\ldots,c_{n^2}$  satisfying $\sum_{i = 1}^{n^2} c_i = 1$ for which
\vspace{-5pt}
\begin{equation*}
\bvec{\lambda}(A+B) = \sum_{i = 1}^{\,n^2} c_{i} P_i(\bvec{\lambda}(A) + \bvec{\lambda}(B)).\vspace{-2pt}
\end{equation*}
\end{lemma}

\begin{proof}
The Ky Fan eigenvalue inequality \cite{KyFan} ensures that
\vspace{-5pt}
\begin{equation}\label{eq:KyFan}
    \sum_{i=1}^{k} \lambda_i(A+B) \leq \sum_{i=1}^{k} \big(\lambda_i(A) + \lambda_i(B)\big)\vspace{-5pt}
\end{equation} 
for all $1\leq k\leq n$. The sum of the eigenvalues of a matrix is its trace. Thus,
\vspace{-4pt}
\begin{equation*} 
\sum_{i=1}^{n} \lambda_i(A+B) = \tr (A+B) = \tr  A + \tr  B = \sum_{i=1}^{n} \big(\lambda_i(A) + \lambda_i(B)\big),\vspace{-4pt}
\end{equation*} 
so equality holds in \eqref{eq:KyFan} for $k=n$. Thus, $\bvec{\lambda}(A+B) \!\prec \bvec{\lambda}(A) + \bvec{\lambda}(B)$ and
Theorem~\ref{Lemma:MajorStoch} provides a doubly stochastic matrix $D$ such that
$\bvec{\lambda}(A+B) = D(\bvec{\lambda}(A) + \bvec{\lambda}(B))$. Finally,
Theorem \ref{Theorem:Birkhoff} provides the desired permutation matrices and coefficients.
\end{proof}

\begin{proposition}\label{Proposition:RntoHermitian}
    If $f$ is a symmetric norm on $\R^n$, then 
    $\norm{A}=f(\boldsymbol{\lambda}(A))$ defines a weakly unitarily invariant norm on $\H_n$.
\end{proposition}

\begin{proof}
    Since $\norm{A}=f(\boldsymbol{\lambda}(A))$ is a symmetric function of the eigenvalues of $A$, it is
    weakly unitarily invariant. We must show that $\norm{\cdot}$ defines a norm on $\H_n$. 
    
    \medskip\noindent\textsc{Positive Definiteness}. First note that a Hermitian matrix $A$ is zero
    if and only if $\bvec{\lambda}(A) = 0$. Therefore, 
    the positive definiteness of $f$ on $\R^n$ implies the positive definiteness of $\norm{\cdot}$. 
    
    \medskip\noindent\textsc{Absolute Homogeneity}. If $c\geq 0$, then $\bvec{\lambda}(cA) = c\bvec{\lambda}(A)$. If $c<0$, then
    \vspace{-2pt}
    \begin{equation*}
        \bvec{\lambda}(cA) = c
        \begin{bmatrix}
        && 1  \\[-8pt]
        & \iddots & \\[-3pt]
        1 & &
        \end{bmatrix}
        \bvec{\lambda}(A).\vspace{-2pt}
    \end{equation*} 
    The homogeneity and symmetry of $f$ therefore imply that 
    \vspace{-2pt}
    \begin{equation*}
        \norm{cA} = f\big(\bvec{\lambda}(cA)\big) = f\big(c\bvec{\lambda}(A)\big) = |c| f\big(\bvec{\lambda}(A)\big) = |c|\norm{A}.\vspace{-2pt}
    \end{equation*} 
    
    \medskip\noindent\textsc{Triangle Inequality}. 
    Suppose  that $A,B \in \H_n$. Lemma \ref{Lemma:EigenPerms} ensures that
    \vspace{-3pt}
    \begin{equation*}
        \norm{A + B} = f\big(\bvec{\lambda}(A+B)\big) = f\bigg(\sum_{i = 1}^{\vphantom{n}\smash{\,n^2}} c_{i} P_i\big(\bvec{\lambda}(A) + \bvec{\lambda}(B)\big)\!\bigg). \vspace{-3pt}
    \end{equation*} 
    The triangle inequality and homogeneity of $f$ on $\R^n$ therefore imply that
    \vspace{-6pt}
    \begin{equation}\label{eq:Triangle}
        \norm{A+B}\leq \sum_{i=1}^{n^2} c_{i} f\big( P_i(\bvec{\lambda}(A) + \bvec{\lambda}(B))\big).\vspace{-8pt}
    \end{equation} 
 The norm $f$ is invariant under permutations  and $\sum_{i = 1}^{n^2} c_i = 1$. Therefore,
 \vspace{-5pt}
    \begin{equation*}
        \sum_{i=1}^{n^2} c_{i} f\big( P_i(\bvec{\lambda}(A) + \bvec{\lambda}(B))\big)
        = \sum_{i=1}^{n^2} c_{i} f\big(\bvec{\lambda}(A) + \bvec{\lambda}(B)\big)= f\big(\bvec{\lambda}(A) + \bvec{\lambda}(B)\big).\vspace{-4pt}
    \end{equation*} 
    Consequently, the triangle inequality for $f$ on $\R^n$ and \eqref{eq:Triangle} imply that
    \begin{equation*}
        \norm{A+B}
        \leq  f\big(\bvec{\lambda}(A) + \bvec{\lambda}(B)\big)
        \leq  f\big(\bvec{\lambda}(A) \big)+ f\big(\bvec{\lambda}(B)\big)
        =\norm{A}+\norm{B}. \qedhere
    \end{equation*}
\end{proof}

\begin{remark}
Theorem \ref{Theorem:Classify} concerns weakly unitarily invariant norms on $\H_n$.
	A norm $\norm{\cdot}$ on $\M_n$, the space of all complex $n \times n$ matrices, is \emph{weakly unitarily invariant} if $\norm{A}=\norm{U^*AU}$ for all $A\in \M_n$ and unitary $U \in \M_n$.  A norm $\Phi$ on the space $C(S)$ of continuous functions on the unit sphere $S\subset \C^n$ is a \emph{unitarily invariant function norm} if  $\Phi(f\circ U)=\Phi(f)$ for all $f\in C(S)$ and unitary $U \in \M_n$. 
	It can be shown that every weakly unitarily invariant norm $\norm{\cdot}$ on $\M_n$ is of the form
	$\norm{A}=\Phi(f_A)$, where $f_A\in C(S)$ is defined by $f_A(\vec{x})=\langle A\vec{x},\vec{x}\rangle$ and $\Phi$ is a unitarily invariant function norm \cite[Thm.~2.1]{Bhatia}, \cite{BhatiaHolbrook}.
\end{remark}


\vspace{-6pt}
\subsection{Complexification of norms}\label{Section:Complexification}
\vspace{-4pt}

In this section we develop machinery for extending norms on $\H_n$ to $\M_n$. Let $\V$ be a complex vector space with a conjugate-linear involution $v\mapsto v^*$. Suppose that the $\R$-linear subspace 
\begin{equation*}
\V_\R=\{v\in \V:v=v^*\}
\end{equation*}
of $*$-fixed points is endowed with a norm $\| \cdot\|$. For each $v \in \V$ and $t \in \R$, we have $e^{it}v+e^{-it}v^* \in \V_{\R}$. 
Lastly, we note that the path $t\mapsto \|e^{it}v+e^{-it}v^*\|$ is continuous for each $v\in \V$. The following is \cite[Prop.~15]{Aguilar}.

\begin{proposition}\label{p:gennorm}
Suppose $d\geq 1$. The following is a norm on $\V$ that extends $\|\cdot\|$:
\vspace{-4pt}
\begin{equation}\label{eq:gennorm}
\cnorm{v}_d= \bigg( \frac{1}{2\pi \binom{d}{d/2}}\int_0^{2\pi}\|e^{it}v+e^{-it}v^*\|^d\,\mathrm{d}t \bigg)^{\!\!1/d}.
\end{equation}
\end{proposition}

\begin{proof}
If $v\in \V_\R$, then $\|e^{it}v+e^{-it}v^*\|=|2\cos t|\|v\|$ and $\cnorm{v}_d=\|v\|$ since
\vspace{-2pt}
\begin{equation*}
\int_0^{2\pi} |2\cos t|^d\,\mathrm{d}t = 2\pi \binom{d}{d/2},\vspace{-7pt}
\end{equation*}
in which $\binom{d}{d/2}$ is interpreted in terms of the gamma function. We now verify that $\cnorm{\cdot}_d$ is a norm on $\V$.

\medskip\noindent\textsc{Positive definiteness}. 
The nonnegativity of $\|\cdot\|$ on $\V_{\R}$ and \eqref{eq:gennorm} ensure that
$\cnorm{\cdot}_d$ is nonnegative on $\V$.
If $v\in \V\backslash\{0\}$, then $v = u + iu'$, in which $u = \frac{1}{2}(v + v^*)$
and $u' = \frac{1}{2}(-iv + iv^*)$ belong to $\V_{\R}$.  Not both of $u$ and $u'$ can be zero, so 
\vspace{-4pt}
\begin{equation*}
t\mapsto \|e^{it}v+e^{-it}v^*\| = \| 2 \cos(t)u + 2 \sin(t) u'\|\vspace{-2pt}
\end{equation*}
is continuous and positive almost everywhere.
Therefore, $\cnorm{v}_d\neq0$.

\medskip\noindent\textsc{Absolute homogeneity}. 
For $r>0$ and $\theta\in\R$, the $\R$-homogeneity of $\|\cdot\|$ and the $2\pi$-periodicity of the integrand in \eqref{eq:gennorm} yield $\cnorm{( re^{i\theta})v }_d = r \cnorm{e^{i\theta}v}  =r\cnorm{v}_d$.

\medskip\noindent\textsc{Triangle inequality}. For $u,v\in \V$, we have
\begin{align*}
&\bigg(\!\int_0^{2\pi}\|e^{it}(u+v)+e^{-it}(u+v)^*\|^d\,\mathrm{d}t\bigg)^{\!\!\smash{1/d}} \\[-2pt]
&\qquad \leq 
\bigg(\!\int_0^{2\pi}\big(\|e^{it}u+e^{-it}u^*\|+\|e^{it}v+e^{-it}v^*\|\big)^d\,\mathrm{d}t\bigg)^{\!\!\smash{1/d}} \\
&\qquad \leq \bigg(\!\int_0^{2\pi}\|e^{it}u+e^{-it}u^*\|^d\,\mathrm{d}t\bigg)^{\!\!\smash{1/d}}
+\bigg(\!\int_0^{2\pi}\|e^{it}v+e^{-it}v^*\|^d\,\mathrm{d}t\bigg)^{\!\!\smash{1/d}}.\\[-22pt]
\end{align*}	
The first inequality follows from the monotonicity of power functions and the triangle inequality for $\|\cdot\|$.
The second inequality holds by the triangle inequality for the $L^d$ norm on the space $C[0,2\pi]$.
\end{proof}


\section{Properties of random vector norms}\label{Section:Main}

A \emph{random vector} in $\R^n$ is a tuple $\X=(X_1, X_2, \ldots, X_n)$, in which 
$X_1, X_2, \ldots, X_n$ are real-valued random variables on a common probability space $(\Omega,\F,\P)$; we assume $\Omega \subset \R$. If $X_1, X_2, \ldots, X_n$ are iid, then we call $\X$ an \emph{iid random vector}.
Let $\X$ be an iid random vector in $\R^n$ and define $f_{\X,d}:\R^n\to \R$ by
\begin{equation}\label{eq:norm}
	f_{\X,d}(\boldsymbol{\lambda}) = \big( \E |\langle\X, \bvec{\lambda}\rangle|^d\big)^{\!\!1/d}
\end{equation}  
for $d\geq 1$. Here, $\Gamma(\cdot)$ denotes the gamma function. Furthermore, if $A\in \H_n$ and $\boldsymbol{\lambda}=(\lambda_1,\lambda_2, \ldots, \lambda_n)$ denotes the vector of eigenvalues 
\begin{equation*}
\lambda_1 \geq \lambda_2 \geq \cdots \geq \lambda_n
\end{equation*}
of $A$, then the function
\begin{equation}\label{eq:NormHermitian}
	\norm{A}_{\X,d} = f_{\X,d}(\boldsymbol{\lambda}) =
	\big( \E |\langle\X, \bvec{\lambda}\rangle|^d \big)^{\!\!\smash{1/d}}
\end{equation} 
on $\H_n$ defines a \emph{random vector norm} on $\H_n$. That this is indeed a norm is proved in 
Theorem \ref{Theorem:Main2} below.

\begin{remark}\label{Rem:redef}
The sources \cite{CGH,Ours} define $\norm{A}_{\X,d}^d= \E |\langle\X, \bvec{\lambda}\rangle|^d / \Gamma(d+1)$. 
This normalization simplifies some of the formulas in Section \ref{Sec:CMS}, although we have come to the conclusion that \eqref{eq:NormHermitian} is the more natural definition. However, the new definition makes it natural to redefine the constant $y_{\bvec{\pi}}$; see Subsection \ref{Subsection:Partitions}.
\end{remark}

The results of Subsection \ref{Section:Complexification} permit us to extend $\norm{\cdot}_{\X,d}$ to $\M_n$ via 
\begin{equation}\label{eq:NormComplex}
	\cnorm{Z}_{\vec{X},d}
	=  \left(\frac{1}{2\pi \binom{d}{d/2}} \int_0^{2 \pi} \norm{ e^{it} Z + e^{-it}Z^*}_{\bvec{X},d}^d \,\mathrm{d}t \right)^{\!\!\smash{1/d}}.
\end{equation}
We refer to the above as a \emph{random vector norm} on $\M_n$. 
Assuming that \eqref{eq:NormHermitian} is a norm, Proposition \ref{p:gennorm} ensures that $\eqref{eq:NormComplex}$ is a norm on $\M_n$ which restricts to $\norm{\cdot}_{\X,d}$, as defined in \eqref{eq:NormHermitian},  on $\H_n$. For clarity, we write $\cnorm{\cdot}_{\X,d}$ when we refer to this extension.

The next theorem states that random vector norms are weakly unitarily invariant norms on $\H_n$ that extend to weakly unitarily invariant norms on $\M_n$, 
and provides probabilistic and combinatorial interpretations of these norms.
The following is \cite[Thm.~3]{CGH} and is proved in the Subsections \ref{subsec:Proof1}--\ref{Subsection:ProofE}.

\begin{theorem}\label{Theorem:Main2}
	Let $d\geq 1$, let $\X=(X_1, X_2, \ldots, X_n)$ be an iid random vector in which the entries $X_1, X_2, \ldots, X_n \in L^d(\Omega,\F,\P)$ are nondegenerate, and let $\boldsymbol{\lambda}=(\lambda_1,\lambda_2, \ldots, \lambda_n)$ be the vector of eigenvalues 
	$\lambda_1 \hspace{-.5pt}\geq\hspace{-.5pt} \lambda_2 \hspace{-.5pt}\geq\hspace{-.5pt} \cdots \hspace{-.5pt}\geq\hspace{-.5pt} \lambda_n$ of $A$. 
	\begin{enumerate}[label=(\alph*)]
		\itemsep1em
		\item $\norm{A}_{\X,d}= \big( \E |\langle \X, \LL\rangle|^d \big)^{\!1/d}$ is a weakly unitarily invariant norm on $\H_n$.
		
		\item If the first $m$ moments of $X_i$ exist, then the function $f:[1,m] \to \R$ defined by 
		$f(d) =\norm{A}_{\X,d}$ is continuous for all $A\in\H_n$.
		
		\item If $2 \leq p \leq q$, then  \vspace{-4pt}
		\begin{align*}
			\norm{A}_{\X,p} &\leq \norm{A}_{\X,q} && \text{for $A\in \H_n$, and} \\[3pt]
			\tbinom{p}{p/2}^{1/p} \cnorm{Z}_{\X,p} &\leq \tbinom{q}{q/2}^{1/q} \cnorm{Z}_{\X,q} && \text{for $Z\in\M_n$}.
		\end{align*} 
		The inequalities are reversed if $1\leq p\leq q \leq 2$.
		
		\item The function $\bvec{\lambda}(A) \mapsto \norm{A}_{\X,d}$ is Schur-convex.
		
		\item If  $d \geq 2$ is even and the $X_i$ admit a moment generating function $M(t)$, then
		\vspace{-2pt}
		\begin{equation}\label{eq:MainMGF}
			\norm{A}_{\X,d}^d = d! \cdot [t^d] M_{\Lambda}(t)
			\quad \text{for $A \in \H_n$},
		\end{equation}
		in which $M_{\Lambda}(t) = \prod_{i=1}^n M(\lambda_i t)$ is the moment generating
		function for the random variable 
		$\Lambda =\langle \X, \LL(A) \rangle$.
		In particular, $\norm{A}_{\X,d}$ is a positive definite, homogeneous,
		symmetric polynomial in the eigenvalues of $A$.
		
		\item If $d\geq 2$ is even and the first $d$ moments of $X_i$ exist, then
		\vspace{-2pt}
		\begin{align}
			\norm{A}_{\X,d}^d
			&= B_{d}(\kappa_1\tr A, \kappa_2\tr A^2, \ldots, \kappa_d\tr A^d) \label{eq:MainBell} \\
			&= \sum_{\bvec{\pi}\vdash d} y_{\bvec{\pi}} \kappa_{\bvec{\pi}} p_{\bvec{\pi}} (\bvec{\lambda})
			\quad \text{for $A \in \H_n$}, \label{eq:RealPermForm}
		\end{align}
		in which $B_d$ is given by \eqref{eq:ExpoBell}, $\kappa_i$ is defined in \eqref{eq:CumulantGen}, $\kappa_{\bvec{\pi}}$ and $y_{\bvec{\pi}}$ are given in \eqref{eq:DefY},
		$p_{\bvec{\pi}} (\bvec{\lambda})$ is defined in \eqref{eq:pTrace},
		and the second sum runs over all partitions $\bvec{\pi}$ of $d$. 
		
		\item If $d\geq 2$ is even, let $\bvec{\pi}=(\pi_1, \pi_2, \ldots,\pi_r)$ be a partition of $d$. 
		Let $\TT_{\bvec{\pi}}(Z)$ be the arithmetic mean of the $\binom{d}{d/2}$ terms obtained by placing $d/2$ adjoints ${}^*$ among the $d$ copies of $Z$ in
		$\vphantom{{\displaystyle1_2}}\smash{(\tr \underbrace{ZZ\cdots Z}_{\pi_1})
			(\tr \underbrace{ZZ\cdots Z}_{\pi_2})
			\cdots
			(\tr \underbrace{ZZ\cdots Z}_{\pi_r})}$. Then $\cnorm{Z}_{\X,d}$, defined in \eqref{eq:NormComplex}, satisfies
		\begin{equation}\label{eq:ExtendedGeneral}
			\cnorm{Z}_{\X,d}^d=  \sum_{\bvec{\pi} \,\vdash\, d}  y_{\bvec{\pi}} \kappa_{\bvec{\pi}} \TT_{\bvec{\pi}}(Z)
			\quad \text{for $Z \in \M_n$},
		\end{equation} 
		where the sum runs over all partitions $\bvec{\pi}$ of $d$. In particular, $\cnorm{Z}_{\X,d}$ is a weakly unitarily invariant norm on $\M_n$ that restricts to $\|Z\|_{\X,d}$ on $\H_n$ which is a positive definite trace polynomial in $Z$ and $Z^*$.
	\end{enumerate}
\end{theorem}

\begin{example}\label{Ex:d=2}
	A more precise definition of $\TT_{\bvec{\pi}}(Z)$ appears in Subsection \ref{Subsection:ProofE},
	although the examples in Section \ref{Section:Examples} better illustrate how to compute \eqref{eq:ExtendedGeneral}. For now, let $d=2$. 
	If $\mu$ and $\sigma^2$ denote the mean and variance of $X_i$,  respectively, then
	\vspace{-4pt}
	\begin{align} 
		\cnorm{Z}_{\X,2}^2 &=   \sum_{\bvec{\pi} \,\vdash\, 2}  y_{\bvec{\pi}} \kappa_{\bvec{\pi}} \TT_{\bvec{\pi}}(Z) =  y_{(2)}\kappa_{(2)}  \TT_{(2)}(Z) +  y_{(1,1)}\kappa_{(1,1)} \TT_{(1,1)}(Z) \nonumber\\ 
		&= \frac{ \kappa_{2} \big(\!\tr(ZZ^*)\!+\tr(Z^*Z) \big)}{2!^1 1!} + \frac{ \kappa_{1}^2 \big(\!\tr(Z)\tr(Z^*)\!+\tr(Z^*)\tr(Z) \big)}{1!^2 2!} \nonumber\\
		&= \label{eq:d=2}  \sigma^2 \norm{Z}_{\operatorname{F}}^2+\mu^2 |\tr Z|^2,
	\end{align}
	where the last equality holds because $\kappa_1=\mu$ and $\kappa_2=\sigma^2$. If $\mu = 0$, then $\cnorm{\cdot}_{\X,2}$
	is a nonzero multiple of the Frobenius norm since the variance $\sigma^2$ 
	is positive by nondegeneracy. As predicted by Theorem \ref{Theorem:Main2}, 
	the norm \eqref{eq:NormComplex} on $\M_n$ reduces to \eqref{eq:MainBell} on $\H_n$ since $B_2(x_1,x_2)=x_1^2+x_2$ and
	\begin{align*}
		\norm{A}_{\X,2}^2
		&=B_2( \kappa_1 \tr A, \kappa_2 \tr A^2)
		= (\kappa_1 \tr A)^2 + \kappa_2 \tr (A^2) \\
		&= \sigma^2 \tr(A^2) + \mu^2 (\tr A)^2, 
	\end{align*}
	which agrees with \eqref{eq:d=2} if $Z = A = A^*$.
\end{example}

\begin{remark}
The original proof that ${\norm{\cdot}}_{\X,d}$ is a weakly unitarily invariant norm on $\H_n$ \cite[Thm.~1.1.(a)]{Ours} requires $d \geq 2$ and uses Lewis' framework for group invariance in convex matrix analysis \cite{LewisGroup}. However, Theorem \ref{Theorem:Main2}.(a) now follows directly from Theorem \ref{Theorem:Classify} and holds for $d\geq 1$.
\end{remark}

\begin{remark}
	Proving that $\cnorm{\cdot}_{\X,d}$ is a norm on $\M_n$ relies on Theorem \ref{Theorem:Main2}.(a), which states that its restriction to $\H_n$ is a norm. 
	On the other hand, demonstrating that \eqref{eq:ExtendedGeneral} is a norm directly seems arduous. 
	To a certain degree, this mirrors the current absence of general certificates for dimension-independent positivity of trace polynomials in $x,x^*$. See \cite{KSV} for the analysis in a dimension-fixed setting.
\end{remark}

\begin{remark}
The positivity of random vector norms is not obvious since
they involve the eigenvalues of $A \in \H_n$ and not their absolute values. 
Consequently, these norms on $\H_n$ do not arise from symmetric gauge functions \cite[Sect.~7.4.7]{HJ}. 
In fact, random vector norms can sometimes distinguish singularly (adjacency) cospectral graphs
(graphs with the same singular values) that are not adjacency cospectral. This feature is not enjoyed by many standard norms
	(for example, the operator, Frobenius, Schatten--von Neumann, and Ky Fan norms). The matrix
	\vspace{-1pt}
	\begin{equation*}
		K = 
		\begin{bmatrix}
			0 & 1 & 1\\[-2pt]
			1 & 0 & 1\\[-2pt]
			1 & 1 & 0
		\end{bmatrix},\vspace{-1pt}
	\end{equation*}
	has eigenvalues $2,-1,-1$ and is the adjacency matrix for the complete
	graph on three vertices.
	The graphs with adjacency matrices $A=\left[\begin{smallmatrix}
		K&0\\0&K
	\end{smallmatrix} \right]$ and $B=\left[\begin{smallmatrix}
	0&K\\K&0
	\end{smallmatrix} \right]$
	are singularly cospectral but not cospectral: their eigenvalues are $-1, -1, -1, -1, 2, 2$ and 
	$-2,-1,-1,1,1,2$, respectively.  
	Moreover, $\cnorm{A}_{\X,6}^6=1350 \neq 1260 = \cnorm{B}_{\X,6}^6$ when 
	the $X_i$ follow a Gamma distribution with $\alpha=1$ and $\beta=1/2$.
\end{remark}

\subsection{Proof of Theorem \ref{Theorem:Main2}.(a)}\label{subsec:Proof1}

The function $f_{\X,d}$ defined in \eqref{eq:norm} is symmetric because the entries of $\X$ are iid. In light of Theorem \ref{Theorem:Classify}, it suffices to show that $f_{\X,d}$ is a norm on $\R^n$ to establish the claim. 
We remark that in \cite[Sec.~3.1]{Ours}, the proofs for absolute homogeneity and the triangle inequality are valid for $d\geq 1$. 
However, the proof for positive definiteness there requires $d\geq 2$. The proof of positive definiteness below first appeared in \cite{CGH} and holds for $d\geq 1$.

\begin{proposition}\label{Proposition:Main1}
The function $f_{\X,d}$ in \eqref{eq:norm} defines a norm on $\R^n$ for all $d\geq 1$.
\end{proposition}

\begin{proof} 

\medskip\noindent\textsc{Positive Definiteness}.  
If $f_{\X,d}(\boldsymbol{\lambda})=0$, then $\E|\langle \X,\boldsymbol{\lambda}\rangle|^d=0$. 
The nonnegativity of $|\langle \X,\boldsymbol{\lambda}\rangle|^d$ ensures that
\begin{equation}\label{eq:Independence}
    \lambda_1X_1+\lambda_2X_2+\cdots+\lambda_nX_n=0 
\end{equation} 
almost surely. Assume that \eqref{eq:Independence} has a nontrivial solution $\boldsymbol{\lambda}$ with nonzero entries $\lambda_{i_1}, \lambda_{i_2}, \ldots, \lambda_{i_k}$. If $k=1$, then $X_{i_k}=0$ almost surely, which contradicts the nondegeneracy of our random variables. If $k>1$, then \eqref{eq:Independence} implies that
\begin{equation}\label{eq:Contradiction}
    X_{i_1}=a_{i_2}X_{i_2}+a_{i_3}X_{i_3}+\cdots +a_{i_k}X_{i_k}
\end{equation}
almost surely, in which $a_{i_j}=-\lambda_{i_j}/\lambda_{i_1}$.  
The independence of $X_{i_1}, X_{i_2}, \ldots, X_{i_k}$ ensures the independence of 
$X_{i_1}, a_{i_2}X_{i_2}, a_{i_3}X_{i_3},\ldots, a_{i_k}X_{i_k}$,
which contradicts \eqref{eq:Contradiction}. Relation \eqref{eq:Independence} therefore has no nontrivial solutions.

\medskip\noindent\textsc{Absolute Homogeneity}. 
The bilinearity of the inner product and the linearity of expectation imply
\begin{equation*}
f_{\X,d}(c\boldsymbol{\lambda})
=\left( \E |c\langle\X, \bvec{\lambda}\rangle|^d\right)^{\smash{\!1/d}}
\!\!=\left( |c|^d\E |\langle\X, \bvec{\lambda}\rangle|^d\right)^{\smash{\!1/d}}
\!\!=|c|f_{\X,d}(\boldsymbol{\lambda}).
\end{equation*}

\smallskip\noindent\textsc{Triangle Inequality}. 
For $\boldsymbol{\lambda}, \boldsymbol{\mu}\in \R^n$, define the random variables 
$X=\langle \X,\boldsymbol{\lambda}\rangle$ and $Y=\langle \X,\boldsymbol{\mu}\rangle$.
Minkowski's inequality ensures that
\begin{equation*}
\big(\E\vert \langle \X, \boldsymbol{\lambda}+\boldsymbol{\mu}\rangle\vert^d\big)^{1/d}
=\big(\E|X+Y|^d\big)^{1/d} \leq \big(\E|X|^d\big)^{1/d}+\big(\E|Y|^d\big)^{1/d}. 
\end{equation*} 
The triangle inequality for $f_{\X,d}$ follows.
\end{proof}

\begin{remark}\label{Remark:Conjugation}
Remark 3.4 of \cite{Ours} is somewhat misleading. There it is stated that the entries of $\X$ must be identically distributed but not necessarily independent. To clarify, the entries of $\X$ being identically distributed guarantees that $\norm{\cdot}_{\X,d}$ satisfies the triangle inequality on $\H_n$. The additional assumption of independence guarantees that $\norm{\cdot}_{\X,d}$ is also positive definite.
\end{remark}

\subsection{Proof of Theorem \ref{Theorem:Main2}.(b)}
\vspace{-2pt}

As usual, let $\boldsymbol{\lambda}$ denote the vector of eigenvalues of $A$. Let $\mu_{Y}$  denote the pushforward measure of the random variable $Y = \langle\X, \bvec{\lambda}\rangle$, so that
\vspace{-1pt}
\begin{equation*}
    \big(f(d)\big)^d = \E|\langle\X, \bvec{\lambda}\rangle|^d  = \int |x|^d \, d\mu_{Y}.\vspace{-1pt}
\end{equation*} 
The bound $|x|^d \leq |x| + |x|^m$ holds for all $x\in \R$ and $1 \leq d \leq m$. Therefore, 
\vspace{-1pt}
\begin{equation*}
    \int |x|^d \, d\mu_{Y} \leq  \int |x|  \, d\mu_Y + \int |x|^m \, d\mu_Y.\vspace{-1pt}
\end{equation*} 
Consequently, we may bound $\big(f(d)\big)^d$ by the $1$- and $m$-norms:
\vspace{-1pt}
\begin{equation*}
\big(f(d)\big)^d \leq \norm{A}_{\X,1}+ \norm{A}_{\X,m}^m.
\end{equation*} 
If $d_i\to d$, then $\int |x|^{d_i}d\mu_Y\to \int |x|^{d}d\mu_Y$ by the dominated convergence theorem. 
The function $d\mapsto f^d$ is therefore continuous. The
continuity of the gamma function establishes continuity for $f^d$ and therefore of $f$. $\hfill \qed$

\vspace{-2pt}
\subsection{Proof of Theorem \ref{Theorem:Main2}.(c)}
\vspace{-2pt}

Let $A \in \H_n$ have eigenvalues $\bvec{\lambda} = (\lambda_1,\ldots,\lambda_n)$, listed in decreasing order, and let $\bvec{X} = (X_1,\ldots,X_n)$ be an iid random vector. 
Lyapunov's inequality \eqref{eq:Lyapunov} ensures that
\vspace{-1pt}
\begin{equation}\label{eq:Probable}
	\norm{A}_{\X,p} = \E\big[ | \inner{ \bvec{X} , \bvec{\lambda} }|^{p}  \big]^{\!\frac{1}{p}}  \!\leq \E\big[ | \inner{ \bvec{X} , \bvec{\lambda} }|^{q}
	 \big]^{\!\frac{1}{q}} \!=  \norm{A}_{\X,q}.
\end{equation}
Now, let $Z \in \M_n$. Consider $f(t) = \norm{e^{it} Z + e^{-it} Z^*}_{\X,q}$ as an element of $L^p[0,2\pi]$.
H\"older's inequality and \eqref{eq:NormComplex} imply the desired inequality:
\vspace{-1pt}
\begin{align*}
	\tbinom{p}{p/2}^{\!1/p}\cnorm{Z}_{\X,p} 
	&= \bigg(\frac{1}{2\pi }\int_0^{2\pi} \norm{ e^{it} A+e^{-it} A^*}_{\X,p}^p\,\mathrm{d}t\bigg)^{\!\!1/p} \\[-4pt]
	&\leq \bigg(\frac{1}{2\pi }\int_0^{2\pi} \norm{ e^{it} A+e^{-it} A^*}_{\X,q}^p\,\mathrm{d}t\bigg)^{\!\!1/p} =  \norm{f}_{L^p}  \\[-4pt]
	&\leq  \norm{f}_{L^q} = \bigg(\frac{1}{2\pi }\int_0^{2\pi} \norm{ e^{it} A+e^{-it} A^*}_{\X,q}^q\,\mathrm{d}t\bigg)^{\!\!1/q}\\
	&= \tbinom{q}{q/2}^{\!1/q}\cnorm{Z}_{\X,q}. 
\end{align*}
The inequalities are reversed if $1\leq p \leq q \leq 2$. $\hfill \qed$

\begin{remark}
	The previous result suggests that suitable constant multiples of the CHS norms may be preferable in
	some circumstances.  However, the benefits appear to be outweighed by the cumbersome
	nature of these constants.
\end{remark}


\subsection{Proof of Theorem \ref{Theorem:Main2}.(d)}
Let $\X$ be a random vector in $\R^n$ whose entries lie in $L^d(\Omega, \mathcal{F}, \P)$ and are identically distributed, and let $f_{\X,d}(\vec{x})$ be defined as in \eqref{eq:norm}. Given vectors $\bvec{x},\bvec{y}\in \R^n$, $f$ satisfies $f(\vec{x}+\vec{y})\leq f(\bvec{x}) + f(\bvec{y})$, as seen in the proof of Theorem \ref{Theorem:Main2}.(a). Homogeneity implies that $f$ is convex on $\R^n$. Finally, $f$ is symmetric since the entries of $\X$ are identically distributed. Recall that a convex function $f:\R^n\to \R$ is  Schur-convex if and only if it is symmetric \cite[p.~258]{Roberts}. It follows that $f$ is Schur-convex. Therefore,
$\bvec{\lambda}(A)\mapsto f(\lambda_1, \lambda_2, \ldots, \lambda_n)=\norm{A}_{\X,d}$ is Schur-convex. $\hfill \qed$

\begin{remark}\label{Remark:SchurNot}
	Note that independence is not required in the previous argument.
\end{remark}

\subsection{Proof of Theorem \ref{Theorem:Main2}.(e)}
Let $d \geq 2$ be even and let $\X=(X_1, X_2, \ldots, X_n)$ be an iid random vector whose entries  admit a moment generating function $M(t)$.     Let $A \in \H_n$ have 
eigenvalues $\lambda_1 \geq \lambda_2 \geq \cdots \geq \lambda_n$.
If $\Lambda =\langle \X, \LL \rangle=\lambda_1X_1+\lambda_2X_2+\cdots +\lambda_n X_n$, then independence ensures that
$M_{\Lambda}(t) = \prod_{i=1}^n M(\lambda_i t)$. Therefore,
\begin{equation*}
 \norm{A}_{\X,d}^d=\E [\Lambda^d] = d! \cdot [t^d] M_{\Lambda}(t) \tag*{\qed}
\end{equation*}

\subsection{Proof of Theorem \ref{Theorem:Main2}.(f)}
Maintain the same notation as above.  However, we only assume existence of the first $d$ moments of the entries of $\X$. In this case,  $M_{\Lambda}(t)$ is a formal series with $\kappa_1,\kappa_2,\ldots,\kappa_d$ determined by \eqref{eq:CumulantMoment} and the remaining
cumulants treated as formal variables.  The moment generating function $M_{\Lambda}(t)$ satisfies
\begin{align*}
M_{\Lambda}(t) 
&= \prod_{i=1}^n M(\lambda_i t) 
\overset{\eqref{eq:CumulantGen}}{=} \exp\!\bigg( \sum_{i=1}^n K(\lambda_i t) \bigg)
\overset{\eqref{eq:CumulantGen}}{=} \exp\!\bigg( \sum_{j=1}^{\infty} \kappa_j (\lambda_1^j +\cdots +\lambda_n^j)\frac{t^j}{j!}\bigg) \\
&=\exp \!\bigg( \sum_{j=1}^{\infty} \kappa_j \tr (A^j)\frac{t^j}{j!}\bigg) 
\overset{\eqref{eq:ExpoBell}}{=}\sum_{\ell=0}^{\infty} B_{\ell}(\kappa_1\tr A, \kappa_2\tr A^2, \ldots, \kappa_{\ell}\tr A^{\ell})\frac{t^{\ell}}{\ell !}. 
\end{align*}
Expanding the right side of \eqref{eq:ExpoBell} yields
\begin{equation}\label{eq:Bxy}
B_{\ell}(x_1,x_2, \ldots, x_{\ell})=\ell ! \!\!\!\sum_{\substack{j_1,j_2, \ldots, j_{\ell}\geq 0\\ j_1+2j_2+\cdots +\ell j_{\ell}=\ell  }}\prod_{r=1}^{\ell} \frac{x_r^{j_r}}{(r!)^{j_r} j_r!}= \sum_{\bvec{\pi}\vdash \ell} x_{\bvec{\pi}} y_{\bvec{\pi}},
\end{equation} 
in which $x_{\bvec{\pi}}=x_{i_1}x_{i_2}\cdots x_{i_j}$ for each partition $\bvec{\pi}=(i_1, i_2, \ldots, i_j)$ of $\ell$. 
Substitute $x_i= \kappa_i \tr (A^i)$ above and obtain
\begin{equation*}
\norm{A}_{\X,d}^d= d! \cdot [t^d] M_{\Lambda}(t) = B_{d}(\kappa_1\tr A, \kappa_2\tr A^2, \ldots, \kappa_d\tr A^d). 
\end{equation*}
Finally, \eqref{eq:Bxy} and the above ensure that
$
\norm{A}_{\X,d}^d = \sum_{\bvec{\pi}\vdash d} y_{\bvec{\pi}}\kappa_{\bvec{\pi}} p_{\bvec{\pi}}
$ 
for $A \in \H_n$. $\hfill \qed$

\subsection{Proof of Theorem \ref{Theorem:Main2}.(g)}\label{Subsection:ProofE}

Let $(x,x^*)^n = \big\{a_1a_2\cdots a_n : a_k\in\{x,x^*\}\!\big\}$ be the set of all words of length $n$ composed of the letters $x$ and $x^*$, and let $|w|_x$ count the occurrences of $x$ in a word $w$. For $Z\in \M_n$, let $w(Z)\in \M_n$ be the natural evaluation of $w$ at $Z$.
For example, if $w = xx^*x^2$, then $|w|_x = 3$ and $w(Z) = Z Z^* Z^2$.
The following lemma is \cite[Lem.~16]{Aguilar}.

\begin{lemma}\label{Lemma:Expr}
	Let $d\geq 2$ be even and $\bvec{\pi}=(\pi_1,\pi_2,\ldots,\pi_r)\vdash d$. 
	For $Z\in \M_n$, 
	\begin{equation}\label{eq:int2poly}
		\begin{split}
			&\frac{1}{2\pi} \int_0^{2\pi}\tr(e^{it} Z+e^{-it} Z^*)^{\pi_1}\cdots \tr(e^{it}Z+e^{-it}Z^*)^{\pi_r}\,\mathrm{d}t \\
			&\qquad\qquad = \!\sum_{\substack{
					w_j\in (x,x^*)^{\pi_j}  \\
					|w_1\cdots w_r|_x = \smash{\frac{d}{2}}
			}} \!\!\tr w_1(Z)\cdots\tr w_r(Z).
		\end{split}
	\end{equation}
\end{lemma}

\begin{proof}
For every Laurent polynomial $f\in \C[z,z^{-1}]$ with constant term $f_0$ we have $\int_0^{2\pi}f(e^{it})\,\mathrm{d}t=2\pi f_0$. Let us view
\begin{equation*}
f=\tr(zZ+z^{-1}Z^*)^{\pi_1}\cdots \tr(zZ+z^{-1}Z^*)^{\pi_r}
\end{equation*}
as a Laurent polynomial in $z$.  The constant term of $f$ is
\begin{equation*}
f_0=\sum_{w_1,\dots,w_r} \tr w_1(Z)\cdots\tr w_r(Z)
\end{equation*}
where the sum runs over all words $w_1,w_2,\ldots,w_r$ in $(x,x^*)^d$ with $|w_j|=\pi_j$ such that the number of occurrences of $x$ in $w_1 w_2\cdots w_r$ equals the number of occurrences of $x^*$ in $w_1 w_2\cdots w_r$. 
\end{proof}

Let $d \geq 2$ be even and let $\bvec{\pi}=(\pi_1,\pi_2,\ldots,\pi_r)\vdash d$. For $Z\in \M_n$, let 
\begin{equation}\label{eq:TDef}
	\TT_{\bvec{\pi}}(Z)=\frac{1}{\binom{d}{d/2}}\sum_{\substack{
			w_j\in (x,x^*)^{\pi_j}  \\
			|w_1\cdots w_r|_x = \smash{\frac{d}{2}}
		}} \!\!\tr w_1(Z)\cdots\tr w_r(Z),
\end{equation}
that is, $\TT_{\bvec{\pi}}(Z)$ is the arithmetic mean of the $\binom{d}{d/2}$ terms obtained by placing $d/2$ adjoints ${}^*$ among the $d$ copies of $Z$ in
\begin{equation*}
\vphantom{{\displaystyle1_2}}(\tr \underbrace{ZZ\cdots Z}_{\pi_1})
	(\tr \underbrace{ZZ\cdots Z}_{\pi_2})
	\cdots
	(\tr \underbrace{ZZ\cdots Z}_{\pi_r}). 
\end{equation*}
Consider the conjugate transpose ${}^*$ on $\V=\M_n$.
The corresponding real subspace of $*$-fixed points is $\V_{\R} = \H_n$. 
If $Z \in \M_n$, applying Proposition \ref{p:gennorm} to the norm $\norm{\cdot}_{\vec{X},d}$ on $\H_n$ ensures that the extension $\cnorm{\cdot}_{\vec{X},d}$ defined by \eqref{eq:NormComplex} is a norm on $\M_n$:
\begin{align*}
	\cnorm{Z}_{\vec{X},d}^d
	\!&\overset{\eqref{eq:NormComplex}}{=}\!  \frac{1}{2\pi \binom{d}{d/2}} \int_0^{2 \pi} \norm{ e^{it} Z + e^{-it}Z^*}_{\bvec{X},d}^d \,\mathrm{d}t \\
	&\overset{\eqref{eq:RealPermForm}}{=}\!  \frac{1}{2\pi \binom{d}{d/2}} \int_0^{2 \pi} \sum_{\bvec{\pi} \,\vdash\, d} y_{\bvec{\pi}} \kappa_{\bvec{\pi}}p_{\bvec{\pi}}( \bvec{\lambda}(e^{it} Z  + e^{-it}Z^*))\,\mathrm{d}t  \\
	&\,\overset{\eqref{eq:pTrace}}{=}\!  \frac{1}{\binom{d}{d/2}}\! \sum_{\bvec{\pi} \,\vdash\, d} \frac{y_{\bvec{\pi}} \kappa_{\bvec{\pi}}}{2\pi} \int_0^{2\pi} \!\tr(e^{it} Z + e^{-it}Z^*)^{\pi_1} \cdots
	\tr (e^{it} Z + e^{-it} Z^*)^{\pi_r} \mathrm{d}t  \\
	&\overset{\eqref{eq:int2poly}}{=}\!  \frac{1}{\binom{d}{d/2}}\! \sum_{\bvec{\pi} \,\vdash\, d} y_{\bvec{\pi}} \kappa_{\bvec{\pi}} \sum_{\substack{
			w_j\in (x,x^*)^{\pi_j}  \\
			|w_1\cdots w_r|_x = \smash{\frac{d}{2}}
		}} \!\!\tr w_1(Z)\cdots\tr w_r(Z) \\
	&\overset{\eqref{eq:TDef}}{=}\!   \sum_{\bvec{\pi} \,\vdash\, d} y_{\bvec{\pi}} \kappa_{\bvec{\pi}} \TT_{\bvec{\pi}}(Z) .
	\tag*{\qed}
\end{align*}

\begin{remark}\label{Remark:Czz}
Here is another way to restrict $\cnorm{\cdot}_{\X,d}$ to the Hermitian matrices. 
The proof of the above shows that $\binom{d}{d/2} \cnorm{A}_{\X,d}^d $ is the coefficient of $z^{d/2}\overline{z}^{d/2}$ in 
\begin{equation*}
\cnorm{zA+\overline{z}A^*}_{\X,d}^d \in \C[z,\overline{z}].
\end{equation*}
\end{remark}

\section{Examples and Applications}\label{Section:Examples}
In this section, several examples are given for the random vector norms $\|\cdot\|_{\X,d}$ on $\H_n$ 
and their extensions $\cnorm{\cdot}_{\X,d}$ to $\M_n$ for varying $d$ and different distributions of the random variables $X_i$. We begin with computations for small $d$ in Subsection \ref{Subsection:Small}, and then consider finite discrete random variables in Subsections \ref{Subsection:Poisson}--\ref{Subsec:Rademacher}. 

Let us pause briefly to recall some facts about discrete random variables.
If the random variable $X$ is supported on $\{a_1, a_2, \ldots, a_{\ell}\} \subset \R$, with
$\P(X=a_j)=q_j>0$ for $1\leq j \leq \ell$ and $q_1+q_2+\cdots +q_{\ell}=1$, then
$\E[X^k]=\sum_{j=1}^{\ell} q_j a_j^k$. Thus, 
\begin{equation}
	M(t)=\sum_{j=1}^{\ell}q_j\bigg(\sum_{k=0}^{\infty}  a_j^k \frac{t^k}{k!}\bigg)=\sum_{j=1}^{\ell} q_j e^{a_jt}.\label{eq:FiniteDiscrete}
\end{equation} 

Continuous random variables are treated in Subsections \ref{Subsection:a-stable}--\ref{Subsection:Gamma}. In particular, the norms arising from the normal distribution are treated in Subsection \ref{Subsection:Normal}. Moreover, we show in Subsection \ref{Subsection:Gamma} that Gamma random variables yield a norm related to the CHS polynomials, studied in Section \ref{Section:Hunter}. This observation will be crucial in generalizing the results of Section \ref{Section:Hunter} in Subsection \ref{Subsection:Hunter}.

Lastly, recall that a norm is determined by its unit ball. 
This provides one way to ``visualize''  random vector norms.  We make use of this idea in 
Figures \ref{Figure:Bernoulli}, \ref{Figure:Bernoulli2}, \ref{Figure:Normal}, \ref{Figure:Pareto124}, \ref{Figure:Exponentials}
below, which depict ordered pairs $(\lambda_1,\lambda_2)$ such that $\|\diag(\lambda_1,\lambda_2)\|_{\X,d} = 1$.

\subsection{The case \texorpdfstring{$\bvec{d=4}$}{d=4}}\label{Subsection:Small}
Let $\X=(X_1, X_2, \ldots, X_n)$, where the $X_i$ are 
nondegenerate iid random variables
such that the stated cumulants and moments exist.
For $d=4$, we obtain trace-polynomial representations of 
$\cnorm{Z}_{d}$ in terms of cumulants or moments using \eqref{eq:ExtendedGeneral}, as was done in Example \ref{Ex:d=2} for $d=2$. This can also be done for $d=6,8,\dots$, 
but we refrain from the exercise.

\begin{example}\label{Example:General4}
	The five partitions of $d=4$ satisfy
	\begin{equation*}
		\begin{array}{rclrclrclrclrcl}
			\kappa_{(4)} &=& \kappa_4,\quad
			&\kappa_{(3,1)} &=& \kappa_1 \kappa_3,\quad
			&\kappa_{(2,2)} &=& \kappa_2^2,\quad
			&\kappa_{(2,1,1)} &=& \kappa_2 \kappa_1^2,\quad 
			&\kappa_{(1,1,1,1)} &=& \kappa_1^4, \\
			y_{(4)} &=& 1,\quad
			&y_{(3,1)} &=& 4,\quad
			&y_{(2,2)} &=& 3,\quad
			&y_{(2,1,1)} &=& 6,\quad 
			&y_{(1,1,1,1)} &=& 1.
		\end{array}
	\end{equation*}
	There are $\binom{4}{2} = 6$ ways to place two adjoints ${}^*$ in a string of four $Z$s.  
	For example, 
	\begin{align*}
		6\TT_{(3,1)}(Z)
		&= (\tr Z^*Z^*Z)(\tr Z) + (\tr Z^*ZZ^*)(\tr Z) + (\tr Z^*ZZ)(\tr Z^*) \\
		&\qquad + (\tr ZZ^*Z^*)(\tr Z) +(\tr ZZ^*Z)(\tr Z^*) +(\tr ZZZ^*)(\tr Z^*)\\
		&=3 \tr (Z^{*2}Z)(\tr Z) +3 (\tr Z^2 Z^*)(\tr Z^*).
	\end{align*}
	Summing over all five partitions yields the following norm on $\M_n$:
	\begin{align}
		\cnorm{Z}_{\X,4}^4
		&= 
		\kappa_1^4 (\tr Z^*)^2 (\tr Z)^2 
		+ \kappa_2 \kappa_1^2 (\tr Z^*)^2 \tr (Z^2) 
		+ \kappa_2 \kappa_1^2 \tr (Z^{*2})(\tr Z)^2  \nonumber \\
		&\qquad+4 \kappa_2 \kappa_1^2  (\tr Z^*) (\tr Z^*Z) (\tr Z)
		+2 \kappa_3 \kappa_1 \tr ( Z^{*2}Z)(\tr Z) \nonumber \\
		&\qquad +2 \kappa_3 \kappa_1 \tr (Z^*) \tr (Z^* Z^2 ) 
		+2 \kappa_2^2 (\tr  Z^*Z)^2
		+ \kappa_2^2 \tr (Z^2) \tr (Z^{*2}) \nonumber \\
		&\qquad +\tfrac{2}{3} \kappa_4 \tr (Z^2 Z^{*2})+\tfrac{1}{3} \kappa_4 \tr (Z Z^* Z Z^*) . \label{eq:Z44}
	\end{align}
	If $Z = A\in \H_n$,
	Theorem \ref{Theorem:Main2}.(f) and \eqref{eq:Bell4} ensure that the above reduces to
	\begin{align*}
		\norm{Z}_{\X,4}^4 
		&=
		\kappa _1^4 (\tr A)^4+6 \kappa _1^2 \kappa _2 \tr(A^2) (\tr A)^2
		+4 \kappa _1 \kappa _3 \tr (A^3) \tr (A) 
		\\
		&\qquad + 3 \kappa _2^2 \tr (A^2)^2+\kappa _4 \tr (A^4) .
	\end{align*}
\end{example}

\subsection{Poisson random variables}\label{Subsection:Poisson}
Let $\X=(X_1, X_2, \ldots, X_n)$, in which the $X_i$ are independent random variables on 
$\{0,1,2,\ldots\}$ distributed according to the probability density function
$f(t)= \frac{e^{-\alpha} \alpha^t}{t!}$,
in which $\alpha >0$.
The moment and cumulant generating functions of the $X_i$ are
$M(t)= e^{\alpha  (e^t-1)}$ and
$K(t) = \alpha  (e^t-1)$,
respectively. Therefore, $\kappa_i = \alpha$ for all $i \in \N$ and hence, for even $d\geq 2$,
\begin{equation*}
	\norm{A}_{\X, d}^d= \sum_{\bvec{\pi}\vdash d}    \alpha^{|\bvec{\pi}|} y_{\bvec{\pi}} p_{\bvec{\pi}} \quad \text{for $A\in \H_n$.}
\end{equation*} 
For example,
\begin{align*}
	\norm{A}_{\X,4}^4
	=\alpha^4(\tr A)^4 \!+\!6\alpha^3(\tr A)^2\tr A^2 \!+\! 4\alpha^2\tr A\tr A^3 \!+\! 3\alpha^2(\tr A^2)^2 \!+\! \alpha \tr A^4.
\end{align*}

\subsection{Bernoulli random variables}\label{Subsection:Bernoulli}
A \emph{Bernoulli} random variable is a discrete random variable $X$ defined by the relation $\P(X=k)=q^k(1-q)^{1-k}$ for $k\in \{0,1\}$ and $0<q<1$. Suppose $d$ is an even integer and $\X$ is a random vector whose entries are independent Bernoulli random variables with parameter $q$. Each $X_i$ satisfies
$\E [X_i^k]=\sum_{j\in \{0,1\}} j^k\P(X_i=j)=q$ for $k \in \N$.
We have
$M(t) = 1-q + qe^t$ and
$K(t) = \log(1-q+qe^t)$.
Moreover, the first few cumulants are 
\begin{equation*}
	\kappa_1=q, ~~~ \kappa_2=q-q^2, ~~~ \kappa_3=q-3q^2+2q^3, ~~~ \kappa_4=q-7q^2+12q^3-6q^4,\ldots
\end{equation*}
For even $d\geq 2$, the multinomial theorem and independence imply that
\begin{equation*}
	\norm{A}_{\X,d}^d
	=\sum_{ k_1+k_2+\cdots+k_n=d}  \binom{d}{k_1,k_2,\ldots,k_n} q^{|I|} \lambda_1^{k_1}\lambda_2^{k_2}\cdots \lambda_n^{k_n},
\end{equation*} 
in which $k_i\geq 0$ and $I=\{i: k_i \neq 0\}$. This can written as
\begin{equation*}
	\norm{A}_{\X,d}^d=\sum_{\bvec{\pi}\, \vdash \, d} |\bvec{\pi}|! \,q^{| \bvec{\pi}|} m_{\bvec{\pi}}(\bvec{\lambda}),
\end{equation*} 
in which $m_{\bvec{\pi}}$ denotes the \emph{monomial symmetric polynomial} 
corresponding to the partition $\bvec{\pi}$ of $d$ \cite[Sec.~7.3]{StanleyBook2}.
To be more specific,
\begin{equation*}
	m_{\bvec{\pi}}(\bvec{x})=\sum_{\smash{\bvec{\alpha}}} x^{\bvec{\alpha}},\vspace{-4pt}
\end{equation*} 
in which the sum is taken over all distinct permutations $\bvec{\alpha}=(\alpha_1, \alpha_2, \ldots, \alpha_r)$ of the entries of $\bvec{\pi}=(i_1, i_2, \ldots, i_r)$ and $x^{\bvec{\alpha}}=x_1^{\alpha_1}x_2^{\alpha_2}\cdots x_r^{\alpha_r}$. For example,
$m_{(1)} =\sum_i x_i$,
$m_{(2)} =\sum_i x_i^2$, and
$m_{(1,1)}=\sum_{i<j}x_ix_j$. 
Figure \ref{Figure:Bernoulli} illustrates the unit balls for these norms in two cases. We remark that $\norm{A}_{\X,d}$ approaches $|\tr A|$ as $q\to 1$, which is not a norm.

\begin{figure}
	\centering
	\includegraphics[width = 0.475\textwidth]{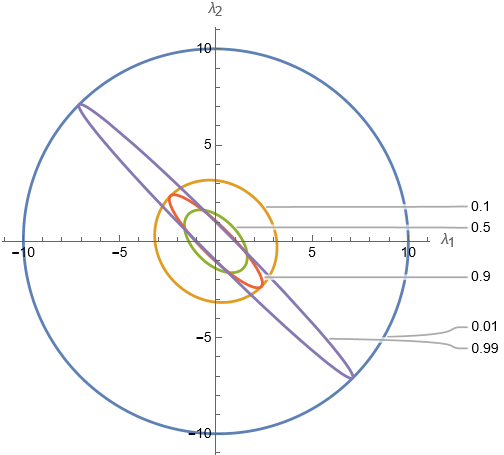}
	\hfill
	\includegraphics[width = 0.475\textwidth]{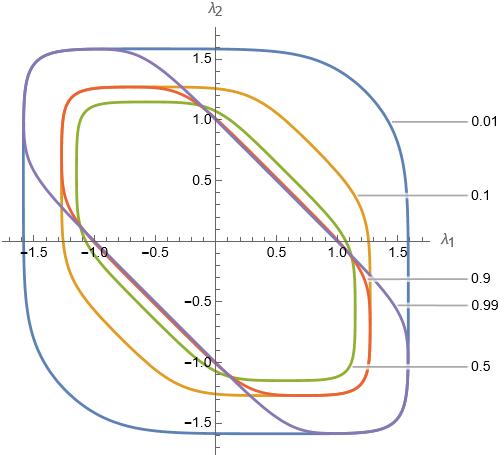}
	\caption{Unit circles for $\|\cdot\|_{\X,d}$, in which $X_1$ and $X_2$ are Bernoulli with varying parameter $q$ and with $d=2$ (\textsc{Left}) and $d=10$ (\textsc{Right}).}
	\label{Figure:Bernoulli}
\end{figure}

\subsection{Rademacher random variables}
\label{Subsec:Rademacher}

Let $\X=(X_1, X_2, \ldots, X_n)$, in which the $X_i$ are independent \emph{Rademacher} random variables defined on $\{-1,+1\}$ by $\P(X_i=1)=\P(X_i=-1)=\frac{1}{2}$. In that case, identity 
\eqref{eq:FiniteDiscrete} yields $M(t)=\cosh t$, so
$M_{\Lambda}(t) = \prod_{i=1}^n\cosh(\lambda_it)$.
Thus, for $n=2$, we have
$\norm{A}_{\X,2}^2  =  \lambda_1^2+\lambda_2^2$,
$\norm{A}_{\X,4}^4  = \lambda_1^4+6 \lambda_1^2 \lambda_2^2+\lambda_2^4$, and
\begin{equation}\label{Eq:Rademacher}
	\norm{A}_{\X,6}^6  = \lambda_1^6+15 \lambda_1^4 \lambda_2^2+15 \lambda_1^2 \lambda_2^4+\lambda_2^6. 
\end{equation}
Let $\gamma_d=\sqrt{2} (\sqrt{\pi})^{-1/d }\Gamma(\frac{d+1}{2})^{1/d}$ 
denote the $d$th moment of a standard normal random variable. The classic Khintchine inequality asserts that 
\begin{equation}
	\bigg(\E \Big|\sum_{i=1}^n \lambda_iX_i \Big|^2\bigg)^{\!\!1/2} \!\!\leq \bigg(\E \Big|\sum_{i=1}^n \lambda_iX_i\Big|^d\bigg)^{\!\!1/d}\!\!\leq \gamma_d\bigg(\E \Big|\sum_{i=1}^n \lambda_iX_i\Big|^2\bigg)^{\!\!1/2}
\end{equation} 
for all $\lambda_1, \lambda_2, \ldots, \lambda_n\in \R$. Moreover, these constants are optimal \cite{Haagerup}. Immediately, we obtain the equivalence of norms
\begin{equation}\label{eq:NormEquivalence}
	\norm{A}_{\mathrm{F}}\leq \norm{A}_{\X,d}\leq \gamma_d\norm{A}_{\mathrm{F}},
\end{equation} 
for all $A\in \H_n(\C)$ and $d\geq 2$ in the case of Rademacher random variables.


\subsection{\texorpdfstring{$\bvec{\alpha}$}{α}-stable random variables}\label{Subsection:a-stable}

A distribution is said to be \emph{stable} if a linear combination, involving strictly positive coefficients, of two independent random variables following this distribution retains the same distribution, up to location and scale parameters. Stable distributions are characterized by four parameters: the stability parameter $\alpha \in (0, 2]$, the skewness parameter $\beta \in [-1, 1]$, the scale parameter $\gamma \in (0, \infty)$, and the location parameter $\delta \in \R$. When $\beta = \delta = 0$, the resulting distribution is called a \emph{symmetric $\alpha$-stable distribution}, denoted by $\vec{S}(\alpha, \gamma)$, and its characteristic function (defined by \eqref{eq:characteristic}) is
\begin{equation*}
\varphi(x)=\exp\!\big(\!-|\gamma x|^{\alpha}\big).
\end{equation*}
The case $\alpha = 2$ represent the normal distribution, while $\alpha = 1$ yield the Cauchy distribution. These are in fact the only instances of a symmetric $\alpha$-stable distribution where there exists a closed form for the probability density function \cite[p.~2]{Nolan}.

These distributions are useful in the context of random vector norms, as they allow us to easily compute the distribution of the random variable $Y:= \langle \X,\bvec{\lambda} \rangle$. Indeed, according to \cite[Prop.~1.4]{Nolan}, if $X$ follows a symmetric $\alpha$-stable distribution with a scale parameter $\gamma$, then 
\begin{equation*}
\lambda X \sim \vec{S}(\alpha,|\lambda|\gamma).
\end{equation*}
Moreover, if $X_1$ and $X_2$ follow symmetric $\alpha$-stable distributions with scale parameters $\gamma_1$ and $\gamma_2$, respectively, their sum $X_1+X_2$ follows the distribution $\vec{S}(\alpha,\gamma)$, where $\gamma=(\gamma_1^\alpha+\gamma_2^\alpha)^{1/\alpha}$.

In the following, suppose that $\alpha\in (1,2)$ and $\gamma>0$. If $X$ follows the distribution $\vec{S}(\alpha,\gamma)$ and $\lambda_j$ are the eigenvalues of the Hermitian matrix $A$, then
\begin{equation*}
Y= \langle \X,\bvec{\lambda} \rangle \sim \vec{S}(\alpha,\gamma \|A\|_{S_\alpha}).
\end{equation*} 
If $\varphi$ is the characteristic function of $Y$ and $d\in [1,\alpha)$, then \eqref{eq:absolutemoment} ensures that
\begin{align*}
	\E\big[|\langle \X,\bvec{\lambda} \rangle |^d\big] &= \frac{2 \Gamma(d+1) \sin\!\big( \tfrac{d\pi}{2} \big)}{\pi} \int_{0}^{\infty}\frac{1-\exp\!\big(\!-\gamma^\alpha \|A\|_{S_\alpha}^{\alpha}t^{\alpha}\big)}{t^{d+1}}\mathrm{d}t \\
	&= \frac{2 \gamma^d \sin\!\big( \tfrac{d\pi}{2} \big) \Gamma(d+1)}{\alpha\sin\!\big(\tfrac{ d\pi}{\alpha}\big)\Gamma\big(\tfrac{d}{\alpha}+1\big)} \|A\|_{S_\alpha}^d.
\end{align*}
Consequently,
\vspace{-6pt}
\begin{equation}\label{Eq:alpha}
	\|A\|_{\X,d} = \gamma \!\left( \frac{2 \sin\!\big( \tfrac{d\pi}{2} \big) \Gamma(d+1)}{\alpha\sin\!\big(\tfrac{ d\pi}{\alpha}\big)\Gamma\big(\tfrac{d}{\alpha}+1\big)} \right)^{\!\!\frac{1}{d}}\! \|A\|_{S_\alpha} \quad \text{for $A\in \H_n$.}
\end{equation}


\subsection{Normal random variables}\label{Subsection:Normal}

Suppose that $\X=(X_1, X_2, \ldots, X_n)$ is a random vector whose entries are independent normal random variables with mean $\mu$ and variance $\sigma^2>0$.  Then
$M(t)=\exp(t\mu+\frac{\sigma^2t^2}{2} )$ and
$K(t) = \frac{\sigma^2 t^2}{2}+\mu  t$;
in particular, $\kappa_1 = \mu$ and $\kappa_2 = \sigma^2$, and all higher cumulants are zero.    
Then
\begin{equation*}
	M_{\Lambda}(t) 
	= \prod_{i=1}^n \exp\!\Big(\lambda_i t\mu+\frac{\sigma^2 \lambda_i^2 t^2}{2} \Big)
	=\exp\! \Big( t\mu\tr A+\frac{\sigma^2\tr(A^2)t^2}{2} \Big).
\end{equation*}
\begin{figure}
	\centering
	\includegraphics[width = 0.475\textwidth]{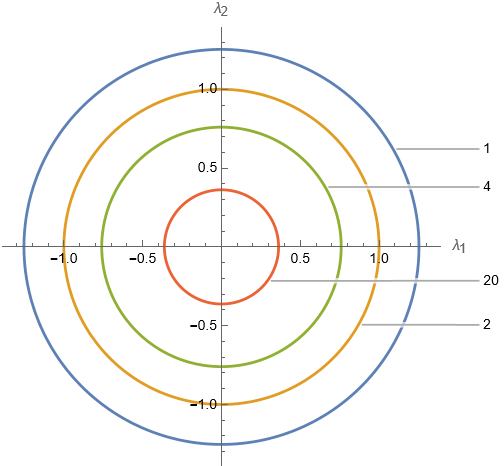}
	\hfill\includegraphics[width = 0.475\textwidth]{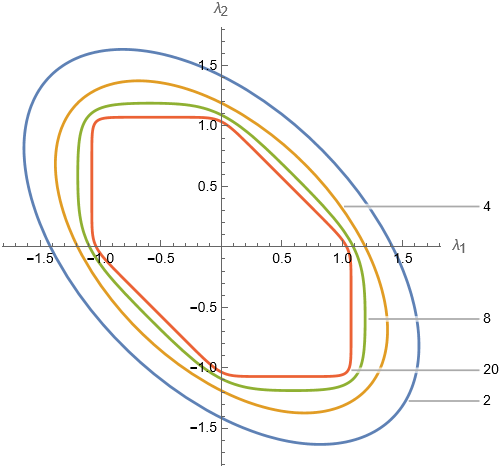}
	\caption{(\textsc{Left}) Unit circles for $\|\cdot\|_{\X,d}$ with $d=1, 2, 4, 20$, in which $X_1$ and $X_2$ are standard normal random variables. (\textsc{Right}) Unit circles for $\norm{\cdot}_{\X,d}$ with $d=2, 4,8, 20$, in which $X_1$ and $X_2$ are Bernoulli with $q=0.5$.}
	\label{Figure:Bernoulli2}
\end{figure}
For even $d\geq 2$, Theorem \ref{Theorem:Main2}.(f) yields
\begin{align*}
	\cnorm{Z}_{\X,2}^2
	&= \sigma^2 \|Z\|_{\operatorname{F}}^2 + \mu^2 |\tr Z|^2 , \\[3pt]
	\cnorm{Z}_{\X,4}^4
	&= 
	\mu^4 (\tr Z)^2 (\tr Z^*)^2
	+\mu^2 \sigma^2 \tr (Z^*)^2\tr (Z^2) +\mu^2 \sigma^2 (\tr Z)^2 \tr (Z^{*2}) \\
	&~~+4 \mu^2 \sigma^2 (\tr Z)  (\tr Z^*) (\tr Z^* Z) +2 \sigma^4 (\tr Z^* Z)^2 
	+\sigma^4 \tr (Z^2) \tr (Z^{*2}).
\end{align*}
Note that since $\kappa_r=0$ for $r \geq 3$,
$\cnorm{Z}_{\X,4}^4$ does not contain summands like
$\tr(Z^*)\tr (Z^* Z^2)$ and $\tr (Z^{*2} Z^2)$, in contrast to the formula in Example \ref{Example:Gamma4}. In the case where $A$ is Hermitian and $d\geq 2$ is even, Theorem \ref{Theorem:Main2}.(e) and the above tell us that for even $d\geq2$,
\begin{equation}\label{eq:Same}
	\norm{A}_{\X,d}^d= d! \sum_{k=0}^{\frac{d}{2}}  \frac{\mu^{2k} (\tr A)^{2k}}{(2k)!} 
	\cdot  \frac{\sigma^{d-2k} \norm{A}_{\operatorname{F}}^{d-2k}}{2^{\frac{d}{2}-k} (\frac{d}{2}-k)!}
	\quad \text{for $A \in \H_n$},
\end{equation}      
in which $\norm{A}_{\operatorname{F}}$ is the Frobenius norm of $A$. 
This formula is convenient for computational purpose. However, since the normal distribution is an $\alpha$-stable distribution (corresponding to $\alpha=2$), the argument presented in Subsection \ref{Subsection:a-stable} allows one to show that 
\begin{equation*}
	Y= \langle \X,\bvec{\lambda} \rangle \sim \vec{N}\big(\mu\tr A, \sigma^2 \|A\|_{\operatorname{F}}^2\big),
\end{equation*}
where the $X_i$ follows the normal distribution $\vec{N}(\mu,\sigma^2)$. Moreover, Winkelbauer \cite{winkelbauer2014moments} provides a formula for the $d$th absolute moments of the normal distribution:
\begin{equation*}
	\E[|X|^d] = \frac{(\sqrt{2}\sigma)^d \Gamma(\frac{d+1}{2})}{\sqrt{\pi}} \Kum\!\Big(\!-\tfrac{d}{2};\,\tfrac{1}{2};\, -\tfrac{\mu^2}{2\sigma^2} \Big),
\end{equation*}
where $\Kum(a;b;z)$ is Kummer's confluent hypergeometric function \eqref{eq:Kummer}. It immediately follows that 
\begin{equation*}
	\|A\|_{\X,d} = \sqrt{2}\sigma \|A\|_{\operatorname{F}} \bigg( \tfrac{ 1}{\sqrt{\pi}} \Gamma(\tfrac{d+1}{2}) \Kum\!\Big(\!-\tfrac{d}{2};\,\tfrac{1}{2};\, -\tfrac{\mu^2 (\tr A)^2}{2\sigma^2 \|A\|_{\operatorname{F}}^2} \Big) \! \bigg)^{\!\!1/d} 	\quad \text{for $A \in \H_n$}.
\end{equation*}
Note that if $\mu=0$, or if $\tr A=0$, then \eqref{Eq:Kummer0} ensures that the random vector norm above reduces to $$\|A\|_{\X,d} = \sqrt{2}\sigma \Big(\tfrac{ 1}{\sqrt{\pi}} \Gamma(\tfrac{d+1}{2}) \Big)^{\!1/d} \! \|A\|_{\operatorname{F}} ,$$ 
a scalar multiple of the Frobenius norm (this agrees with \eqref{Eq:alpha}). It is no surprise that the unit circles when $n=2$ and $\mu=0$ in Figure \ref{Figure:Bernoulli2} (\textsc{Left}) are simple circles. If $\mu\neq 0$, then the unit circles for $\norm{\cdot}_{\X,d}$ are approximately elliptical;
see Figure \ref{Figure:Normal}.

\begin{figure}
	\centering
	\includegraphics[width = 0.475\textwidth]{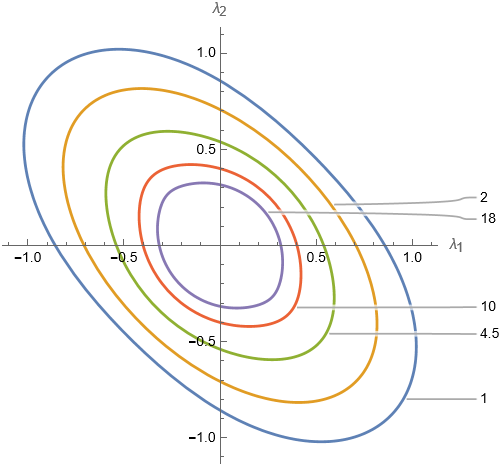}\hfill 
	\includegraphics[width = 0.475\textwidth]{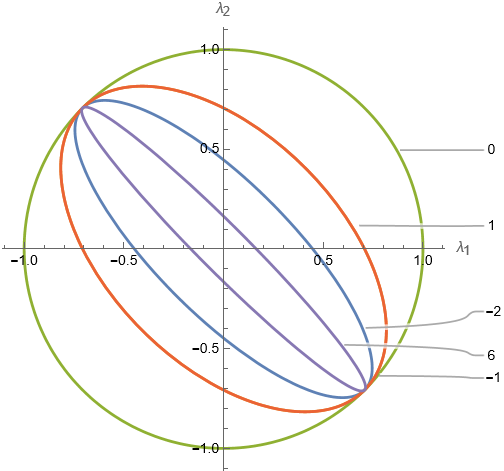}
	\caption{(\textsc{Left}) Unit circles for $\|\cdot\|_{\X,d}$ with $d=1, 2, 4.5, 10, 18$, in which $X_1$ and $X_2$ are normal random variables with $\mu=\sigma=1$. (\textsc{Right}) Unit circles for $\|\cdot\|_{\X,10}$, in which $X_1$ and $X_2$ are normal random variables with means $\mu=-2, -1, 0, 1, 6$ and variance $\sigma^2=1$.}
	\label{Figure:Normal}
\end{figure}

\subsection{Pareto random variables}\label{Subsection:Pareto}
Let $\alpha, x_m>0$. A random variable $X$ is a \emph{Pareto} random variable with parameters $\alpha$ and $x_m$ if it is distributed according to the probability density function
\begin{align*}
	f_X(t)=
	\begin{cases}
		\dfrac{\alpha x_m^{\alpha}}{t^{\alpha+1}} & \text{if $t\geq x_m$},\\ 
		0 & \text{if $t< x_m$}.
	\end{cases}
\end{align*} 
The moments that exist are 
$\mu_k = \frac{\alpha}{\alpha - k}$ for $k < \alpha$.
For even $d\geq2$ with $d < \alpha$, the multinomial theorem and independence of the random variables yield
\begin{align*}
	\norm{A}_{\X ,d}^d 
	&= \E[ \langle \X , \bvec{\lambda}\rangle^d ] 
	= \E\big[(\lambda_1 X_1 + \lambda_2 X_2 +  \cdots + \lambda_n X_n)^d \big] \\
	&= \E \bigg[ \sum_{ k_1+k_2+\cdots+k_n=d} 
	\binom{d}{k_1,k_2,\ldots,k_n} \prod_{i=1}^n (\lambda_i X_i)^{k_i} \bigg] \\
	&=  \sum_{ k_1+k_2+\cdots+k_n=d} 
	\binom{d}{k_1,k_2,\ldots,k_n} \prod_{i=1}^n \lambda_i^{k_i} \E\big[X_i^{k_i}\big] \\
	&=  \sum_{ k_1+k_2+\cdots+k_n=d} 
	\binom{d}{k_1,k_2,\ldots,k_n} \prod_{i=1}^n \frac{\alpha \lambda_i^{k_i}}{\alpha - k_i} ,
\end{align*}
where each $k_i \geq 0$. 
In particular,
$\lim_{\alpha \to\infty}\norm{A}_{\X_{\alpha},d}^d = (\tr A)^d$
and, since every term in the sum with $k_i<d$ for all $i=1,2,\dots,n$ vanishes in the limit, we have
\vspace{-10pt}%
\begin{align*}
	\lim_{\alpha \to d^{+}}
	(\alpha-d)\norm{A}_{\X_{\alpha},d}^d
	&= \lim_{\alpha \to d^{+}} (\alpha-d)\!\! \sum_{ k_1+k_2+\cdots+k_n=d} 
	\binom{d}{k_1,k_2,\ldots,k_n} \prod_{i=1}^n \frac{\alpha \lambda_i^{k_i}}{\alpha - k_i}  \\
	&= \lim_{\alpha \to d^{+}}(\alpha-d)
	\sum_{i=1}^n \binom{d}{d}\frac{d\lambda_i^{d}}{\alpha-d} = d \sum_{i=1}^n \lambda_i^{d}
	= d\norm{A}_{S_d}^d,
\end{align*}
where $\norm{A}_{S_d}$ is the Schatten $d$-norm on $\H_n$.

\begin{example}
	For $n=2$, $\alpha>2$, and $x_m=1$, 
	\vspace{-5pt}%
	\begin{align*}
		\norm{A}_{\X ,2}^2 &=\alpha\!\bigg(\frac{\lambda_1^2}{\alpha-2}+\frac{2\alpha\lambda_1\lambda_2}{(\alpha-1)^2}+\frac{\lambda_2^2}{\alpha-2} \bigg)\!, \quad \text{and}\\
		\norm{A}_{\X ,4}^4 &= \alpha\! \bigg(\frac{\lambda_1^4}{\alpha-4} + \frac{4\alpha\lambda_1^3\lambda_2}{\alpha^2-4\alpha+3} + \frac{6\alpha\lambda_2^2\lambda_1^2}{(\alpha-2)^2} + \frac{4\alpha\lambda_1\lambda_2^3}{\alpha^2-4\alpha+3} +
		\frac{\lambda_2^4}{\alpha-4}\bigg)\!.
	\end{align*}
	Figure \ref{Figure:Pareto124} (\textsc{Left}) illustrates the unit circles for $\norm{\cdot}_{\X,2}$ with varying $\alpha$; 
	as $\alpha\to \infty$ the unit circles approach the parallel lines at $y=\pm 1-x$. 
	Figure \ref{Figure:Pareto124} (\textsc{Right}) depicts the unit circles for $\norm{\cdot}_{\X,d}$ with fixed $\alpha$ and varying $d$. 
\end{example}

\begin{figure}
	\centering
	\includegraphics[width = 0.45\textwidth]{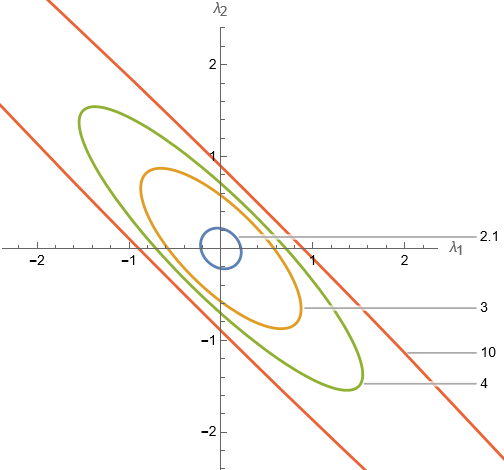}
	\hfill
	\includegraphics[width = 0.45\textwidth]{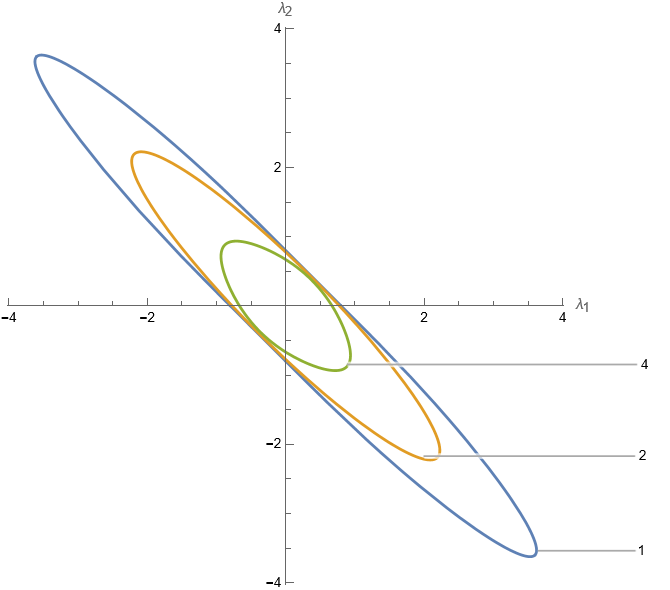}
	\caption{(\textsc{Left}) Unit circles for $\norm{\cdot}_{\X,2}$, in which $X_1$ and $X_2$ are independent Pareto random variables with $\alpha=2.1, 3, 4, 10$ and $x_m=1$. (\textsc{Right}) Unit circles for $\norm{\cdot}_{\X,d}$, in which $X_1$ and $X_2$ are independent Pareto random variables with $\alpha=5$ and $d=1, 2, 4$. }
	\label{Figure:Pareto124}
\end{figure}


\subsection{Laplace random variables}\label{Subsection:Laplace}
Let $\X=(X_1, X_2, \ldots, X_n)$, where 
the $X_i$ are independent Laplace random variables distributed according to the probability density function
$f(x)=\frac{1}{2\beta} e^{-|x-\mu|/\beta}$,
where $\mu \in \R$ and $\beta>0$.
The moment and cumulant generating functions of the $X_i$ are
$M(t)=\frac{e^{\mu t}}{1-\beta^2t^2}$ and
$K(t) = \mu  t-\log (1-\beta^2 t^2)$,
respectively.
The cumulants are
\vspace{-6pt}
\begin{equation*}
	\kappa_r = 
	\begin{cases}
		\mu & \text{if $r=1$},\\
		2 \beta^r (r-1)! & \text{if $r$ is even},\\
		0 & \text{otherwise}.
	\end{cases}
\end{equation*}
For even $d\geq 2$, it follows that $\norm{A}_{\X,d}^d$ is the $d$th term in the Taylor expansion of 
\begin{equation}\label{eq:Laplace}
	\norm{A}_{\X,d}^d 
	= [t^d] \prod_{i=1}^n\frac{ d! e^{\mu \lambda_i t}}{1-\beta^2\lambda_i^2t^2}
	=  e^{\mu t \tr A } [t^d]\prod_{i=1}^n\frac{ d!}{1-\beta^2\lambda_i^2t^2}.
\end{equation} 

\begin{example}
	Let $\mu=\beta=1$ and let $d\geq 2$ be an even integer. Expanding the terms in \eqref{eq:Laplace} gives
	\begin{equation*}
		M_{\Lambda}(t)=e^{ t\tr A }\prod_{i=1}^n\frac{ 1}{1-\lambda_i^2t^2}=\Big(\sum_{k=0}^{\infty} (\tr A)^k\frac{t^k}{k!}  \Big)\Big(\sum_{k=0}^{\infty} h_k(\lambda_1^2, \lambda_2^2, \ldots, \lambda_n^2)t^{2k}    \Big),
	\end{equation*}
	which implies that
	$
	\norm{A}_{\X,d}^d= d!\sum_{k=0}^{d/2} \frac{(\tr A)^{2k}}{(2k)!} h_{\frac{d}{2}-k}(\lambda_1^2, \lambda_2^2, \ldots, \lambda_n^2)
	$.
\end{example}

\subsection{Uniform random variables}\label{Subsection:Uniform}
Let $\X=(X_1, X_2, \ldots, X_n)$, where the $X_i$ are independent and uniformly distributed on $[a,b]$. 
Each $X_i$  has probability density $f(x)=(b-a)^{-1}\bvec{1}_{[a,b]}$, where $\bvec{1}_{[a,b]}$ is the indicator function of $[a,b]$. Then 
\begin{equation}\label{eq:MomentUniform}
	\mu_k = \E  [X_i^k] = \int_{-\infty}^{\infty} x^k f(x)\,\mathrm{d}x= \frac{h_k(a,b)}{k+1},
\end{equation} 
in which $h_k(a,b)$ is the CHS polynomial of degree $k$ in the variables $a,b$. The moment and cumulant generating functions of each $X_i$ are
$M(t)=\frac{e^{bt}-e^{at}}{t(b-a)}$ and
$K(t) = \smash{\log (\frac{e^{t (b-a)}-1}{t (b-a)})+a t}$.
The cumulants are
\begin{equation*}
	\kappa_r =  
	\begin{cases}
		\frac{a+b}{2} & \text{if $r=1$},\\[2pt]
		\frac{B_r}{r}(b-a)^r & \text{if $r$ is even},\\[2pt]
		0 & \text{otherwise},
	\end{cases}
\end{equation*}
in which $B_r$ is the $r$th Bernoulli number \cite{Gould}. 
Theorem \ref{Theorem:Main2}.(e) ensures that 
\begin{equation}\label{eq:GeneralSinh}
	\norm{A}_{\X,d}^d = d! \cdot [t^d]\prod_{i=1}^n\frac{ e^{b\lambda_i t}-e^{a\lambda_i t}}{\lambda_it(b-a)}
	\quad \text{for $A \in \H_n$}.
\end{equation}

\begin{example}
	If $[a,b]=[0,1]$, then
	$M_{\Lambda}(t) = \prod_{i=1}^n \frac{e^{\lambda_i t}-1}{\lambda_i t}$,
	and hence for $A \in \H_n$,
	\begin{align*}
		\norm{A}_{\X,2}^2  &= \tfrac{1}{6}(2 \lambda_1^2+3 \lambda_1 \lambda_2+2 \lambda_2^2), \\
		\norm{A}_{\X,4}^4  &= \tfrac{1}{30}( 6 \lambda_1^4 +15 \lambda_1^3 \lambda_2 +20 \lambda_1^2 \lambda_2^2 +15 \lambda_1 \lambda_2^3+6 \lambda_2^4  
		).
	\end{align*}
	Unlike \eqref{Eq:Rademacher}, these symmetric polynomials are not obviously positive definite
	since $\lambda_1^3 \lambda_2$ and $\lambda_1 \lambda_2^3$ need not be nonnegative.
\end{example}

\begin{example}\label{Example:UniformSymmetric}
	If $[a,b]=[-1,1]$, then 
	\begin{equation*}
		\cnorm{Z}_{\X,4}^4 = \tfrac{1}{45}\big(
		10 (\tr Z^*Z)^2 \!+\! 5 \tr (Z^2)\! \tr (Z^{*2}) \!-\! 4 (\tr Z^2Z^{*2}) \!-\! 2 \tr (ZZ^*ZZ^*)\big)
	\end{equation*}
	for $Z \in \M_n$, which is not obviously positive, let alone a norm.  Similarly,
	\begin{equation*}
		\norm{A}_{\X,6}^6
		= \tfrac{1}{63}
		\big(35 (\tr A^2)^3-42 \tr(A^4) \tr (A^2)+16 \tr(A^6) \big)
		\quad \text{for $A \in \H_6$}
	\end{equation*}
	has a nonpositive summand.
	Since
	$
	M_{\Lambda}(t)= \prod_{i=1}^n\frac{\sinh(\lambda_it)}{\lambda_i t} 
	$
	is an even function of each $\lambda_i$, the corresponding norms are polynomials
	in even powers of the eigenvalues (so positive definiteness is no surprise, although the triangle inequality 
	is nontrivial). Figure \ref{Figure:Exponentials} (\textsc{Right}) shows the unit circles for this norm when $n=2$ for varying $d$.
\end{example}


\subsection{Gamma random variables}\label{Subsection:Gamma}
Let $\X=(X_1, X_2, \ldots, X_n)$, in which the $X_i$ are independent Gamma random variables with probability density
\begin{equation}\label{eq:Gamma}
	f(t)=\begin{cases}
		\frac{1}{\beta^{\alpha} \Gamma(\alpha)} t^{\alpha - 1} e^{-t/\beta} & \text{if $t> 0$},\\[2pt] 
		0 & \text{if $t\leq 0$}.
	\end{cases}
\end{equation}  
Here $\alpha, \beta>0$ (note that $\alpha = k/2$ and $\beta = 2$ yield a $\chi^2$
random variable with $k$ degrees of freedom, and $\alpha=\beta=1$ is the standard exponential distribution). Then
$M(t) = (1 - \beta t)^{-\alpha}$ and
$K(t) = - \alpha \log(1 - \beta t)$,
so 
\begin{equation}\label{eq:KappaGamma}
	\kappa_r = \alpha \beta^r(r-1)! \quad \text{for $r \in \N$}.
\end{equation}
For even $d\geq 2$,
\begin{equation}\label{eq:GammaMGF}
	\norm{A}_{\X,d}^d 
	= [t^d] \prod_{i=1}^n \frac{d!}{(1-\beta \lambda_i t)^{\alpha}}
	= d! \cdot [t^d] \Bigg(\frac{\beta^{-n}t^{-n}}{ p_A(\beta^{-1}t^{-1})}\Bigg)^{\!\!\alpha}
	\quad \text{for $A \in \H_n$},
\end{equation} 
in which $p_A(t) = \det(tI-A)$ denotes the characteristic polynomial of $A$.

\begin{example}\label{Example:Gamma4}
	Since $\kappa_1 = \alpha \beta$ and $\kappa_2 = \alpha \beta^2$, \eqref{eq:d=2} becomes
	$$\cnorm{Z}_{\X,2}^2
	=  \alpha  \beta^2 \|Z\|_{\operatorname{F}}^2 +  \alpha^2 \beta^2 |\tr Z|^2
	\quad \text{for $Z\in \M_n.$}$$
	Similarly, \eqref{eq:Z44} yields
	\begin{align*}
		\cnorm{Z}_{\X,4}^4
		=\,  
		&\alpha^4 \beta^4 (\tr Z)^2 (\tr Z^*)^2
		+\alpha^3 \beta^4 (\tr Z^*)^2 \tr(Z^2) 
		+2 \alpha^2 \beta^4 (\tr Z^* Z)^2
		\\
		&+4 \alpha^3 \beta^4 (\tr Z)(\tr Z^*)(\tr Z^* Z) +\alpha^3 \beta^4 (\tr Z)^2 \tr(Z^{*2})
		 \\
		& +\alpha^2 \beta^4 \tr (Z^2) \tr (Z^{*2}) +4 \alpha^2 \beta^4 \tr(Z^*)\tr (Z^* Z^2)
		\\
		&
		+4 \alpha^2 \beta^4 \tr (Z) \tr (Z^{*2} Z) +2 \alpha  \beta^4 \tr (Z^* Z Z^* Z)
		+4 \alpha  \beta^4 \tr (Z^{*2} Z^2)
		.
	\end{align*}
\end{example}

The case $\alpha=\beta=1$, which corresponds to standard exponential random variables, is of particular interest. Indeed, for even $d\geq 2$, \eqref{eq:GammaMGF} reduces to
\begin{equation}\label{Eq:CHS_noproof}
	\norm{A}_{\X,d}^d = d!\, h_d(\lambda_1, \lambda_2, \ldots, \lambda_n)= d!\!\sum_{1\leq k_1\leq\cdots\leq k_d\leq n} \lambda_{k_1}\lambda_{k_2}\cdots \lambda_{k_d},
\end{equation}
a multiple of the \emph{complete homogeneous symmetric polynomial} of degree $d$ in the eigenvalues $\lambda_1, \lambda_2, \ldots, \lambda_n$ of the Hermitian matrix $A$. These norms, which have been studied in detail in \cite{Aguilar}, are the main object of study in Section \ref{Sec:CMS}. Moreover, the unit balls for this special norm are illustrated in Figure \ref{Figure:Exponentials} (\textsc{Left}) below.

\begin{figure}
	\centering
	\includegraphics[width = 0.475\textwidth]{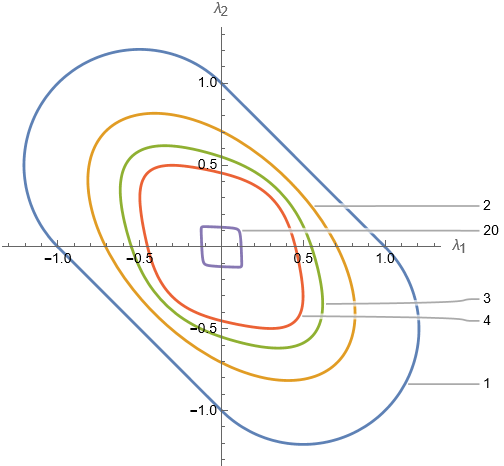}
	\hfill\includegraphics[width = 0.475\textwidth]{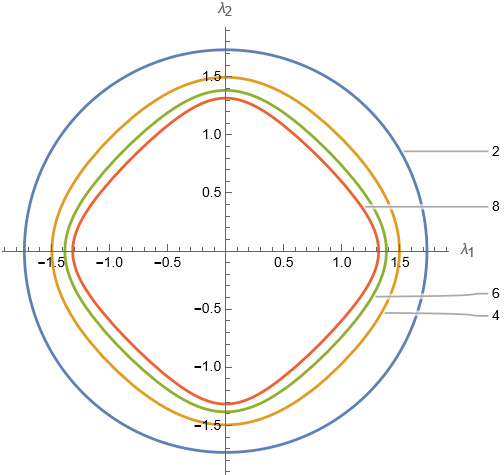}
	\caption{(\textsc{Left}) Unit circles for $\|\cdot\|_{\X,d}$ with $d=1, 2, 3, 4, 20$, in which $X_1$ and $X_2$ are Gamma random variables with $\alpha=\beta=1$. (\textsc{Right}) Unit circles for $\norm{\cdot}_{\X,d}$ with $d=2, 4, 6,8$, in which $X_1$ and $X_2$ are Uniform random variables on $[-1,1]$.}
	\label{Figure:Exponentials}
\end{figure}

\section{Lower and Upper bounds}
\label{Section:LowerandUpperBounds}

In Subsection \ref{Subsec:Rademacher}, it was shown that if $\X=(X_1, X_2, \ldots, X_n)$ is such that the $X_i$ are independent Rademacher random variables, then the Khintchine inequality yields
\begin{equation}\label{eq:Classic}
	\bigg(\E \Big|\sum_{i=1}^n \lambda_iX_i \Big|^2\bigg)^{\!\!1/2} \!\!\leq \bigg(\E \Big|\sum_{i=1}^n \lambda_iX_i\Big|^d\bigg)^{\!\!1/d}\!\!\leq \gamma_d\bigg(\E \Big|\sum_{i=1}^n \lambda_iX_i\Big|^2\bigg)^{\!\!1/2}
\end{equation} 
for all $\lambda_1, \lambda_2, \ldots, \lambda_n\in \R$, where $\gamma_d=\sqrt{2} (\sqrt{\pi})^{-1/d }\Gamma(\frac{d+1}{2})^{1/d}$ 
denote the $d$th moment of a standard normal random variable. 
In general, if $X_1, X_2, \ldots, X_n$ are iid random variables, a comparison of the form \eqref{eq:Classic} is called a \emph{Khintchine-type inequality}. 
Establishing a Khintchine-type inequality here is equivalent to establishing an equivalence between the random vector norms on $\H_n$ of different orders, as in \eqref{eq:NormEquivalence}. Since $\H_n$ is finite dimensional, such inequalities exist.
However, establishing Khintchine-type inequalities with explicit constants is more involved \cite{ENT1, ENT2, Havrilla, LO}.

In this section, we shall seek to establish these types of inequalities in several contexts. More precisely, since our main object of study are the random vector norms, we shall prove lower and upper bounds of different forms for the random vector norms on $\H_n$ and on $\M_n$. Some simple bounds are given in Subsection \ref{subsec:simplebound}, while more involved bounds, which are independent of the dimension $n$ of the matrices, are studied in Subsection \ref{subsec:involvedbound}. We begin by proving this short lemma.

\begin{lemma}\label{lem:generalize}
	Let $d \geq 1$ and $\vec{X} = (X_1,\dots , X_n)$, in which $X_1, \dots , X_n \in L^d(\Omega,\mathcal{F},\vec{P})$ are iid random variables. Let $\|\cdot\|:\M_n\to\R$ be a (submultiplicative) matrix norm, and suppose there exists a constant $k$ (which may depend on $n$ and $d$) such that
	\begin{enumerate}[label=(\roman*)]
		\itemsep.5em
		\item $\|A\|_{\X,d}\leq k \|A\|$ for all $A\in \H_n$; and
		\item $\|Z\|=\|Z^*\|$ for all $Z\in\M_n$.
	\end{enumerate}
	Then $\cnorm{Z}_{\X,d} \leq 2k \binom{d}{d/2}^{\!-1/d} \|Z\|$ for all $Z\in\M_n$.
\end{lemma}
\begin{proof}
	According to \eqref{eq:NormComplex},
	\begin{align*}
		2\pi \tbinom{d}{d/2}\cnorm{Z}_{\X,d}^d &= \int_0^{2\pi} \!\big\|e^{it}Z + e^{-it} Z^* \big\|_{\X,d}^d \,\mathrm{d}t \leq k^d \!\int_0^{2\pi} \!\big\|e^{it}Z + e^{-it} Z^* \big\|^d \mathrm{d}t \\
		&\leq k^d \!\int_0^{2\pi} \!\big(\big\|e^{it}Z\big\| + \big\|e^{-it} Z^* \big\| \big)^d \mathrm{d}t \leq k^d \!\int_0^{2\pi} \!\big(\|Z\| + \|Z^* \| \big)^d \mathrm{d}t \\
		&\leq 2^d k^d \!\int_0^{2\pi} \!\|Z\|^d \,\mathrm{d}t = 2\pi \cdot 2^d k^d \|Z\|^d. \tag*{\qedhere}
	\end{align*}
\end{proof}

\subsection{Simple bounds}
\label{subsec:simplebound}

Each random vector norm on $\M_n$ is bounded below by an explicit positive multiple of the trace seminorm. That is, the random vector norms of a matrix can be related to the mean of its eigenvalues (not singular values). The following is a generalization of \cite[Thm.~32]{Aguilar}.

\begin{theorem}\label{Theorem:Trace_ineq}
	Let $d \geq 1$ and $\vec{X} = (X_1,X_2,\ldots , X_n)$, where $X_1,X_2, \ldots , X_n \in L^d(\Omega,\mathcal{F},\vec{P})$ are iid random variables of distribution $X$. Then
	\begin{gather*}
		\cnorm{Z}_{\X,d} \geq \frac{\|I_n\|_{\X,d}}{n } |\tr Z| \quad\mathrm{for}~Z\in \M_n,
	\end{gather*}
	with equality if and only if $Z$ is a multiple of the identity.
\end{theorem}
\begin{proof}
	Suppose $A\in \H_n$. Theorem \ref{Theorem:Main2}.(d) shows that the random vector norms are Schur-convex. That is, $\|A\|_{\X,d} \leq \|B\|_{\X,d}$ if $\bvec{\lambda}(A) \prec \bvec{\lambda}(B)$, with equality if and only if $\bvec{\lambda}(A)$ is a permutation of $\bvec{\lambda}(B)$. 
	For $A \in \H_n$, $\bvec{\lambda}(A)$ majorizes
	$\bvec{\mu} = (\mu, \mu,\ldots, \mu)\in \R^n$, in which 
	$\mu=\tr A/n$. Hence,
	\begin{align*}
		\norm{A}_{\X,d}^d &= \E\big[ |\inner{\X,\bvec{\lambda}}|^d \big] \geq \E\big[ |\langle \X,\bvec{\mu} \rangle|^d \big] \\
		&= |\mu|^d \E\big[ |X_1+X_2+\cdots+X_n|^d \big] \\
		&= |\mu|^d \|I_n\|_{\X,d}^d,
	\end{align*}
	with equality if and only if $A=\mu I_n$. Now, let $Z\in \M_n$ and define the matrix $B(t)=e^{it} Z+e^{-it} Z^* \in \H_n$ for $t \in \R$. 
	Then it follows from \eqref{eq:NormComplex} and the above that
	\begin{align}
		\cnorm{Z}_{\X,d}
		&\geq  \frac{\|I_n\|_{\X,d}}{n}\Bigg(\frac{ 1}{ 2\pi \binom{d}{d/2}} \int_0^{2\pi} |\tr B(t)|^{d}\,\mathrm{d}t\Bigg)^{\!\!1/d}\!. \label{eq:Admu}
	\end{align}
	Combining this and the following, where $\tr Z = |\tr Z| e^{i\phi}$,
	\begin{align*}
		\int_0^{2\pi} |\tr B(t)|^d \,\mathrm{d}t
		&=  \int_0^{2\pi} \big| e^{it} \tr Z + e^{-it} \tr Z^* \big|^d\,\mathrm{d}t \\[-.4pt]
		&=  |\tr Z|^d \int_0^{2\pi} \big| e^{it+i\phi} + e^{-it-i\phi} \big|^d\,\mathrm{d}t \\[-.4pt]
		&=  2^d|\tr Z|^d \int_0^{2\pi}  \big|\sin(t+\phi+\tfrac{\pi}{2})\big|^d\,\mathrm{d}t \\[-.4pt]
		&=  2^d|\tr Z|^d \int_0^{2\pi}  \big|\sin(t)\big|^d\,\mathrm{d}t \\[-.4pt]
		&=  2\cdot 2^d|\tr Z|^d \int_0^{\pi}  \sin^d(t) \,\mathrm{d}t \\[-.4pt]
		&=2\pi {\binom{d}{d/2}} |\tr Z|^d
	\end{align*}
	yield the desired inequality. Note that the fourth identity holds because of the $2\pi$-periodicity of $\sin(t)$, while the last is \cite[Rem.~5.2]{Bouthat2}. 
	Moreover, the continuity of the integrand ensures that equality occurs in \eqref{eq:Admu} if and only if 
	$B(t) = \mu(t) I_n$ for all $t \in \R$.  An operator-valued Fourier expansion of $\mu(t)$ then reveals that
	$e^{it} Z+e^{-it} Z^* = ( \sum_{n \in \Z} \widehat{\mu}(n) e^{int} ) I_n$,
	so $Z = \widehat{\mu}(1) I_n$. Conversely, equality holds in \eqref{eq:Admu} 
	if $Z$ is a multiple of the identity.
\end{proof}

In the following, an upper bound for the random vector norms is established, putting in relation both $\|\cdot\|_{\X,d}$ and $\cnorm{\cdot}_{\X,d}$ with the Schatten norms via a constant dependent on $n,d$ and the $d$th absolute moment of the underlying distribution.

\begin{theorem}
 Let $d \geq 1$ and $\vec{X} = (X_1,\dots , X_n)$, where $X_1, \dots , X_n \in L^d(\Omega,\mathcal{F},\vec{P})$ are iid random variables of distribution $X$. Then
 \vspace{-1pt}
 \begin{align*}
 	\|A\|_{\X,d}^d &\leq n^{d-1}\E\big[|X|^d\big] \|A\|_{S_{d}}^d \quad\mathrm{for}~A\in \H_n, \\[5pt]
 	\cnorm{Z}_{\X,d}^d &\leq \frac{2n^{d-1}\E\big[|X|^d\big]}{\binom{d}{d/2}} \|Z\|_{S_{d}}^d \quad\mathrm{for}~Z\in \M_n.\\[-19pt]
 \end{align*}
\end{theorem}
\begin{proof}
  Since $x^d$ is a convex function, the finite form of Jensen's inequality yields
  \vspace{-2pt}
  \begin{equation*}
       |\langle \X,\bvec{\lambda}\rangle |^d \leq n^d \! \left(\frac{|\lambda_1 X_1| + \cdots + |\lambda_n X_n|}{n}\right)^{\!d} \!\leq n^d \frac{|\lambda_1 X_1|^d + \cdots + |\lambda_n X_n|^d}{n}.\vspace{-2pt}
   \end{equation*}
  Therefore,
  \vspace{-5pt}
  \begin{align*}
       \E\big[|\langle \X,\bvec{\lambda}\rangle |^d\big] &\leq  n^{d-1} \!\Big( \E\big[|\lambda_1 X_1|^d\big] + \cdots + \E\big[|\lambda_n X_n|^d\big] \Big) \\
       &= n^{d-1} \!\left( |\lambda_1|^d \E\big[|X_1|^d\big] + \cdots + |\lambda_n|^d \E\big[|X_n|^d\big] \right) \\
       &=\, n^{d-1} \!\left( |\lambda_1|^d  + \cdots + |\lambda_n|^d \right) \E\big[|X|^d\big] \\
       &=\, n^{d-1} \|A\|_{S_d}^d \E\big[|X|^d\big] .\\[-19pt]
   \end{align*}
   The general cases immediately follow from Lemma \ref{lem:generalize}.
\end{proof}

If the distribution of the random variables $X_i$ is symmetric (that is, $\E [X_i] =0$), the Marcinkiewicz--Zygmund inequality \eqref{Eq:M-Z} yields a pair of inequalities on the random vector norms of the form $c(X,d) \|A\|_{\operatorname{F}}$, where $\|\cdot\|_{\operatorname{F}}$ is the Frobenius norm and $c(X,d)$ is some constant independant of $n$. 
The following is \cite[Lem.~1]{Bouthat1}.

\begin{theorem}\label{ineq - main}
	Let $d\geq 2$ and $\X = (X_1, X_2, \ldots , X_n)$, where $X_1,X_2,  \ldots , X_n \in L^d(\Omega,\mathcal{F},\vec{P})$ are iid random variables of distribution $X$ satisfying $\E[X]=0$. Then
	\vspace{-3pt}
	\begin{equation*}
		A_d^{\frac{1}{d}} \E\hspace{-.3pt}\big[|X|^2\big]^{\frac{1}{2}} \|A\|_{\operatorname{F}} \leq \|A\|_{\X,d}  \leq \big(B_d \E\hspace{-.3pt}\big[|X|^d\big]\big)^{\frac{1}{d}}  \|A\|_{\operatorname{F}} ,
	\end{equation*}
	where $A_d$ and $B_d$ are the constants in the Marcinkiewicz--Zygmund inequality which depend only on $d$.
\end{theorem}

\begin{proof}
	The Marcinkiewicz--Zygmund inequality implies that
	\begin{equation*}
	A_d \E \!\left[ \bigg( \sum_{i=1}^n |\lambda_i X_i|^2 \bigg)^{\!\!d/2} \right] \leq \E\!\left[|\langle \X,\bvec{\lambda}\rangle |^d \right] \leq B_d \E \!\left[ \bigg( \sum_{i=1}^n |\lambda_i X_i|^2 \bigg)^{\!\!d/2} \right]. \vspace{-3pt}
	\end{equation*}
	To establish the desired estimates, it suffices to bound $\E \hspace{-.3pt}\big[ \big( \sum_{i=1}^n |\lambda_i X_i|^2 \big)^{\!d/2} \big]$. To obtain the upper bound, observe that the finite form of Jensen's inequality and the convexity of $x^{\smash{d/2}}$ ensures that
	\begin{align*}
		\E \!\left[ \bigg( \sum_{i=1}^n |\lambda_i X_i|^2 \bigg)^{\!\!d/2} \right] \!&=\! \E \!\left[ \bigg( \frac{\lambda_1^2 |X_1|^2+\cdots+ \lambda_n^2 |X_n|^2}{\lambda_1^2 + \cdots + \lambda_n^2}\bigg)^{\!\!d/2} \right] \!\cdot (\lambda_1^2 + \cdots + \lambda_n^2)^{\frac{d}{2}} \\
		&\leq \E \!\left[ \frac{ \lambda_1^2 |X_1|^d+\cdots+ \lambda_n^2 |X_n|^d}{\lambda_1^2 + \cdots + \lambda_n^2} \right] \!\cdot (\lambda_1^2 + \cdots + \lambda_n^2)^{\frac{d}{2}} \\
		&= \E\hspace{-.3pt}\big[|X|^d\big] \|A\|_{\operatorname{F}}^d\,.
	\end{align*}
	Hence, $\E\hspace{-.3pt}\big[|\langle \X,\bvec{\lambda}\rangle |^d \big] \leq B_d \E\hspace{-.3pt}\big[|X|^d\big] \|A\|_{\operatorname{F}}^d$. 
	To establish the lower bound, simply apply the probabilistic form of Jensen's inequality:
	\vspace{-3pt}
	\begin{align*}
		\E \!\left[ \bigg( \sum_{i=1}^n |\lambda_i X_i|^2 \bigg)^{\!\!d/2} \right] \!\geq\!  \bigg( \sum_{i=1}^n \lambda_i^2 \E|X_i|^2 \bigg)^{\!\!\frac{d}{2}} \!\!=\!  \bigg( \E|X|^2 \sum_{i=1}^n \lambda_i^2 \bigg)^{\!\!\frac{d}{2}} \!\!=  \E\hspace{-.3pt}\big[|X|^2\big]^{\frac{d}{2}} \|A\|_{\operatorname{F}}^d \,,
	\end{align*}
	and the result follows.
\end{proof}

\subsection{Dimension-free bounds}
\label{subsec:involvedbound}

In several situations, it is preferable to obtain dimension-free bounds, that is, where the constants do not rely on the dimension of the matrices (see Section \ref{Section:Submultiplicativity}). In this section, we establish a pair of inequalities between $\cnorm{\cdot}_{\X,d}$ and $\cnorm{\cdot}_{\X,2}$. Recall that the $d$th \emph{standardized absolute moment} of the random variable $X$ is defined as $\Tilde{\mu}_d = \frac{\E[|X-\mu|^d ]}{\sigma^d}$, where $\sigma$ is the standard deviation and $\mu$ is the mean of $X$ (see Subsection \ref{Subsection:Probability}). The following propositions are \cite[Prop.~4]{Bouthat1} and \cite[Prop.~5]{Bouthat1}.

\begin{proposition}\label{prop - main_ineq}
	Let $d\geq 2$ and $\X = (X_1,X_2, \ldots , X_n)$, where $X_1,X_2, \ldots , X_n \in L^d(\Omega,\mathcal{F},\vec{P})$ are iid random variables of $d$th standardized absolute moment $\tilde{\mu}_d$. Then
	\vspace{-4pt}
	\begin{equation*} 
	\sqrt{2} \binom{d}{d/2}^{\!\!-\frac{1}{d}} \cnorm{Z}_{\X,2}
	\leq \cnorm{Z}_{\X,d} \leq 4 \Bigg(\frac{  B_d \Tilde{\mu}_d }{2 \binom{d}{d/2}}\Bigg)^{\!\!\frac{1}{d}}  \cnorm{Z}_{\X,2} ,
	\end{equation*}
	where $B_d$ is the constant in the Marcinkiewicz--Zygmund inequality \eqref{Eq:M-Z}.
\end{proposition}
\begin{proof}
	The lower bound is Theorem \ref{Theorem:Main2}.(c). To prove the upper bound, first suppose that $A\in\H_n$, and observe that Jensen's inequality ensures that
	\begin{align*}
		|\langle \X,\bvec{\lambda} \rangle |^d =  2^d \bigg| \frac{\langle \X-\mu,\bvec{\lambda} \rangle + \mu \tr A}{2} \bigg|^d \!\leq 2^d  \frac{|\langle \X-\mu,\bvec{\lambda} \rangle|^d + |\mu \tr A|^d}{2}.
	\end{align*}
	Since $\E[X-\mu]=0$, it follows from Lemma \ref{ineq - main} that
	\begin{align*}
		\E |\langle \X,\bvec{\lambda} \rangle |^d  &\leq 2^{d-1} \!\left(  \E|\langle \X-\mu,\bvec{\lambda} \rangle|^d  + |\mu \tr A|^d \right) \\
		&\leq 2^{d-1}  \!\left( B_d \tilde{\mu}_d \sigma^d \|A\|_{\operatorname{F}}^d + |\mu \tr A|^d \right),
	\end{align*}
	where $\tilde{\mu}_d \sigma^d= \E[|X-\mu|^d ]$. An application of Lyapunov's inequality \eqref{eq:Lyapunov} reveals that $\tilde{\mu}_d \sigma^d= \E |X-\mu|^d \geq \E[ |X-\mu|^2]^{d/2} = \sigma^d$. Moreover, according to Remark \ref{rem - constant}, $B_d \geq 1$ for any $d\geq 2$. Therefore, $B_d \Tilde{\mu}_d  \geq 1$ and thus
	\begin{align*}
		\E\hspace{-.3pt}\big[ |\langle \X,\bvec{\lambda} \rangle |^d \big] &\leq  
		2^{d-1} \!\left( B_d \Tilde{\mu}_d \sigma^d \|A\|_{\operatorname{F}}^d + |\mu \tr A|^d \right) \\
		&\leq 2^{d-1} \!\left( B_d \Tilde{\mu}_d \sigma^d \|A\|_{\operatorname{F}}^d + B_d\Tilde{\mu}_d|\mu \tr A|^d \right) \\
		&= 2^{d-1} B_d \Tilde{\mu}_d  \!\left( \sigma^d \|A\|_{\operatorname{F}}^d + |\mu \tr A|^d \right).
	\end{align*}
	It then follows from the classical vector $p$-norm inequality that 
	\begin{align}\label{eq - 3}
		\E\hspace{-.3pt}\big[ |\langle \X,\bvec{\lambda} \rangle |^d \big] &\leq  2^{d-1} B_d \Tilde{\mu}_d  \Big( \sigma^d \|A\|_{\operatorname{F}}^d + |\mu \tr A|^d \Big) \nonumber\\[-2pt]
		&\leq 2^{d-1} B_d \Tilde{\mu}_d \Big( \sigma^2 \|A\|_{\operatorname{F}}^2 + \mu^2 |\tr A|^2 \Big)^{\!\smash{\frac{d}{2}}} \\
		&= 2^{\smash{d-1}} B_d \Tilde{\mu}_d \cnorm{A}_{\X,2}^d, \nonumber
	\end{align}
	where the last identity stems from Example \ref{Ex:d=2}. 
	Consequently, 
	\begin{equation*}
		\|A\|_{\X,d}^d = \E\hspace{-.3pt}\big[|\langle \X,\bvec{\lambda} \rangle | ^d  \leq 2^{d-1} B_d \Tilde{\mu}_d  \cnorm{A}_{\X,2}^{d}
	\end{equation*}
	for any $A\in \H_n$. The conclusion then follows from Lemma \ref{lem:generalize}.
	\end{proof}

The following result addresses the case $1\leq d \leq 2$. The proof closely follows what was done in Proposition \ref{prop - main_ineq}. However, since $d\leq 2$, several of the arguments are no longer valid, and more technical inequalities are required.

\begin{proposition}\label{prop - 1-2}
	Let $1\leq d\leq 2$ and $\eta>2$, and let $\X = (X_1, X_2, \dots , X_n)$, where $X_1, X_2, \dots , X_n \in L^{\eta}(\Omega,\mathcal{F},\vec{P})$ are iid random variables of $d$th standardized absolute moment $\tilde{\mu}_d$. Then
	\begin{equation*} 
	4\Bigg( \frac{\left( 2B_{\eta} \Tilde{\mu}_{\eta}  \right)^{\!\smash{\frac{d-2}{\eta-2}}}}{8  \binom{d}{d/2}} \Bigg)^{\!\!\frac{1}{d}} \cnorm{Z}_{\X,2} 
	\leq \cnorm{Z}_{\X,d} 
	\leq  \sqrt{2}  \binom{d}{d/2}^{\!\!-\frac{1}{d}}  \cnorm{Z}_{\X,2},
	\end{equation*}
	where $B_{\eta}$ is the constant in the Marcinkiewicz--Zygmund inequality \eqref{Eq:M-Z}.
\end{proposition}
\begin{proof}
	The upper bound is Theorem \ref{Theorem:Main2}.(c). If $Z=0$, the lower bound is trivial. Hence, suppose that $Z\neq 0$. In that case, Hölder's inequality ensures that for $A\in\H_n$,
	\begin{align*}
		\E\big[|XY|\big] \leq \E\hspace{-.3pt}\big[|X|^p\big]^{\smash{\frac{1}{p}}}\, \E\hspace{-.3pt}\big[|Y|^q\big]^{\smash{\frac{1}{q}}}.
	\end{align*}
	Setting $X= |\langle \X,\bvec{\lambda} \rangle|^{\smash{\frac{\eta-2 }{\eta-d} d}}$, $Y= |\langle \X,\bvec{\lambda} \rangle|^{\smash{\frac{2-d}{\eta-d}\eta}}$, and $p=\frac{\eta-d}{\eta-2}$ yields
	\begin{align}\label{Eq:idk}
		\E\hspace{-.3pt}\big[|\langle \X,\bvec{\lambda} \rangle|^2\big]^{\smash{\frac{\eta-d}{\eta-2}}} \!\leq \E\hspace{-.3pt}\big[|\langle \X,\bvec{\lambda} \rangle|^d\big] \,\E\hspace{-.3pt}\big[|\langle \X,\bvec{\lambda} \rangle|^{\eta}\big]^{\smash{\frac{2-d}{\eta-2}}} .
	\end{align}
	Recall that \eqref{eq - 3} states that $\E\hspace{-.3pt}\big[ |\langle \X,\bvec{\lambda} \rangle |^{\eta} \big] \leq  2^{\eta-1} B_{\eta} \Tilde{\mu}_{\eta} \E\hspace{-.3pt}\big[ |\langle \X,\bvec{\lambda} \rangle |^2 \big]^{\smash{\frac{\eta}{2}}}$. 
	It then follows from \eqref{Eq:idk} and the above that
	\begin{equation*}
	\E\hspace{-.3pt}\big[|\langle \X,\bvec{\lambda} \rangle|^2\big]^{\smash{\frac{d}{2}}} \leq 
	\big(2^{\eta-1} B_{\eta} \Tilde{\mu}_{\eta}  \big)^{\!\smash{\frac{2-d}{\eta-2}}}   \E\hspace{-.3pt}\big[|\langle \X,\bvec{\lambda} \rangle|^d\big],
	\end{equation*}
	which implies that
	\begin{equation*}
	2^{\smash{d-2}} \!\left( 2B_{\eta} \Tilde{\mu}_{\eta}  \right)^{\!\smash{\frac{d-2}{\eta-2}}} \|A\|_{\X,2}^{d} \leq \|A\|_{\X,d}^d ~\quad\text{for any $A\in\H_n$.}
	\end{equation*}
	Hence, it follows that for $Z\in \M_n$,
	\begin{align*}
		\cnorm{Z}_{\X,d}^d &= \frac{1}{2\pi \binom{d}{d/2}} \int_0^{2\pi} \big\|e^{it}Z + e^{-it} Z^* \big\|_{\X,d}^d \,\mathrm{d}t \\
		&\geq \frac{2^{\smash{d}} \!\left( 2B_{\eta} \Tilde{\mu}_{\eta}  \right)^{\!\smash{\frac{d-2}{\eta-2}}}}{8\pi  \binom{d}{d/2}} \int_0^{2\pi} \big\|e^{it}Z + e^{-it} Z^* \big\|_{\X,2}^d \,\mathrm{d}t .
	\end{align*}
	In this case, Jensen's inequality cannot be useful. Instead, observe that
	\begin{align*}
		\big\|e^{it}Z + e^{-it} Z^* \big\|_{\X,2}^d &= \frac{\|e^{it}Z + e^{-it} Z^* \|_{\X,2}^2}{\cnorm{e^{it}Z + e^{-it} Z^*}_{\X,2}^{2-d}} \geq \frac{\|e^{it}Z + e^{-it} Z^* \|_{\X,2}^2}{\big(\cnorm{e^{it}Z}_{\X,2} + \cnorm{e^{-it} Z^*}_{\X,2}\big)^{2-d}} \\
		&= \frac{\big\|e^{it}Z + e^{-it} Z^* \big\|_{\X,2}^2}{2^{2-d} \cnorm{Z}_{\X,2}^{2-d}}.
	\end{align*}
	It follows that
	\begin{align*}
		\cnorm{Z}_{\X,d}^d &\geq \frac{2^{\smash{d}} \!\left( 2B_{\eta} \Tilde{\mu}_{\eta}  \right)^{\!\smash{\frac{d-2}{\eta-2}}}}{8\pi  \binom{d}{d/2}} \int_0^{2\pi} \big\|e^{it}Z + e^{-it} Z^* \big\|_{\X,2}^d \,\mathrm{d}t \\
		&\geq \frac{2^{\smash{2d}} \!\left( 2B_{\eta} \Tilde{\mu}_{\eta}  \right)^{\!\smash{\frac{d-2}{\eta-2}}}}{8  \binom{d}{d/2}  \cnorm{Z}_{\X,2}^{2-d}} \cdot \frac{1}{4\pi} \int_0^{2\pi} \big\|e^{it}Z + e^{-it} Z^* \big\|_{\X,2}^2 \,\mathrm{d}t \\[2pt]
		&= \frac{4^d \!\left( 2B_{\eta} \Tilde{\mu}_{\eta}  \right)^{\!\smash{\frac{d-2}{\eta-2}}}}{8  \binom{d}{d/2}} \cnorm{Z}_{\X,2}^d. \tag*{\qedhere}
	\end{align*}
\end{proof}

\section{Submultiplicativity of random vector norms}\label{Section:Submultiplicativity}

Recall that a \emph{matrix norm} is a norm on $\M_n$ which is also submultiplicative on $\M_n$. Theorem \ref{Theorem:Main2}.(a) establishes that $\norm{\cdot}_{\X,d}$ is a weakly unitarily invariant norm on $\H_n$. Proposition \ref{p:gennorm} further establishes that $\cnorm{\cdot}_{\X,d}$ is a norm on $\M_n$. However, neither of them are necessarily submultiplicative. 
For instance, according to Example \ref{Ex:d=2}, the random vector norm of order $2$ induced by a random variable of mean and standard variation $1$ of the matrices $\left[\begin{smallmatrix}
	0&1\\1&0
\end{smallmatrix}\right]$ and $I_2$ are equal to $1$ and $\sqrt{3}$ respectively. Therefore, this norm is not submultiplicative since
\begin{equation*}
\|K^2\|_{\X,2} = \|I_2\|_{\X,2}=\sqrt{3} > 1^2= \|K\|_{\X,2}^2.
\end{equation*} 
In general, if $\|\cdot\|$ is a norm on $M_n$, then there is always a scalar multiple of it (which may depend upon the dimension $n$) which is submultiplicative \cite[Thm.~5.7.11]{HJ}. Indeed, consider $\gamma=\max_{\|Z_1\|=1=\|Z_2\|} \|Z_1 Z_2\|$. Then
\begin{align*}
	\gamma \|Z_1 Z_2\| = \|Z_1\| \|Z_2\| \cdot \gamma \left\| \frac{Z_1}{\|Z_1\|} \frac{Z_2}{\|Z_2\|} \right\| \leq  \gamma\|Z_1\| \cdot \gamma \|Z_2\|.
\end{align*}
If such a scalar is also independent of $n$, then one can simply scale the underlying distribution of the entries of $\X$ by an appropriate constant to obtain a single matrix norm for every dimension $n$, that is, $\gamma\cnorm{\cdot}_{\X,d} = \cnorm{\cdot}_{\gamma\X,d}$ if $\gamma>0$. This is not possible if the constant depends on $n$. Note that, for $A,B\in \H_n$, the product $AB$ might not even be Hermitian. Hence, the matter of submultiplicativity of the random vector norms only make sense for $\cnorm{\cdot}_{\X,d}$. 

In this section, we establish the following result, which is \cite[Thm.~2]{Bouthat1}.
\begin{theorem}\label{thm - main}
	Let $d\geq 1$ and $\X = (X_1, X_2,\ldots , X_n)$, where $X_1, X_2,\ldots , X_n \in L^p(\Omega,\mathcal{F},\vec{P})$, in which $p=\max\{d,\eta\}$ for some $\eta>2$. Then there exists a constant $\gamma_d>0$, independent of $n$, such that $\gamma_d\cnorm{Z}_{\X,d}$ is a (submultiplicative) matrix norm on $M_n$.
\end{theorem}

\subsection{The case \texorpdfstring{$d=2$}{d=2}}
	\label{sec - main}
	
Recall that when $d=2$, Example \ref{Ex:d=2} reveals that
\begin{equation*}
\cnorm{Z}_{\X,2}^2 \,=\,  \sigma^2 \|Z\|_{\operatorname{F}}^2 + \mu^2 |\tr Z|^2.
\end{equation*}
This elegant formulation makes it possible to show in the proposition below that $\cnorm{Z}_{\X,2}$ is a submultiplicative norm when multiplied by a constant only dependent on the mean $\mu$ and the standard deviation $\sigma$ of $X_i$. 
The following is \cite[Prop.~6]{Bouthat1}.

\begin{proposition}\label{prop - d=2}
Let $d=2$ and $\X = (X_1,X_2,\ldots , X_n)$ where $X_1,X_2,\ldots , X_n \in L^d(\Omega,\mathcal{F},\vec{P})$ are iid random variables. Then there exists a constant $\gamma>0$, independent of $n$, such that $\gamma\cnorm{Z}_{\X,2}$ is a matrix norm on $M_n$.
\end{proposition}
\vspace{-1pt}
\begin{proof}
Let $\gamma^2= \frac{\sigma^2+\mu^2}{\sigma^4}$ (which is independent from $n$). Then
\begin{align*}\label{eq - ineq}
	\gamma \cnorm{Z}_{\X,2} &= \frac{\sqrt{\sigma^2+\mu^2}}{\sigma^2} \sqrt{\sigma^2 \|Z\|_{\operatorname{F}}^2 +  \mu^2 |\tr Z|^2} \geq \frac{\sqrt{\sigma^2+\mu^2}}{\sigma^2} \sqrt{\sigma^2 \|Z\|_{\operatorname{F}}^2} \\
	&= \frac{\sqrt{\sigma^2+\mu^2}}{\sigma}  \|Z\|_{\operatorname{F}}. \nonumber
\end{align*}
Cauchy--Schwarz and the submultiplicativity of the Frobenius norm implies that
\begin{align}
	\gamma \cnorm{Z_1 Z_2}_{\X,2} &= \frac{\sqrt{\sigma^2+\mu^2}}{\sigma^2} \sqrt{\sigma^2 \|Z_1 Z_2\|_{\operatorname{F}}^2 +  \mu^2 |\tr Z_1 Z_2|^2}\\
	&\leq \frac{\sqrt{\sigma^2+\mu^2}}{\sigma^2} \sqrt{\sigma^2 \|Z_1\|_{\operatorname{F}}^2 \|Z_2\|_{\operatorname{F}}^2 +  \mu^2 \|Z_1\|_{\operatorname{F}}^2 \|Z_2\|_{\operatorname{F}}^2}\\
	&= \frac{\sqrt{\sigma^2+\mu^2}}{\sigma^2} \sqrt{\sigma^2 +  \mu^2} \,\|Z_1\|_{\operatorname{F}} \|Z_2\|_{\operatorname{F}}\\
	&= \bigg(\frac{\sqrt{\sigma^2+\mu^2}}{\sigma} \,\|Z_1\|_{\operatorname{F}} \bigg) \bigg(\frac{\sqrt{\sigma^2+\mu^2}}{\sigma} \,\|Z_2\|_{\operatorname{F}} \bigg)\\
	&\leq \big( \gamma \cnorm{Z_1}_{\X,2} \big) \big( \gamma \cnorm{Z_2}_{\X,2} \big).\tag*{\qedhere}
\end{align}
\end{proof}

\begin{remark}\label{Remark:d=2}
The constant $\gamma^2= \frac{\sigma^2+\mu^2}{\sigma^4}$ is the smallest constant such that $\gamma \cnorm{\cdot}_{\X,2}$ is submultiplicative, if independence from $n$ is required. Indeed, consider the matrix
\begin{equation*}
A_n = J_n - I_n,
\end{equation*}
where $J_n$ is the $n\times n$ all-ones matrix and $I_n$ is the $n\times n$ identity matrix.
Then
\begin{equation*}
\gamma \cnorm{A_n}_{\X,2} = \frac{\sqrt{\sigma^2+\mu^2}}{\sigma} \sqrt{n^2-n}
\end{equation*}
and
\begin{equation*}
\gamma \cnorm{A_n^2}_{\X,2} = \frac{\sqrt{\sigma^{2}+\sigma^{2}}}{\sigma^{2}}\sqrt{\sigma^{2}n\left(n-1\right)\left(n^{2}-3n+3\right)+\mu^{2}n^{2}\left(n-1\right)^{2}}.
\end{equation*}
It follows that
\begin{align*}
	\frac{\gamma \cnorm{A_n^2}_{\X,2}}{\gamma \cnorm{A_n}_{\X,2} \cdot \gamma \cnorm{A_n}_{\X,2}} = \sqrt{1-\frac{\sigma^{2}}{\sigma^{2}+\mu^{2}}\frac{2n-3}{n(n-1)}} \,\xrightarrow{n\to\infty}\, 1,
\end{align*}
As a direct consequence, the norm $\cnorm{\cdot}_{\X,2}$ is a matrix norm for all $n\geq 1$ if and only if $2\sigma^2 \geq 1+\sqrt{1+4\mu^2}$.
\end{remark}

\subsection{Proof of Theorem \ref{thm - main}}

Let us first prove the following short lemma.

\begin{lemma}\label{thm - useful}
	Let $N(\cdot)$ be a norm on $M_n$ and let $\|\cdot\|$ be a matrix norm on $M_n$. If $C_m$ and $C_M$ are positive constants such that
	\begin{equation*}
		C_m \|Z\| \leq N(Z) \leq C_M \|Z\|\quad \text{for all } A\in M_n,
	\end{equation*}
	then $(C_M/C_m^2 )N(\cdot)$ is a matrix norm.
\end{lemma}
\begin{proof}
	Let $\gamma=C_M/C_m^2$. A direct computation yields
	\begin{align}
		\gamma N(Z_1 Z_2) &= \frac{C_M}{C_m^2} N(Z_1 Z_2) \leq \frac{C_M}{C_m^2} C_M \|Z_1 Z_2\| \leq \frac{C_M^2}{C_m^2} \|Z_1\| \|Z_2\| \nonumber\\
		&\leq \frac{C_M^2}{C_m^2} C_m^{-1} N(Z_1) C_m^{-1} N(Z_2) = \gamma N(Z_1) \, \gamma N(Z_2). \tag*{\qedhere}
	\end{align}
\end{proof}

Now let $\gamma^2= \frac{\sigma^2+\mu^2}{\sigma^4}$, as defined in Proposition \ref{prop - d=2}, and suppose that $d\geq 2$. Proposition \ref{prop - main_ineq} implies that 
\begin{equation*}
\sqrt{2} \binom{d}{d/2}^{\!\!-\frac{1}{d}} \gamma\cnorm{Z}_{\X,2}
\leq \gamma\cnorm{Z}_{\X,d} \leq 4 \Bigg(\frac{  B_d \Tilde{\mu}_d }{2 \binom{d}{d/2}}\Bigg)^{\!\!\frac{1}{d}}  \gamma\cnorm{Z}_{\X,2} .
\end{equation*}
Since $\gamma\cnorm{\cdot}_{\X,2}$ is submultiplicative by Proposition \ref{prop - d=2}, Lemma \ref{thm - useful} ensures that $\gamma_d \cnorm{Z}_{\X,d}$ is submultiplicative, where
\begin{align*}
	\gamma_d &= 4 \Bigg(\frac{  B_d \Tilde{\mu}_d }{2 \binom{d}{d/2}}\Bigg)^{\!\!\frac{1}{d}}  \!\!\Bigg/\!   2 \binom{d}{d/2}^{\!\!-\frac{2}{d}} 
	\!= 2 \Big(\tfrac{1}{2}  B_d \Tilde{\mu}_d \tbinom{d}{d/2}\Big)^{\!\frac{1}{d}}\!. 
\end{align*}
If $1\leq d \leq 2$, then Proposition \ref{prop - 1-2} ensures that 
\begin{equation*} 
4\Bigg( \frac{\left( 2B_{\eta} \Tilde{\mu}_{\eta}  \right)^{\!\smash{\frac{d-2}{\eta-2}}}}{8  \binom{d}{d/2}} \Bigg)^{\!\!\frac{1}{d}} \gamma\cnorm{Z}_{\X,2} 
\leq \gamma\cnorm{Z}_{\X,d} 
\leq  \sqrt{2}  \binom{d}{d/2}^{\!\!-\frac{1}{d}} \gamma \cnorm{Z}_{\X,2}.
\end{equation*}
Once again, since $\gamma\cnorm{\cdot}_{\X,2}$ is submultiplicative, Lemma \ref{thm - useful} ensures that $\gamma_d \cnorm{Z}_{\X,d}$ is submultiplicative, where
\begin{align*}
	\gamma_d = \sqrt{2}  \binom{d}{d/2}^{\!\!-\frac{1}{d}}     \!\!\Bigg/\!   16\Bigg( \frac{\left( 2B_{\eta} \Tilde{\mu}_{\eta}  \right)^{\!\smash{\frac{d-2}{\eta-2}}}}{8  \binom{d}{d/2}} \Bigg)^{\!\!\frac{2}{d}}
	\!= \frac{\sqrt{2}}{16}  \Big( 64  \tbinom{d}{d/2} \!\left( 2B_{\eta} \Tilde{\mu}_{\eta}  \right)^{\!\smash{2\frac{2-d}{\eta-2}}} \Big)^{\!\frac{1}{d}} .
\end{align*}
Since $\gamma_d$ does not depend on $n$ for any $d\geq 1$, we are done. $\hfill \qed$

\smallskip
\begin{remark}
	In Subsection \ref{Subsection:a-stable}, it was shown if the $X_i$ follow a symmetric $\alpha$-stable distribution with scale parameter $\gamma$, then the induced random vector norm is
	\vspace{-6pt}
	\begin{equation}
		\|A\|_{\X,d} = \gamma \!\left( \frac{2 \sin\!\big( \tfrac{d\pi}{2} \big) \Gamma(d+1)}{\alpha\sin\!\big(\tfrac{ d\pi}{\alpha}\big)\Gamma\big(\tfrac{d}{\alpha}+1\big)} \right)^{\!\!\frac{1}{d}}\! \|A\|_{S_\alpha} \quad \text{for $A\in \H_n$}
	\end{equation}
	(see \eqref{Eq:alpha}). Now, applying the techniques of the proof of Proposition \ref{prop - main_ineq}, we obtain
	\vspace{-3pt}
	\begin{equation*}
		\pi^{d-1}   \|Z\|_{S_\alpha}^d \leq \frac{\alpha \binom{d}{d/2} \sin\!\big(\tfrac{ d\pi}{\alpha}\big)\Gamma\big(\tfrac{d}{\alpha}+1\big)}{(2\gamma)^d \sin\!\big( \tfrac{d\pi}{2} \big) \Gamma(d+1)} \cnorm{Z}_{\X,d}^d \leq  2 \|Z\|_{S_\alpha}^d.
	\end{equation*}
	If $\gamma$ is such that $\gamma^d \geq  \smash{ \alpha(2\pi^2)^{1-d} \tbinom{d}{d/2}  \frac{\sin(\frac{d\pi}{\alpha}) \Gamma(d/\alpha+1)}{\sin(\frac{d\pi}{2}) \Gamma(d+1)} }$, then it follows from Lemma \ref{thm - useful} that $\cnorm{Z}_{\X,d}$ is submultiplicative. 
	
	However, for $\alpha\in(1,2)$, we have $\E[|X_i|^d] < \infty$ if and only if $d\in (-1,\alpha)$ \cite[p.~108]{Nolan}. Consequently, $X_i\not\in L^{\eta}(\Omega,\mathcal{F},\vec{P})$ for any $\eta>2$. Since $\cnorm{Z}_{\X,d}$ is submultiplicative, this shows that the assumption $X\in L^{\smash{\eta}}(\Omega,\mathcal{F},\vec{P})$ in Theorem \ref{thm - main} is not necessary.
\end{remark}

\section{CHS norms}
\label{Sec:CMS}

Most of the material in this section comes from \cite{Aguilar}.
In the following, $d\geq 2$ is an even integer. The random vector norms arising from the standard exponential distribution, which were computed in Subsection \ref{Subsection:Gamma}, have a very interesting form: $\norm{A}_{\X,d}^d = d!\, h_d(\lambda_1, \lambda_2, \ldots, \lambda_n)$ (see \eqref{Eq:CHS_noproof}). Let us show this identity in more detail. Let $\alpha=\beta=1$ in \eqref{eq:Gamma} and \eqref{eq:GammaMGF}, which corresponds to a standard exponential distribution. Then Theorem \ref{Theorem:Main2}.(e) ensures that
\begin{equation}\label{eq:CHS1}
	\norm{A}_{\X,d}^d
	= [t^d] \prod_{i=1}^n\frac{d!}{1-\lambda_it}=[t^d] \frac{d!}{t^np_A(t^{-1})}
	\quad \text{for $A \in \H_n$},
\end{equation} 
which is \cite[Thm.~20]{Aguilar}.
For even $d \geq 2$, \eqref{eq:CHSGenerating} then yields
\begin{align}
	\norm{A}_{\X,d}^d 
	&= [t^d] \prod_{i=1}^n\frac{d!}{1-\lambda_it}
	=d! \cdot [t^d] \sum_{r=0}^{\infty} h_r(\lambda_1, \lambda_2, \ldots, \lambda_n) t^r \\
	&= d! \, h_d(\lambda_1, \lambda_2, \ldots, \lambda_n).
\end{align}
If, instead of considering the standard exponential variables, one consider the gamma distribution with parameters $\alpha=1$ and $\beta=(d!)^{-1/d}$, then one retrieves
\begin{align}\label{eq:CHS2}
	\norm{A}_{\X,d}^d = h_d(\lambda_1, \lambda_2, \ldots, \lambda_n).
\end{align}
For simplicity, because of their strong connection with complete homogeneous symmetric polynomials, we refer to this norm as the \emph{CHS norm} of order $d$ and write
\begin{equation}\label{Theorem:MainCHS}
	\norm{A}_{d}=\Big( h_{d}\big(\lambda_1(A), \lambda_2(A), \ldots, \lambda_n(A)\big)\Big)^{1/d} \quad\text{for $A\in H_n$.}
\end{equation} 

\begin{remark}
	In \cite{CGH,Ours}, the CHS norms were obtained by considering the parameters $\alpha=\beta=1$, corresponding to a standard exponential distribution. However, due to the diverging definition of the random vector norms (see Remark \ref{Rem:redef}), we use the parameters $\alpha=1$ and $\beta=(d!)^{-1/d}$ to simplify the following.
\end{remark}

Parts (a) and (e) of Theorem \ref{Theorem:Main2} ensure that the CHS norms are weakly unitarily invariant on $\H_n$ and Schur-convex relative to the vector of eigenvalues $\bvec{\lambda}$ of $A\in\H_n$. 
Moreover, from \eqref{eq:KappaGamma} with $\alpha=1$ and $\beta=(d!)^{-1/d}$, we have $\kappa_i=(d!)^{-i/d} (i-1)!$. Therefore, if $z_{\bvec{\pi}}=\prod_{i\geq 1}i^{m_i}m_i!$,
\begin{equation*}
	y_{\bvec{\pi}} \kappa_{\bvec{\pi}} =\frac{d! \prod_{i\geq 1} \big[(d!)^{-1/d} (i-1)!\big]^{m_i}}{\prod_{i\geq 1} (i!)^{m_i} m_i!}=\frac{1}{\prod_{i\geq 1}i^{m_i}m_i!} = \frac{1}{z_{\bvec{\pi}}}
\end{equation*} 
for any partition $\bvec{\pi}$. 
If $\TT_{\bvec{\pi}} : \M_{n}\to \R$ is defined as in Theorem \ref{Theorem:Main2}.(g), then the same result extends the CHS norms to the whole space $\M_n$:
\begin{equation}\label{eq:Extended}
	\cnorm{Z}_{d}= \bigg( \sum_{\bvec{\pi} \,\vdash\, d} \frac{ \TT_{\bvec{\pi}}(Z)}{z_{\bvec{\pi}}}\bigg)^{\!\!\smash{1/d}}
	\quad \text{for $Z \in \M_n$}.
\end{equation} 
In particular, when $Z\in \H_n$, the above restricts to $\|Z\|_d$ on $\H_n$. More precisely, Theorem \ref{Theorem:Main2}.(f) and \eqref{eq:DefY} imply that for even $d\geq 2$ and $A \in \H_n$,
\begin{equation}\label{eq:StanleyPowerSum}
	h_d(\lambda_1, \lambda_2, \ldots, \lambda_n)
	= \norm{A}_{d}^d=\sum_{\bvec{\pi}\vdash d} y_{\bvec{\pi}} \kappa_{\bvec{\pi}}p_{\bvec{\pi}}
	=\sum_{\bvec{\pi}\vdash d}\frac{p_{\bvec{\pi}}}{z_{\bvec{\pi}}} ,
\end{equation} 
in which $p_{\bvec{\pi}}$ is given by \eqref{eq:pTrace}.
This recovers the combinatorial representation of even-degree CHS polynomials \cite[Prop.~7.7.6]{StanleyBook2}.

\begin{remark}
	Ahmadi, de Klerk, and Hall consider
	norms on $\R^n$ that arise from multivariate homogeneous polynomials   
	\cite[Thm.~2.1]{AKH}. 
	However, the convexity of the even-degree CHS polynomials is difficult to verify directly, so
	their methods do not seem to apply in the present setting.
\end{remark}

\begin{remark}\label{Remark:TriangleIneq}
	The inequality \eqref{eq:Triangle}, which helps to verify the triangle inequality for the CHS norms, 
	has no positivity assumptions.
	For $p \in \N$, the analogous inequality 
	\begin{equation}\label{eq:WeakTriangle}
		h_p(\vec{x}+\vec{y})^{1/p} \leq h_p(\vec{x})^{1/p}+h_p(\vec{y})^{1/p}
		\qquad \text{for $\vec{x},\vec{y} \in \R_{\geq 0}^n$}
	\end{equation}
	has been rediscovered several times.
	McLeod \cite[p.~211] {McLeod} and Whiteley \cite[p.~49]{Whiteley}
	say it was first conjectured by A.C.~Aitken.  Variants of \eqref{eq:WeakTriangle} can be found in \cite{Sra}.
\end{remark}

\smallskip
	
We now establish several properties of the CHS norms. First, we show in Subsection \ref{Subsection:Hunter} how our derivation of the CHS norms allow one to deduce Hunter's theorem (see Section \ref{Section:Hunter}) and generalize it.
Then, in Subsection \ref{Section:Properties}, it is shown how the CHS norm of a Hermitian matrix can be computed rapidly and exactly
from its characteristic polynomial 
and recursion. 
This leads quickly to a determinantal interpretation of these norms in the general case $Z\in \M_n$. 
Next, we precisely identify those CHS norms induced by inner products in Subsection \ref{Subsection:Inner}. Lastly, we discuss equivalence constants with the operator norm in Subsection \ref{Section:Bounds}.

\subsection{A generalization of Hunter's positivity theorem}\label{Subsection:Hunter}
In Subsection \ref{Subsection:CHSExpectation}, we saw that the positivity of the CHS norms 
recovers Hunter's theorem \cite{Hunter} (although the term \emph{CHS norms} was not used at the time). Further examining the way they arise from the random vector norms associated to the Gamma distribution with $\alpha=1$ and $\beta=(d!)^{-1/d}$ allows one to obtain the following theorem, which generalizes Hunter's theorem \cite{Hunter}, corresponding to the case $\alpha = 1$.

\begin{theorem}\label{Theorem:GeneralHunter}
	For even $d\geq 2$ and $\alpha \in \N$, 
	\begin{equation*}
		H_{d,\alpha}(x_1,x_2, \ldots, x_n)= \sum_{\substack{\bvec{\pi}\vdash d\\ |\bvec{\pi}|\leq\alpha}} 
		c_{\bvec{\pi}} h_{\bvec{\pi}}(x_1,x_2,\ldots,x_n),\vspace{-3pt}
	\end{equation*}
	where 
	the sum runs over all partitions $\bvec{\pi}=(\pi_1,\pi_2,\ldots,\pi_r)$ of $d$, is positive definite on $\R^n$. 
	Here $h_{\bvec{\pi}}=h_{\pi_1}h_{\pi_2}\cdots h_{\pi_r}$ is a product of complete homogeneous symmetric 
	polynomials and
	\vspace{-4pt}
	\begin{equation*}
		c_{\bvec{\pi}} = \frac{\alpha !}{ (\alpha-r)! \prod_{i=1}^r m_i!} = \binom{\alpha}{r} \binom{r}{m_1,m_2,\dots,m_r},
	\end{equation*}
	where $r=|\bvec{\pi}|$ denotes the number of parts in $\bvec{\pi}$
	and $m_i$ is the multiplicity of $i$ in $\bvec{\pi}$.
\end{theorem}

\begin{proof}
	Let $\alpha\in \N$ and define the polynomials $P_{\ell}^{(\alpha)}(x_1, x_2, \ldots, x_{\ell})$ by
	\begin{equation}\label{eq:PPalpha}
		P_0^{(\alpha)}=x_0=1
		\quad \text{and} \quad 
		\Big(1+ \sum_{r=1}^{\infty} x_r t^r \Big)^{\alpha}
		=\sum_{\ell=0}^{\infty}P_{\ell}^{(\alpha)}(x_1, x_2, \ldots, x_{\ell})t^{\ell}.\vspace{-5pt}
	\end{equation} 
	Then
	\begin{equation}\label{eq:PPCombo}
		P_{\ell}^{(\alpha)}( x_1,x_2, \ldots, x_{\ell})
		\,\,=\!\!
		\sum_{\substack{ i_1, i_2, \ldots, i_{\alpha}\leq \ell \\ i_1+i_2+\cdots+i_{\alpha}=\ell}} x_{i_1}x_{i_2}\cdots x_{i_{\alpha}}
		=\sum_{\substack{\bvec{\pi}\vdash \ell \\ |\bvec{\pi}|\leq\alpha}} c_{\bvec{\pi}} x_{\bvec{\pi}}.
	\end{equation}
	Let $\X$ be a random vector whose $n$ components are iid and
	distributed according to \eqref{eq:Gamma} with $\beta = 1$.
	Let $A\in \H_n$ have eigenvalues $x_1,x_2,\ldots,x_n$.
	For even $d \geq 2$,
	\begin{align*}
		\norm{A}_{\X,d}^d
		&\overset{\eqref{eq:GammaMGF}}{=} [t^d] \bigg( \prod_{i=1}^k \frac{1}{1-x_i t} \bigg)^{\!\!\alpha} 
		\overset{\eqref{eq:CHS2}}{=} [t^d] \bigg(1+\sum_{r=1}^{\infty} h_{r}(x_1, x_2, \ldots, x_n) t^{r} \bigg)^{\!\!\alpha} \\
		&\overset{\eqref{eq:PPalpha}}{=} [t^d] \sum_{\ell=0}^{\infty} P_{\ell}^{(\alpha)}(h_1, h_2, \ldots, h_{\ell})t^{\ell} \\
		&\overset{\eqref{eq:PPCombo}}{=} [t^d]\sum_{\ell=0}^{\infty}
		\bigg(
		\sum_{\substack{\bvec{\pi}\vdash \ell \\ |\bvec{\pi}|\leq\alpha}} c_{\bvec{\pi}} h_{\bvec{\pi}}(x_1, x_2, \ldots, x_n) \bigg)t^{\ell}. \\[-21pt]
	\end{align*}
	Hence,
	$
	\sum_{\!\!\substack{\bvec{\pi}\vdash d \\ |\bvec{\pi}|\leq \alpha}} c_{\bvec{\pi}} h_{\bvec{\pi}}(x_1, x_2, \ldots, x_n) 
	= \norm{A}_{\X,d}^d$,
	which is positive definite.
\end{proof}

\begin{corollary}\label{Corollary:Hunter}
	For even $d \geq 2$, the complete symmetric homogeneous polynomial 
	$h_d(x_1,x_2,\ldots,x_n)$ is positive definite.
\end{corollary}

Note that if $\alpha = 2$, then we obtain the positive definite symmetric polynomial
$H_{d,2}(x_1,x_2, \ldots, x_n)= \sum_{i=0}^d h_i (x_1,x_2, \ldots, x_n) h_{d-i}(x_1,x_2, \ldots, x_n)$. More generally, we have the relation
$
\sum_{\ell=0}^{\infty} H_{\ell, \alpha}t^{\ell}=(\sum_{\ell=0}^{\infty} h_{\ell}t^{\ell})(\sum_{\ell=0}^{\infty} H_{\ell, \alpha-1}t^{\ell}).
$
This implies that the sequence $\{H_{d,\alpha}\}_{\alpha\geq 1}$ satisfies the recursion \vspace{-4pt}
\begin{equation}
	H_{d,\alpha}=\sum_{i=0}^d h_i H_{d-i, \alpha-1}.\label{eq:HRecursion}
\end{equation}
For example, let $j=4$ and $\alpha=3$. There are four partitions $\bvec{\pi}$ of $j$ with $|\bvec{\pi}|\leq 3$. These are $(1,1,2)$,  $(1,3)$, $(2,2)$ and $(4)$. 
Therefore,
\begin{align*}
	H_{4,3}(x_1, x_2, x_3,x_4)&=c(1,1,2)h_1^2h_2+c(1,3)h_1h_3+c(2,2)h_2^2+c(4)h_4\\
	&=\frac{3!}{0! 2! 1!}h_1^2h_2+\frac{3!}{1! 1! 1!}h_1h_3+\frac{3!}{1! 2!}h_2^2+\frac{3!}{2!1!}h_4\\[2pt]
	&=3h_1^2h_2+6h_1h_3+3h_2^2+3h_4
\end{align*} 
is a positive definite symmetric polynomial. In light of \eqref{eq:HRecursion}, we can also write
$H_{4,3}(x_1, x_2, x_3,x_4)=\sum_{i=0}^4 h_i H_{4-i, 2}=H_{4,2}+h_1H_{3,2}+h_2H_{2,2}+h_3H_{1,2}+h_4$.

\subsection{Exact computation via characteristic polynomial}\label{Section:Properties}

In this section, it is shown how the CHS norm of a Hermitian matrix can be computed rapidly and exactly from its characteristic polynomial and recursion. This then leads to a determinantal formula for the CHS norm of a general matrix. The following theorem, which is \cite[Thm.~20]{Aguilar}, involves only formal series manipulations.

\begin{theorem}\label{Theorem:Characteristic}
Let $p_A(x)$ denote the characteristic polynomial of $A \in \H_n$.
For even $d\geq 2$, $\norm{A}_d^d$ is the coefficient of $x^d$ in the Taylor expansion about the origin of 
\begin{equation*}
\frac{1}{\det(I - x A)} = \frac{1}{ x^n p_A(1/x)}.
\end{equation*}
\end{theorem}

\begin{proof}
Let $p_A(x) = (x-\lambda_1)(x-\lambda_2)\cdots ( x - \lambda_n)$.
For $|x|$ small, we have
\begin{equation*}
\prod_{k=1}^n \frac{1}{1-\lambda_k x} 
= \frac{1}{x^n}\prod_{k=1}^n \frac{1}{x^{-1}-\lambda_k } 
= \frac{1}{x^n p_A(1/x)} ;
\end{equation*}
the apparent singularity at the origin is removable.   Now observe that
\begin{equation*}
\prod_{k=1}^n \frac{1}{1-\lambda_k x}  = \frac{1}{\det \diag(1 - \lambda_1 x,1 - \lambda_2x,  \ldots, 1 - \lambda_n x)} = \frac{1}{\det(I - xA)}
\end{equation*}
by the spectral theorem. The results then follows from \eqref{eq:CHSGenerating} and \eqref{eq:CHS2}.
\end{proof}

\begin{example}\label{Example:Fibo}
Let $A =\left[\begin{smallmatrix}
	1&1\\1&0
\end{smallmatrix}\right]$.  Then $p_A(z) = x^2 - x - 1$ and
\begin{equation*}
\frac{1}{x^2 p_A(1/x)} = \frac{1}{1 -x-x^2} = \sum_{n=0}^{\infty} f_{n} x^n,
\end{equation*}
in which $f_n$ is the $n$th Fibonacci number, defined by 
$f_{n+2} = f_{n+1} + f_n$ and $f_0 = f_1 = 1$.  Consequently,
$\norm{A}_d^d = f_{d}$ for even $d \geq 2$.  
\end{example}

If $A \in \H_n$ is fixed,
the sequence $h_d(\lambda_1,\lambda_2,\ldots,\lambda_n)$ satisfies the Newton--Gerard identities, defined in \eqref{eq:CHSPowerSumRec}, which states that
\begin{equation*}
	h_d(x_1,x_2,\ldots,x_n) =  \frac{1}{d} \sum_{i=1}^d h_{d-i}(x_1,x_2,\ldots,x_n) (x_1^i+x_2^i+\cdots+x_n^i).
\end{equation*} 
For even $d \geq 2$, it follows that
\begin{equation}\label{eq:Recursive}
	h_d(\bvec{\lambda}(A)) = \frac{1}{d} \sum_{i=1}^d h_{d-i}(\bvec{\lambda}(A))  \tr (A^i),
\end{equation}
which can be used to easily compute $\norm{A}_d^d = h_d( \bvec{\lambda}(A))$ recursively. If in addition $d$ is small, then there is an even simpler method. 
Since $p_A(x)$ is monic,  $\widetilde{p_A}(x) = x^n p_A(1/x)$ has constant term $1$. Thus,
\begin{equation*}
	\sum_{d=0}^{\infty} h_d(\lambda_1,\lambda_2,\ldots,\lambda_n) x^d 
	= \frac{1}{\widetilde{p_A}(x) } = \frac{1}{1 - (1-\widetilde{p_A}(x))} = \sum_{d=0}^{\infty} (1-\widetilde{p_A}(x))^d.
\end{equation*}
The desired $h_d(\lambda_1,\lambda_2,\ldots,\lambda_n)$ can then be computed by simply expanding the geometric series above
to the appropriate degree. Theorem \ref{Theorem:Characteristic} leads to the following general determinantal interpretation, which is \cite[Thm.~26]{Aguilar}.

\begin{theorem}\label{Theorem:Determinant}
	Let $Z\in\M_n$. For even $d\geq 2$, $\binom{d}{d/2} \cnorm{Z}_d^d$ is the coefficient of $z^{d/2}\overline{z}^{d/2}$ in the formal series expansion of $\det(I-zZ-\overline{z}Z^*)^{-1}$ about the origin.
\end{theorem}

\begin{proof}
If $H \in \H_n$, the coefficient of $x^d$ in $\det(I-xH)^{-1}$ is $\norm{H}_d^d$ by Theorem \ref{Theorem:Characteristic}. 
By plugging in $H=zZ+\overline{z}Z^*$ and treating the resulting expression as a series in $z$ and $\overline{z}$, Remark \ref{Remark:Czz} 
implies that the coefficient of $z^{d/2}\overline{z}^{d/2}$ equals $\binom{d}{d/2} \cnorm{Z}_d^d$.
\end{proof}

\begin{remark}
	Polynomials of the form $p=\det(I-zA-\overline{z}A^*)\in\C[z,\overline{z}]$
	are the real-zero polynomials in $\C[z,\overline{z}]$ \cite{HV}.  They are characterized by the conditions
	$p(0)=1$ and that $x\mapsto p(\alpha x)$ has only real zeros for every $\alpha\in\C$. 
	Such polynomials are studied in the context of hyperbolic \cite{Ren} and stable polynomials \cite{Wag}.
\end{remark}

\begin{example}
Let $Z= \left[\begin{smallmatrix}
	0&1\\0&0
\end{smallmatrix}\right]$.  Then
\begin{equation*}
\det(I-zZ-\overline{z}Z^*)^{-1} = \frac{1}{1-\overline{z} z} = \smash{ \sum_{n=0}^{\infty} \overline{z}^n z^n},
\end{equation*}
and hence $\norm{Z}_d^d = \smash{\binom{d}{d/2}^{\!-1}}$ for even $d \geq 2$.
\end{example}

\begin{example}
For $Z=\left[\begin{smallmatrix}
	0 & 1 & 0 \\
	0 & 0 & 1 \\
	1 & 0 & 0
\end{smallmatrix}\right]$, we have \vspace{-3pt}
\begin{equation*}
\det(I-zZ-\overline{z}Z^*)^{-1} = 
\frac{1}{1 -z^3 -3 z \overline{z}- \overline{z}^3}.
\end{equation*}
Computer algebra reveals that $\cnorm{Z}_2^2 = \cnorm{Z}_4^4 = \frac{3}{2}$, 
$\cnorm{Z}_6^6 = \frac{29}{20}$, and
$\cnorm{Z}_8^8 = \frac{99}{70}$.
\end{example}

\begin{example}
The matrices $Z_1=\left[\begin{smallmatrix}
	0 & 0 & 0 \\
	0 & 1 & i \\
	0 & i & -1
\end{smallmatrix}\right]$ and $Z_2=\left[\begin{smallmatrix}
0 & 0 & 1 \\
0 & 0 & i \\
1 & i & 0
\end{smallmatrix}\right]$ satisfy
\begin{equation*}
\frac{1}{\det(I-zZ_1-\overline{z}Z_1^*)} = \frac{1}{1 - 4 z \overline{z}} =  \frac{1}{\det(I-zZ_2-\overline{z}Z_2^*)},
\end{equation*}
so $\cnorm{Z_1}_d = \cnorm{Z_2}_d$ for even $d\geq 2$.
These matrices are not similar (let alone unitarily similar)
since $Z_1$ is nilpotent of order two and $Z_2$ is nilpotent of order three.
\end{example}

\subsection{Inner products}\label{Subsection:Inner}
Theorem \ref{Theorem:Main2}.(g) says that $\cnorm{\cdot}_d$, defined in \eqref{eq:Extended}, is a norm on $\M_n$ for even $d\geq 2$.  
When are these norms are induced by an inner product? The following is \cite[Thm.~31]{Aguilar}

\begin{theorem}\label{Theorem:Inner}
The norm $\cnorm{\cdot}_d$ on $\M_n$ (and its restriction to $\H_n$) is induced by an inner product if and only if
$d=2$ or $n=1$
\end{theorem}

\begin{proof}
If $n=1$ and $d \geq 2$ is even, then $\cnorm{Z}_d$ is a fixed positive multiple of $|z|$ for all $Z = [z] \in \M_1$.
Consequently, $\cnorm{\cdot}_d$ is induced by a positive multiple of the inner product
$\inner{Z_1,Z_2} = \overline{z_2}z_1$, in which $Z_1=[z_1]$ and $Z_2 = [z_2]$.

If $d=2$ and $n \geq 1$, then we have $\cnorm{Z}_2^2 = \tfrac{1}{2}\tr(A^*A) + \tfrac{1}{2} \tr(A^*)\tr (A) $ (see Example~\ref{Ex:d=2}),
which one can see is induced by the inner product
$\inner{A,B} = \tfrac{1}{2}\tr(B^*A) + \tfrac{1}{2}\tr (B^*)\tr (A)$.  

We must show that in all other cases,
the norm $\cnorm{A}_{d}=( h_{d}( \bvec{\lambda}(A)) )^{1/d}$ on $\H_n$ does not arise from an inner product. 
For $n \geq 2$, let $A = \diag(1,0,0,\ldots)$ and $B= \diag(0,1,0,\ldots,0) \in \H_n$.  Then
$\cnorm{A}_d^2 = \cnorm{B}_d^2 = 1$.
Using \eqref{eq:NormHermitian} and the fact that the sum of gamma random variables is again a gamma random variable, observe that
$\cnorm{A+B}_d^2 = (d+1)^{2/d}$. A similar argument also shows that
$\cnorm{A-B}_d^2 = 1$.
A classical result of Jordan and von Neumann \cite{Jordan} says that a vector-space norm $\norm{\cdot}$ arises from an inner product
if and only if it satisfies the parallelogram identity
$\norm{ \vec{x} + \vec{y}}^2 + \norm{ \vec{x} - \vec{y} }^2
= 2 \big( \norm{ \vec{x} }^2 + \norm{ \vec{y} }^2 )$
for all $\vec{x}, \vec{y}$.
If $\norm{\cdot}_d$ satisfies the parallelogram identity, then
$(d+1)^{2/d} + 1 = 2(1+1)$;
that is, $(d+1)^2 = 3^d$.  The solutions are $d=0$ (which does not yield an inner product) and $d=2$ (which, as we showed above, does).
Therefore, for $n \geq 2$ and $d\geq 2$,
the norm $\norm{\cdot}_d$ on $\H_n$ does not arise from an inner product, and so neither does $\cnorm{\cdot}_d$.
\end{proof}

\vspace{-11pt}
\subsection{Equivalence constants}\label{Section:Bounds}
\vspace{-4pt}

Any two norms on a finite-dimensional vector space are equivalent, so each norm $\norm{\cdot}_d$ on $\H_n$ is equivalent to the operator norm $\onorm{\cdot}$, which is equal to the biggest singular value. We compute admissible equivalence constants below \cite[Thm.~38]{Aguilar}.

\begin{theorem}\label{Theorem:Constants}
For $A \in \H_n$ and even $d \geq 2$, 
\vspace{-4pt}
\begin{equation*}
\bigg(\frac{1}{2^{\frac{d}{2}} (\frac{d}{2})!}\bigg)^{\!\!1/d}\!\onorm{A} 
\leq \norm{A}_d 
\leq \binom{n+d-1}{d}^{\!\!1/d}\!\onorm{A} .\vspace{-4pt}
\end{equation*}
The upper inequality is sharp if and only if $A$ is a multiple of the identity.
\end{theorem}

\begin{proof}
For $A \in \H_n$ and even $d\geq 2$, the triangle inequality yields
\vspace{-4pt}
\begin{align*}
\norm{A}_d^{d}
&=  h_{d}( \lambda_1(A),\lambda_2(A),\ldots,\lambda_n(A) )  \\
&\leq  h_{d}( |\lambda_1(A)| , |\lambda_2(A)|,\ldots ,|\lambda_n(A)| )  \\
&\leq  h_{d}( \onorm{A}, \onorm{A},\ldots, \onorm{A} ) \\
&= \onorm{A}^{d} \,h_{d}(1,1,\ldots,1) \\
&=\onorm{A}^{d} \binom{n+d-1}{d} .\\[-20pt]
\end{align*}
Equality occurs if and only if $\lambda_i(A) = |\lambda_i(A)| = \onorm{A}$ for $1 \leq i \leq n$;
that is, if and only if $A$ is a multiple of the identity.
Now, Hunter established that 
\vspace{-4pt}
\begin{equation*}
h_{2p}(\vec{x}) \geq \frac{1}{2^p p!} \norm{\vec{x}}^{2p},\vspace{-2pt}
\end{equation*}
in which $\norm{ \vec{x} }$ denotes the Euclidean norm of $\vec{x} \in \R^n$ (see Theorem \ref{Theorem:HunterBound}).
Let $d=2p$ and conclude 
\vspace{-7pt}
\begin{equation*}
\norm{A}_d \geq  \bigg( \frac{1}{2^{\frac{d}{2}} (\frac{d}{2})!} \bigg)^{\!\!1/d} \!\norm{A}_{\operatorname{F}}
\geq  \bigg( \frac{1}{2^{\frac{d}{2}} (\frac{d}{2})!} \bigg)^{\!\!1/d} \!\onorm{A},\vspace{-3pt}
\end{equation*}
in which $\norm{A}_{\operatorname{F}}$ denotes the Frobenius norm of $A \in \H_n$.
\end{proof}

\begin{remark}
For $Z \in \M_n$, the upper bound in Theorem \ref{Theorem:Constants}.(b) may be applied to
$e^{it} Z + e^{-it} Z^*$ to deduce that
\vspace{-4pt}
\begin{equation*}
\cnorm{Z}_d \leq
\Bigg(\frac{ \binom{n+d-1}{d}}{2\pi {\binom{d}{d/2}}}\int_0^{2\pi} 
\onorm{ e^{it} Z+e^{-it} Z^*}^d\,\mathrm{d}t\Bigg)^{\!\!1/d}
\!\!\leq 2  \Bigg(\frac{\binom{n+d-1}{d}}{ {\binom{d}{d/2}}}\Bigg)^{\!\!1/d} \! \onorm{Z}.\vspace{-2pt}
\end{equation*}
\end{remark}

\begin{remark}
Hunter's lower bound was improved by Baston \cite{Baston}, who proved that
\vspace{-4pt}
\begin{equation*}
h_{2p}(\vec{x}) \geq \frac{1}{2^p p!} \bigg( \sum_{i=1}^n x_i^2 \bigg)^{\!\!p} + \gamma_p \bigg( \sum_{i=1}^n x_i \bigg)^{\!\!2p}\vspace{-4pt}
\end{equation*}
for $\vec{x} =(x_1,x_2,\ldots,x_n) \in \R^n$,
in which $\gamma_p = \frac{1}{n^p}( \binom{n+2p-1}{2p} \frac{1}{n^p} - \frac{1}{2^p p!} ) > 0$.
Equality holds if and only if $p=1$ or $p\geq 2$ and the $x_i$ are all equal.
\end{remark}


\vspace{-6pt}
\section{Open Questions}\label{Section:Questions}
\vspace{-2pt}

In Section \ref{Section:Submultiplicativity} it was shown that if $\norm{\cdot}$ is a norm on $\M_n$, then there is a scalar multiple of it (which may depend upon $n$)
that is submultiplicative. In the same section, we showed that if the coefficients $X_i$ of the random vector $\X$ satisfy $X_i\in L^{\max\{d,\eta\}}(\Omega,\mathcal{F},\P)$ for some $\eta>2$, then $\cnorm{\cdot}_{\X,d}$ is submultiplicative when multiplied by a constant independent of $n$. For example, Remark \ref{Remark:d=2} 
ensures that the random vector norm of order $2$ is submultiplicative if and only if $2\sigma^2 \geq 1+\sqrt{1+4\mu^2}$.

However, the example in Subsection \ref{Subsection:a-stable} shows that this hypothesis is not necessary. This naturally leads to the following problem.

\smallskip\noindent\textbf{Problem 1.} Characterize those $\X$ that give rise to submultiplicative norms. In particular, what are the best constants $\gamma_d$, independent of $n$, such that $\gamma_d \norm{ \cdot }_{\X,d}$ is submultiplicative?

\smallskip
The result of Theorem \ref{Theorem:Main2}.(c) suggests that it might be better in some situations to consider certain scalar multiples of the random vector norms instead of the norm itself. In particular, as was detailed above, it is natural to wonder when these ``good'' scalars  are independent of $n$. For instance, the coefficient in \eqref{Eq:alpha} is independent of $n$, but when $d\to \infty$, this coefficient tends to 0. Consequently, the random vector norms have no natural meaning behind the expression $\cnorm{\cdot}_{\X,\infty}$, unlike other norms like the operator and Schatten norms. Hence, we naturally propose the following question.

\smallskip\noindent\textbf{Problem 2.} 
Characterize those $\X$ that give rise to norms $\|\cdot\|_{\X,d}$ which, under multiplication by a scalar $\gamma_d$ independent of $n$, remain a norm when $d\to\infty$.

\smallskip
Examining the definition of random vector norms, one might wonder how much these can be generalized. In particular:

\smallskip\noindent\textbf{Problem 3.} 
Generalize random vector norms to compact operators on infinite-dimensional Hilbert spaces.

\smallskip
For the standard exponential distribution, Theorem \ref{Theorem:Inner} provides an answer to the next question.
An answer to the question in the general setting eludes us.

\smallskip\noindent\textbf{Problem 4.} 
Characterize the norms $\cnorm{\cdot}_{\X,d}$ that arise from an inner product.

\smallskip
Similarly, for the standard exponential distribution, we provided equivalence constants between the random vector norms and the operator norm in Theorem \ref{Theorem:Constants}. However, we have not been able to show that the lower bound is sharp.

\smallskip\noindent\textbf{Problem 5.} 
Find a sharp lower bound in Theorem \ref{Theorem:Constants}.

\smallskip
Other unsolved questions come to mind.

\smallskip\noindent\textbf{Problem 6.} 
Identify the extreme points with respect to random vector norms.

\smallskip\noindent\textbf{Problem 7.} 
Characterize norms on $\M_n$ or $\H_n$ that arise from random vectors.

\smallskip
Lastly, we emphasize that if one uses \eqref{eq:Extended} to evaluate $\cnorm{A}_{\X,d}^d$, there are many repeated terms. For example, $(\tr A^*A)(\tr A)(\tr A^*) = (\tr AA^*)(\tr A^*)(\tr A)$ because of the cyclic invariance of the trace and the commutativity of multiplication. If one chooses a single representative for each such class of expressions and simplifies,
one gets expressions such as \eqref{eq:d=2} and \eqref{eq:Z44}. This leads to the problem:

\smallskip\noindent\textbf{Problem 8.}
Give a combinatorial interpretation of the resulting coefficients in \eqref{eq:Extended}.

\medskip
For motivation, the reader is invited to consider
{\small
\begin{align*}
	\cnorm{Z}_6^6 = \smash{\frac{1}{720}}&\smash{\Big(} (\tr Z)^3 \tr(Z^*)^3 + 3 \tr(Z) \tr(Z^*)^3 \tr(Z^2) + 9 (\tr Z)^2 \tr(Z^*)^2 \tr(Z^*Z) \\ 
	& + 9 \tr(Z^*)^2 \tr(Z^2) \tr(Z^*Z) +  18 \tr(Z) \tr(Z^*) \tr(Z^*Z)^2 + 2 \tr(Z^*)^3 \tr(Z^3)  \\
	&  + 3 (\tr Z)^3 \tr(Z^*) \tr(Z^{*2}) + 9 \tr(Z) \tr(Z^*) \tr(Z^2) \tr(Z^{*2}) + 6 \tr(Z^*Z)^3 \\ 
	& + 9 (\tr Z)^2 \tr(Z^*Z) \tr(Z^{*2})+ 9 \tr(Z^2) \tr(Z^*Z) \tr(Z^{*2}) + 2 (\tr Z)^3 \tr(Z^{*3}) \\ 
	&+ 6 \tr(Z^*) \tr(Z^{*2}) \tr(Z^3) + 18 \tr(Z) \tr(Z^*)^2 \tr(Z^*Z^2) + 4 \tr(Z^3) \tr(Z^{*3}) \\ 
	&+ 36 \tr(Z^*) \tr(Z^*Z) \tr(Z^*Z^2) + 18 \tr(Z) \tr(Z^{*2}) \tr(Z^*Z^2)\\ 
	& + 18 (\tr Z)^2 \tr(Z^*) \tr(Z^{*2}Z) + 18 \tr(Z^*) \tr(Z^2) \tr(Z^{*2}Z) + 18 \tr(Z^2) \tr(Z^{*3}Z) \\ 
	&+ 36 \tr(Z) \tr(Z^*Z) \tr(Z^{*2}Z) + 36 \tr(Z^*Z^2) \tr(Z^{*2}Z) + 6 \tr(Z) \tr(Z^2) \tr(Z^{*3})\\
	& + 18 \tr(Z^*)^2 \tr(Z^*Z^3) + 18 \tr(Z^{*2}) \tr(Z^*Z^3) + 18 \tr(Z) \tr(Z^*) \tr(Z^*ZZ^*Z) \\
	&+ 18 \tr(Z^*Z) \tr(Z^*ZZ^*Z) + 36 \tr(Z) \tr(Z^*) \tr(Z^{*2}Z^2)+ 18 (\tr Z)^2 \tr(Z^{*3}Z)  \\ 
	&+ 36 \tr(Z^*Z) \tr(Z^{*2}Z^2) + 36 \tr(Z^*) \tr(Z^*ZZ^*Z^2) + 36 \tr(Z^*) \tr(Z^{*2}Z^3) \\ 
	&+ 36 \tr(Z) \tr(Z^{*2}ZZ^*Z) + 36 \tr(Z) \tr(Z^{*3}Z^2) + 12 \tr(Z^*ZZ^*ZZ^*Z) \\ 
	& + 36 \tr(Z^{*2}Z^2Z^*Z) + 36 \tr(Z^{*2}ZZ^*Z^2) + 36 \tr(Z^{*3}Z^3) \smash{\Big)}. 
\end{align*}
}

\begin{acknowledgement}
	SRG partially supported by NSF grant DMS-2054002.  He thanks the International Linear Algebra Society (ILAS) for their generous support and sponsoring his Hans Schneider ILAS Lecture at IWOTA 2023 and his ILAS Lecture at JMM 2024.
\end{acknowledgement}

\bibliographystyle{amsplain}
\bibliography{RandomNorms-IWOTA}

\end{document}